\documentclass[reqno]{amsart}
\usepackage[scale=0.75, centering, headheight=14pt]{geometry}
\usepackage[latin1]{inputenc}
\usepackage[T1]{fontenc}
\usepackage{lmodern}
\usepackage[english]{babel}
\usepackage{stmaryrd}
\usepackage{tikz}
\usepackage{bbm}
\usepackage{amsmath,amssymb,amsfonts,amsthm}
\usepackage{mathtools,accents}
\usepackage{mathrsfs}
\usepackage{xfrac}
\usepackage{array} 
\usepackage{aliascnt}
\usepackage{booktabs} 
\usepackage{array} 
\usepackage{url}

\usepackage{verbatim} 
\usepackage{subfig} 
\usepackage{mathrsfs}
\usepackage{mathrsfs, dsfont}
\usepackage{amssymb}
\usepackage{amsthm}
\usepackage{amsmath,amsfonts,amssymb,esint}
\usepackage{graphics,color}
\usepackage{enumerate}
\usepackage{mathtools,centernot}

\usepackage{microtype}
\usepackage{paralist} 
\usepackage{cases}
\usepackage[initials,alphabetic]{amsrefs}
\allowdisplaybreaks

\usepackage{braket}

\usepackage{bm}

\usepackage[citecolor=blue,colorlinks]{hyperref}
\addto\extrasenglish{}

\usepackage[
    type={CC},
    modifier={by-nc-nd},
    version={4.0},
]{doclicense}

\usepackage{enumerate}
\usepackage{xcolor}

\usepackage{aliascnt}

\makeatletter
\def\newaliasedtheorem#1[#2]#3{
  \newaliascnt{#1@alt}{#2}
  \newtheorem{#1}[#1@alt]{#3}
  \expandafter\newcommand\csname #1@altname\endcsname{#3}
}
\makeatother

\makeatletter
\newcommand{\myitem}[1]{%
\item[(#1)]\protected@edef\@currentlabel{#1}%
}
\makeatother

\theoremstyle{plain}

\newaliasedtheorem{lemma}[theorem]{Lemma}
\newaliasedtheorem{proposition}[theorem]{Proposition}
\newaliasedtheorem{claim}[theorem]{Claim}
\newaliasedtheorem{corollary}[theorem]{Corollary}

\newtheorem{maintheorem}{Theorem}

\DeclareMathOperator{\diver}{div}

\theoremstyle{definition}
\newaliasedtheorem{definition}[theorem]{Definition}
\newaliasedtheorem{example}[theorem]{Example}

\newtheorem{OQ}[]{Open Question}

\theoremstyle{remark}
\newaliasedtheorem{remark}[theorem]{Remark}
\newaliasedtheorem{ex}[theorem]{Example}

\numberwithin{equation}{section}

\newcommand{\e}{{\rm e}}

\def\eps{\varepsilon}

\def\R{\mathbb R}
\def\N{{\mathbb N}}
\def\Z{{\mathbb Z}}
\def\T{{\mathbb T}}

\DeclareMathOperator{\stat}{stat}
\DeclareMathOperator{\err}{err}

\DeclareMathOperator{\supp}{supp}
\DeclareMathOperator{\chess}{chess}

\DeclareMathOperator{\loc}{loc}

\DeclareMathOperator{\initial}{in}

\newcommand{\weak}{\overset{*}{\rightharpoonup}}

\DeclareMathOperator{\dist}{dist}

\newcommand{\p}{\partial}

\newcommand{\lesssimlarge}{\lesssim_{q \gg 1}}

\title[Nontrivial absolutely continuous part of anomalous dissipation measures in time]
 {Nontrivial absolutely continuous part   of anomalous dissipation measures in time}
 
\author[C. J. P. Johansson and M. Sorella]{Carl Johan Peter Johansson and Massimo Sorella}

\address{Massimo Sorella
\hfill\break EPFL SB, Station 8, CH-1015 Lausanne, Switzerland}
\email{massimo.sorella@epfl.ch}

\address{Carl Johan Peter Johansson
\hfill\break EPFL SB, Station 8, CH-1015 Lausanne, Switzerland}
\email{carl.johansson@epfl.ch}

\thanks{Author's Accepted Manuscript of an article published in Journal of Differential Equations in Volume 453, Part 4 \url{https://doi.org/10.1016/j.jde.2025.113912}, published by Elsevier Inc. \copyright2025 Elsevier Inc.}

\begin{document}
\doclicenseThis

\begin{abstract}
We answer  \cite[Question 2.2 and Question 2.3]{BDL22} in dimension $4$ by building new examples of solutions to the forced $4d$ incompressible Navier-Stokes equations, which exhibit anomalous dissipation, related to the zeroth law of turbulence \cite{kolmo}. 
We also prove that the unique smooth solution $v_\nu$ of the $4d$ Navier--Stokes equations with time-independent body forces is $L^\infty$-weakly* converging to a solution of the forced Euler equations $v_0$ as the viscosity parameter $\nu \to 0$. Furthermore, the sequence
 $\nu |\nabla v_\nu|^2$ is weakly* converging (up to subsequences), in the sense of measure, to $\mu \in \mathcal{M} ((0,1) \times \T^4)$ and $\mu_T = \pi_{\#} \mu $ has a non-trivial absolutely continuous part where $\pi$ is the projection onto the time variable. Moreover,  we also show that $\mu$ is close, up to an error measured in $H^{-1}_{t,x}$, to the Duchon--Robert distribution  $ \mathcal{D}[v_0] $ of the solution to the $4d$ forced Euler equations.  Finally, the kinetic energy profile of $v_0$ is smooth in time.
Our result relies on a new anomalous dissipation result for  the advection--diffusion equation with a divergence free $3d$ autonomous velocity field and the study of the $3+\frac{1}{2} $ dimensional  incompressible Navier--Stokes equations.  This study motivates some  open problems.
\end{abstract}

\maketitle

\section{Introduction}

We 
 study 
the Navier--Stokes equations with a body force  on the $4$-dimensional torus $\T^4 \simeq \R^4 / \Z^4$, namely 
\begin{align} 
\tag{NS}
\label{e:NSE}
\begin{cases}
\partial_t v_{\nu} + v_\nu \cdot \nabla v_\nu + \nabla p_\nu = \nu \Delta v_\nu + F_\nu;
\\
\diver v_\nu =0; \notag
\end{cases}
\end{align}
where $v_\nu : [0, + \infty) \times \T^4 \to \R^4$ is the velocity field, $p_\nu : [0, \infty) \times \T^4 \to \R$ is the pressure, $\nu \geq 0$ is the viscosity parameter and $F_\nu : [0, + \infty) \times \T^4 \to \R^4$ is a force that may depend on $\nu$. In the case $\nu = 0$ equations~\eqref{e:NSE} reduce to the Euler equations with body force $F_0$.  We consider the Navier--Stokes equations \eqref{e:NSE} with a prescribed initial datum $v_{\initial, \nu } \in L^\infty$. The goal of this paper is to study the possible behaviour of the dissipation sequence $\nu |\nabla v_\nu|^2$ of turbulent suitable Leray-Hopf solutions (see \cite{ckn}) to the Navier--Stokes equations with body forces.
Based on data,  in the infinite Reynolds number limit, it is expected to have a continuous in time dissipation behaviour of the solutions. We clarify this physical idea giving precise mathematical definitions and statements and raising some open problems in this direction.
We first introduce the concept of {\em physical  solutions} to the Euler equations with a body force.  This definition is motivated by the still unsolved physical Onsager conjecture (see for instance in \cite{D17Thesis} for a mathematical definition) and it is consistent to the definition given in \cite[Definition 3]{CFLS16}.
\begin{definition}[Physical solutions] \label{d:physical}
    Let $d\geq 2$ and $v \in L^3 ((0,1) \times \T^d)$ be a distributional solution to the incompressible Euler equations with force $F_0 \in L^{3/2}((0,1) \times \T^d)$ in dimension $d$ with a divergence free initial datum $v_{\initial} \in L^2$. We say that $v $ is a {\em physical  solution} if:
    \begin{enumerate}
        \item there exists a sequence of initial data $\{ v_{\initial, \nu } \}_{\nu > 0}$ such that  $v_{\initial, \nu} \to v_{\initial}$ in $L^2$,
        \item \label{d:property-2} there exists a sequence of suitable Leray--Hopf solutions $\{ v_\nu \}_{\nu >0} \subset L^3_{t,x}$ of \eqref{e:NSE} (see \cite{ckn}) with  initial data $v_{\initial, \nu}$ and forces $\{ F_\nu \}_{\nu >0}$
        such that  there exists a subsequence $\{ \nu_q \}_{q \in \N}$ for which we have $v_{\nu_q} \rightharpoonup v$ as $\nu_q \to 0$ in the sense of distribution,
        \item \label{d:property-3} the forces are such that $F_\nu \equiv F_0$ for any $\nu >0$.
    \end{enumerate}
\end{definition}

Property \eqref{d:property-2} in the previous definition is motivated by the fact that the dissipation $\nu |\nabla v_\nu|^2$ is assumed to be  bounded in $L^1_{t,x}$ independently of $\nu$ in the theory of turbulence (see for instance \cite{kolmo,frisch}) and the regularity of $v$ is required to be $L^3$ because of the results by Duchon--Robert \cite[Proposition 1 and Proposition 2]{DR00}. More precisely, in \cite{DR00}, the authors proved that for any $d \geq 2$ and $(v_\nu, p_\nu)  \in L^3 \times L^{3/2}$ {weak} solution  to the incompressible Navier--Stokes equations \eqref{e:NSE} (or weak solution $(v_0, p_0) \in L^3 \times L^{3/2}$ to the incompressible Euler equations, namely \eqref{e:NSE} with $\nu=0$) with force $F_\nu \in L^{3/2} ((0,1) \times \T^d)$  there exists a distribution $\mathcal{D} [v_\nu] $, such that 
\begin{align}\label{duchon-robert}
    \partial_t \frac{| v_\nu |^2}{2} + \diver \left (v_\nu \left (\frac{|v_\nu|^2}{2} + p_\nu \right ) \right ) + \mathcal{D}[v_\nu] +  \nu |\nabla v_\nu|^2 = F_\nu \cdot v_\nu +  \frac{\nu}{2} \Delta |v_\nu|^2  \qquad \text{in } (0,1) \times \T^d
\end{align} 
holds for any $\nu \in [0, \infty) $
 in the sense of distribution.  { For suitable Leray--Hopf solutions the distribution $\mathcal{D} [v_\nu] $ is non-negative  and hence is a measure.}
 In the setting { of suitable Leray--Hopf solutions} we are interested in the anomalous dissipation, namely whether it holds true that
 \begin{align} \label{e:zeroth-law}
     \limsup_{\nu \downarrow 0} \nu \int_0^1 \int_{\T^4} | \nabla v_{\nu} |^2 >0 \,,
 \end{align}
 which is related to the zeroth-law of turbulence, \cite{kolmo} {precisely \eqref{e:zeroth-law} with $\limsup$ replaced by $\liminf$.} 
 
In this context, we remark that  there are several examples (also with force $F_0 \equiv 0$) in which the Duchon-Robert distribution $\mathcal{D}[v_0]$ is not identically zero for 3d Euler equations (see for instance \cite{DLL09,DLL13,Is18,BDLIS,BDLLOnsager,BV18,DLK20,S93,Sh97,Sh00,BMNV21,NV22}), most of them relying on the so called convex integration technique. However none of these examples, to the authors knowledge, are  proved to be {\em physical solutions}.
 Since 
 the existence of {\em physical  solutions} which exhibit anomalous dissipation in the general setting for $d \geq 3$ is a hard problem, the authors in \cite{BDL22} suggested to allow the body forces of the forced Navier--Stokes equations 
 $\{ F_{\nu} \}_{\nu >0}$ to depend on $\nu >0$. 
 This allows to exploit examples of anomalous dissipation for the advection-diffusion equation by studying the forced 2+$\frac{1}{2}$ Navier--Stokes equations. This framework was already used in \cite{JY21} to prove enhanced dissipation.
  Nevertheless, we want to rule out trivial examples of \eqref{e:zeroth-law}  given by solutions to the heat equation or the linear Stokes equations with forces depending on $\nu \geq 0$
 (see for instance \cite[Remark 1.2]{BCCDLS22}) in order to get the non linearity in \eqref{e:NSE} involved in the zeroth law of turbulence. 
 Hence a uniform in $\nu$ regularity must be imposed on the sequence of body forces $\{ F_\nu \}_{\nu > 0}$, following the proposal by \cite{BDL22}
 \begin{align} \label{e:regularity-forces-nu}
     \sup_{\nu > 0} \| F_{\nu } \|_{L^{1 + \sigma} ((0,1); C^{\sigma} (\T^d))} < \infty \, \quad \text{{for some $\sigma > 0$}}.
 \end{align}

 \begin{remark}
    As already noticed in \cite{BDL22}, if $v_{\nu} : [0,1] \times \T^d \to \R^d $ solves the linear Stokes equations  $\partial_t v_\nu - \nu \Delta v_{\nu} + \nabla p_\nu= F_{\nu}$ with divergence free initial datum $v_{\initial} \in L^2$, pressure $p_\nu : [0,1] \times \T^d \to \R$ solving $\Delta p_\nu = \diver F_\nu$, $\diver v_\nu =0$ and forces $\{F_\nu\}_{\nu \geq 0}$ which satisfies \eqref{e:regularity-forces-nu}, then the $\limsup$ in \eqref{e:zeroth-law} is 0. 
    Using energy estimates, we also notice that to rule out the
    anomalous dissipation 
    \eqref{e:zeroth-law} for the linear
    Stokes equations it is also sufficient that $F_\nu \to F_0$ in $L^1((0,1); L^2(\T^d))$ 
    as $\nu \to 0$.
\end{remark}

    We now coherently define {\em weak physical  solutions}, inspired by \cite{FL22}, and {\em anomalous dissipation measures}.

\begin{definition}[Weak physical solutions + anomalous dissipation measure] \label{d:weak-physical}
   The definition of {\em weak physical solutions} is the same as {\em physical solutions}  in Definition \ref{d:physical} replacing property \eqref{d:property-3} with: the forces $\{ F_\nu \}_{\nu >0}$ {satisfy} \eqref{e:regularity-forces-nu} or $F_\nu \to F_0$ in $L^1((0,1); L^2(\T^d))$.
     
     We say that $\mu \in \mathcal{M} ((0,1) \times \T^d)$ is an {\em anomalous dissipation measure} associated to $v$ if (up to not relabelled subsequences) we have 
     $$ \mu= \lim_{\nu \to 0} (\nu |\nabla v_\nu|^2 + \mathcal{D}[v_\nu]) \,,$$
     where the convergence is weak* in $\mathcal{M}((0,1) \times \T^4)$.
     We finally say that the push-forward $\mu_T = \pi_{\#} \mu \in \mathcal{M}((0,1))$ with respect to the projection in time map $\pi : (0,1) \times \T^d \to (0,1)$ is  the  corresponding {\em anomalous dissipation measure in time}.
\end{definition}

\begin{remark}
Notice that if $v$ is a {\em weak physical solution} then there exists an anomalous dissipation measure $\mu$ associated to $v$. Indeed, by energy estimates one can prove that 
$ \| v_\nu \|_{L^\infty_t L^2_x}^2 \leq \| v_{\initial , \nu}  \|_{L^2_x}^2 + 2 \|F_\nu \|_{L^1_t L^2_x} \| v_\nu \|_{L^\infty_t L^2_x}  $ and
therefore we have
$ \| v_\nu \|_{L^\infty_t L^2_x} \leq    \|F_\nu \|_{L^1_t L^2_x} +   ( \| v_{\initial , \nu}  \|_{L^2_x}^2 + \|F_\nu \|_{L^1_t L^2_x}^2 )^{1/2} $,
from which we conclude that
$$\sup_{\nu > 0} \| \mathcal{D}[v_\nu] \|_{TV} +  \| \nu |\nabla v_\nu|^2 \|_{TV} \leq  \sup_{\nu > 0} \| v_{\initial, \nu} \|_{L^2}^2 +  \left (  \|F_\nu \|_{L^1_t L^2_x} +   ( \| v_{\initial , \nu}  \|_{L^2_x}^2 + \|F_\nu \|_{L^1_t L^2_x}^2 )^{1/2} \right ) \|F_\nu \|_{L^1_t L^2_x}   < \infty \,. $$
The existence of an anomalous dissipation measure  $\mu \in \mathcal{M}((0,1) \times \T^d)$ is now straightforward  from the weak-* compactness in the space of measures up to subsequences.
\end{remark}
 
 In the very recent paper \cite{BDL22} the authors proved the existence of a {\em weak physical solution} $v \in L^\infty$ to the Euler equations with body force $F_0 \in L^\infty ((0,1); C^\alpha (\T^3))$ for any $\alpha <1$ such that up to subsequences  $\{ \nu_q \}_{q}$
 \begin{enumerate}
     \item \label{intro-1} $\| v_{\nu_q} - v \|_{L^3 (0,1) \times \T^3)} \to 0 $ as $\nu_q \to 0$,
     \item \label{intro-3}
     $\sup_{\nu_q} \| v_{\nu_q} \|_{L^\infty } < \infty \,,$
     \item \label{intro-2} $\| F_{\nu_q} - F_0 \|_{L^\infty ((0,1) ; C^\alpha (\T^3))}  \to 0$ as $\nu_q \to 0$.
 \end{enumerate}
 In \cite{BCCDLS22} the authors improved the result in the following way approaching the Onsager criticality: for any $\alpha <1/3$, there exists a {\em weak physical solution} such that \eqref{intro-1} holds, \eqref{intro-2}  holds with norm $\| \cdot \|_{L^{1+ \sigma} ((0,1); C^{\sigma} (\T^3))}$ for some $\sigma$ and \eqref{intro-3} holds with norm $\| \cdot \|_{L^3 ((0,1); C^\alpha (\T^3))}$.

 In both the results of \cite{BDL22,BCCDLS22} $  \nu |\nabla v_{\nu}|^2$ are weakly* converging in measure (up to not relabelled subsequences) to a measure which is singular in time, concentrated at $t=1$. In the following theorem we provide an example of a {\em weak physical solution} such that 
  \eqref{e:zeroth-law} holds and (up to not relabelled subsequences)
  $$\nu |\nabla v_{\nu}|^2 \rightharpoonup^* \mu $$
  weakly* in the sense of measures. In addition,  
  $\mu_T = \pi_{\#} \mu $ has a non-trivial absolutely continuous part with respect the Lebesgue measure $\mathcal{L}^1$ on $(0,1)$. 
   We remark that the absolute  continuity of the measure $ \pi_{\#} \mu $ is observed  both in experiments \cite{sree84,PKV02} and simulations \cite{KI03,sree98} of turbulent flows. More precisely at a mathematical level, assuming that the sequence $\{v_\nu \}_{\nu >0}$ has the following bound 
  $$ \sup_{\nu \in (0,1)} \| v_\nu \|_{L^3_t B^{1/3}_{3, \infty}} < \infty \,,
  $$
 which is in accordance with the Kolmogorov four-fifth law, we rigorously justify the absolute continuity of $\pi_\# \mu$ thanks to \cite[Theorem 1.6]{Isett13}. 
 Furthermore, we give positive answers to two open problems in fluid dynamics \cite[Question 2.2 and Question 2.3]{BDL22} related to the Navier-Stokes equations on the $4$-dimensional torus.

 \begin{maintheorem}  \label{t_Onsager}
Let $\beta \in (0, 1/4)$. For any $\alpha \in [0, 1)$ there exist a divergence-free initial datum $v_{in} \in L^\infty(\T^4; \R^4) $ with $\| v_{\initial} \|_{L^2 (\T^4)} =1$
and a time-independent force $F_0 \in C^{\alpha } (\T^4; \R^4)$  such that there exists a weak physical  solution $v_0 \in L^\infty$ of the $4d$ Euler equations with force $F_0$ such that 
\begin{equation}\label{eq:EnergyInEuler}
e(t) = \frac{1}{2} \int_{\T^4} | v_0 (t, x )|^2 dx
\end{equation}
is smooth in $[0,1]$, non-increasing, $e(1) < e(0)$ and  $\int_{\T^4} F_0 (x) \cdot v_0 (t,x) dx =0 $ for any $t \in (0,1)$.

More precisely there exists a sequence of {time-independent} forces $\{ F_\nu \}_{\nu \geq 0} \subset C^{\alpha } (  \T^4; \R^4)$ and initial data $\{ v_{\initial, \nu} \}_{\nu > 0} \subset C^\infty(\T^4; \R^4) $ such that $v_{ \initial , \nu} \to v_{\initial}$ in $L^2(\T^4)$, $F_{\nu} \to F_0$ in $C^{\alpha} (\T^4) $ as $\nu \to 0$ and for any $\nu > 0$ there exists a unique smooth solution $(v_{\nu}, p_\nu)$ to \eqref{e:NSE} with $v_\nu(0,\cdot) = v_{\initial , \nu}(\cdot)$ satisfying $\sup_{\nu \in [0,1]} \| v_\nu \|_{L^{\infty } ([0,1] \times \T^4)} < \infty$, 
the anomalous dissipation \eqref{e:zeroth-law} and 
 $(v_\nu , p_\nu) \weak (v_0, p_0)$ weakly$\ast$ {up to subsequences} in $L^\infty ((0,1) \times \T^3 )$.
 
 Furthermore, there exists an anomalous dissipation measure $\mu \in \mathcal{M}((0,1) \times \T^4)$  such that $\| \mu \|_{TV} \geq 1/4$ and  up to not relabelled subsequences we have
 \begin{align} \label{e:convergence-measure}
     \nu |\nabla v_\nu|^2 \rightharpoonup^* \mu
 \end{align} 
 in the sense of measures in $(0,1) \times \T^4$ and $\mu$ is close to $\mathcal{D}[v_0]$, given by formula \eqref{duchon-robert}, in $H^{-1}$
 \begin{align} \label{e:mu-duchon-robert}
     \| \mu - \mathcal{D}[v_0] \|_{H^{-1} ((0,1) \times \T^4)} \leq \beta \,.
 \end{align} 
The corresponding anomalous dissipation measure in time $\mu_T = \pi_{\#} \mu  \in \mathcal{M} (0,1)$ has a non-trivial absolutely continuous part w.r.t. the Lebesgue measure $\mathcal{L}^1$ and its singular part $\mu_{T, \text{sing}}$ is such that $\| \mu_{T, \text{sing}} \|_{TV} \leq \beta$.
\end{maintheorem}

\begin{remark}
    We highlight the following facts:
    \begin{itemize}
        \item We prove \eqref{e:mu-duchon-robert} by showing that up to not relabelled subsequences $v_\nu$ is close to $v_0$ in $L^2((0,1) \times \T^4)$, namely $\| v_{\nu} - v_0 \|_{L^2}^2 \leq \beta$ and the uniform in $\nu>0$ $L^\infty$ bound on $v_\nu$. However,  the authors do not know if it is true that $v_\nu \to v_0$ in $L^2$ up to not relabelled subsequences.
        \item If one could show the $L^2 ((0,1) \times \T^4)$ strong convergence of $v_\nu$ to $v_0$ (up to not relabelled subsequences) in the previous theorem, then one would be able to show in the statement that  (up to not relabelled subsequences) the following distributional limit holds
        \begin{align} \label{e:kolmo-duchon-robert}
        \mathcal{D}[v_0]= \lim_{\nu \to 0} (\nu |\nabla v_\nu|^2 + \mathcal{D}[v_\nu])
        \end{align}
        thanks to the uniform in $\nu>0$ $L^\infty$ bound on  $v_\nu$, which implies the $L^3$ strong convergence of $v_\nu$ to $v_0$, because all the terms pass into the distributional limit in \eqref{duchon-robert}. Unfortunately, the $L^2$ strong convergence is a non-trivial property. Observe also that for any $\nu >0$ in our setting $\mathcal{D}[v_\nu] \equiv 0$ because the solution $v_\nu$ is smooth for any $\nu >0$.
        \item The examples in \cite{BDL22,BCCDLS22}  have the strong convergence property of $v_\nu$ to $v_0$ in $L^3$, however their measures $\mu$ defined by \eqref{e:convergence-measure} are concentrated at time $t=1$, therefore the distributional limit in $(0,1) \times \T^3$ \eqref{e:kolmo-duchon-robert} holds, but in this case since we test with compactly supported functions in the time interval $(0,1)$, the limiting Duchon-Robert distribution is such that $\mathcal{D}[v_0] \equiv 0$ in  $(0,1) \times \T^3$. Notice that we can extend those examples ``in a natural way'' to $(0,2) \times \T^3$, but then the proof in \cite{BDL22,BCCDLS22} of the strong convergence of $v_\nu$ to $v_0$ as $\nu \to 0$ in $L^3$ fails (which does not allow us to conclude the distributional limit \eqref{e:convergence-measure} in a straightforward way). 
      {     \item After this manuscript, the authors in \cite{JS24} show that any anomalous dissipation measure advected by an autonomous velocity field is absolutely continuous. By applying \cite[Proposition 4.3]{JS24}  in Step 6 of the proof of Theorem \ref{t_Onsager}, we believe  that it can be rigorously justified that $\mu_{T, sing} \equiv 0$.
     }
    \end{itemize}
\end{remark}

We remark that the distributional identity \eqref{e:kolmo-duchon-robert} implies that the Duchon-Robert distribution is a measure (under the condition that the vanishing viscosity sequence is a sequence of suitable Leray--Hopf solutions), but this identity is not always satisfied (see Proposition \ref{prop:duchon-anomalous} for an example), proving that in some cases the vanishing viscosity sequence of the $3d$ forced  Navier--Stokes equations \eqref{e:NSE} is not a good approximation in $L^p_{t,x}$ (for any $p \in [1, \infty]$) of the $3d$ forced Euler equations (see Corollary \ref{corollary:duchon-anomalous}).
This discussion motivates the following 
 open problem.
 
\begin{OQ} \label{OQ-physical}
Let $d \geq 3$. 
     Is there a (weak) {\em physical solution} $v_0$ to the $d$-dimensional Euler equations with force $F_0$ such that the distributional limit
    \begin{align*} 
        \mathcal{D}[v_0]= \lim_{\nu \to 0} (\nu |\nabla v_\nu|^2 + \mathcal{D}[v_\nu])
    \end{align*}
    holds and the Duchon--Robert distribution  $\mathcal{D}[v_0]$ (see \eqref{duchon-robert}) is not identically zero, proving in particular that in this case the Duchon--Robert distribution is a nontrivial {\em anomalous dissipation measure}. Furthermore, is it possible to have one of the following?
    \begin{itemize}
        \item $\mathcal{D}[v_0]$ is an absolutely continuous measure with respect to the Lebesgue measure $\mathcal{L}^1 \otimes \mathcal{L}^d$.
        \item The spatial dimension of $\supp  ( \mathcal{D}[v_0]) $ is $\gamma \in (d-1, d)$.
        We refer to  \cite[Definition 2.5 and Definition 2.8]{DRI22} for a mathematical definition of the spatial dimension of $\supp  \mathcal{D}[v_0]$.\footnote{The definition of the spatial dimension of the set $\supp ( \mathcal{D}[v_0])$ should be a mathematical definition coherent to the ``fractal dimension'' used in \cite{M74,M75} and described in \cite{FSN78} as  ``a measure of the extent to which the regions in which dissipation is concentrated fill space'', which can possibly differ from those proposed in \cite{DRI22} for the purpose of the open question.} {See also \cite{LDRTDMIPI25} where the relation between intermittency and lower dimensional dissipation is studied.}
    \end{itemize} 
\end{OQ}

To the authors knowledge, also finding a single example of the identity \eqref{e:kolmo-duchon-robert} (where the Duchon--Robert distribution is non-trivial) is an open problem, which is indeed one part of Open Question \ref{OQ-physical} (see \cite[Section B]{eyinksurvey} for a discussion about that identity).
The reason for which it is relevant to ask for $\mathcal{D}[v_0]$ to be absolutely continuous with respect to $\mathcal{L}^1 \otimes \mathcal{L}^d$ is the homogeneity assumption in \cite{kolmo}. The second point of Open Question \ref{OQ-physical} is motivated by the $\beta$-model for intermittent turbulent flows proposed by Frisch-Sulem-Nelkin in \cite{FSN78}, which predicts a correction of the so called structure function  in terms of the spatial dimension of $\supp  \mathcal{D}[v_0]$. In \cite{AGH84} the authors fit experimental data to predict that the spatial dimension of $\mathcal{D}[v_0]$ for turbulent flows is  close to $2.8 \in (2,3)$ in dimension $d=3$.

\begin{remark}
    Theorem \ref{t_Onsager}, in this theory, provides an example which is close to get the identity \eqref{e:kolmo-duchon-robert} thanks to \eqref{e:mu-duchon-robert} and  we believe (without a proof of it) that there exists $0< t_1 < t_2 < 1$ such that  $\supp (\mathcal{D}[v_0]) = [t_1, t_2] \times A \times \{ 1/2\} \times \T  \subset (0,1) \times \T^2 \times \T \times \T = (0,1) \times \T^4$, where $\mathcal{L}^2(A) >0$ which in particular implies that the dimension of $\supp (\mathcal{D}[v_0])$ is $4$ in the $5$-dimensional set $(0,1) \times \T^4$.  Indeed, notice that if we were able to prove the identity \eqref{e:kolmo-duchon-robert} proving in particular that the Duchon--Robert distribution is a measure, we could have applied \cite[Theorem 1.2]{DRDI22} to conclude that in our case the Hausdorff dimension of the Duchon--Robert distribution  is at least $4$  in $(0,1) \times \T^4$ (see also \cite{LS18,LS18-SIAM,S05} for related results). 
\end{remark}

Another interesting open problem is asking if one can get the absolutely continuous in time anomalous dissipation measure with the  optimal regularity extending Theorem \ref{t_Onsager}. 
 Notice that in \cite{BCCDLS22} the authors get the Onsager critical regularity $L^3_t C^{1/3 - \varepsilon}_x$ but the anomalous dissipation measure in that case is expected to be a Dirac delta at time $t=1$ (without a proof of it). 

\begin{OQ}
 Let $\varepsilon>0$, $d \geq 3$. Is there a (weak) {\em physical  solution} to the $d$-dimensional Euler equations with force $F_0$ such that the following hold: \eqref{e:kolmo-duchon-robert}, \eqref{e:zeroth-law}, the corresponding anomalous dissipation measure in time (see Definition \ref{d:weak-physical})
 is absolutely continuous w.r.t. the Lebesgue measure $\mathcal{L}^1$ and
    $$ \sup_{\nu \in (0,1)} \| v_\nu \|_{L^3_t C^{1/3- \varepsilon}} < \infty \,?$$
\end{OQ}

  The proof of Theorem \ref{t_Onsager} relies on the study of the advection--diffusion equation 
\begin{align} \label{e:ADV-DIFF}
\tag{ADV-DIFF}
\begin{cases}
\partial_t  \theta_\kappa +  u  \cdot \nabla  \theta_\kappa = \kappa \Delta  \theta_\kappa,
\\
 \theta_{\kappa} (0, \cdot ) = \theta_{\initial} (\cdot )\, . \notag
 \end{cases}
\end{align}
 where $u$ is a given divergence free autonomous velocity field $u : \T^3 \to \R^3$, $\theta_{\initial } : \T^3 \to \R$ is a given initial datum, $\kappa \geq 0$ is the diffusivity parameter and the unknown is $\theta_{\kappa} : [0,1] \times \T^3 \to \R$. In the case of $\kappa=0$ this equation is known as the advection equation. The basic physical example is the advection of the temperature through diffusive means. 
Our study is motivated by the Obukhov and Corrsin \cite{Obukhov,Corrsin} prediction of the same scaling as the Navier--Stokes equations in the K41 theory \cite{kolmo} for solutions of the advection-diffusion equation with a turbulent advecting flow in the idealized regime of small diffusivity $\kappa >0$.
 To the best of our knowledge, the first example of anomalous dissipation  
\begin{equation}\label{diss_main_OC} 
\limsup_{\kappa \to 0}  \, \kappa \int_0^1 \int_{\T^2} | \nabla \theta_\kappa|^2 \, dx\,dt >0 
\end{equation}
for the advection-diffusion equation is given  in \cite{Gautam}, where the authors proved that for any $\alpha \in [0,1)$ there exists a 2d ``mixing''  divergence free velocity field  of regularity $L^1_t C^{\alpha}_x$ for which the authors proved anomalous dissipation for some initial data in $H^2$. {Note that the authors of \cite{Gautam} prove \eqref{diss_main_OC} with $\limsup$ replaced by $\liminf$, a stronger result.} This result was later improved by \cite{CCS22}, where the authors constructed a new mixing velocity field, reminiscent of Depauw's example \cite{Depauw}, using alternating shear flows with super-exponentially decreasing and separated scales, for which they can get non selection of solutions by vanishing viscosity and anomalous dissipation in the full supercritical Obukhov--Corrsin regime. Precisely, for any $p \in [2, \infty]$ and $p^\circ $ such that $\frac{2}{p^\circ} + \frac{1}{p} =1$ and for any $\alpha, \beta $ such that $\alpha + 2 \beta <1$ the authors proved that there exists a 2d divergence free velocity field $u \in L^p_t C^{\alpha}_x$ such that there exists $\theta_{\initial} \in C^\infty$ and the unique solutions of the advection-diffusion equation satisfy $\sup_{\kappa \in (0,1)} \| \theta_{ \kappa} \|_{ L^{p^\circ}((0,1); C^\beta_x)} < \infty$\footnote{This regularity has been proved for any initial datum.} and exhibit anomalous dissipation.
In a  recent paper \cite{AV23} the authors,  based on ideas from quantitative homogenization, proved that for any $\alpha < 1/3$ there exists a 2d velocity field $u \in L^\infty_t C^\alpha_x$ such that the unique solutions of the advection-diffusion equation exhibit anomalous dissipation for any initial data in $H^1$, with the additional property that the velocity field has not a special ``singular'' time. In a very recent paper \cite{BBS23} the authors used  convex integration  and quantitative homogenization theory to prove that  anomalous dissipation for advection diffusion equation  holds for any $H^1$ initial data advected by a dense set of weak solutions to the 3d Euler equations with regularity $u \in C^{\alpha}_{t,x}$ (for $\alpha < 1/3$ fixed but arbitrary).
Moreover it is also stated in \cite{AV23}, that in a forthcoming paper \cite{ARV23} regularity of the solutions will be proved uniformly in $\kappa \in (0,1)$ in the supercritical Obukhov--Corrsin regime with $p=p^\circ = \infty$, $\alpha + 2 \beta <1$ and $\alpha < 1/3$. 
Even more recently, in \cite{EL23} the authors proved anomalous dissipation  for any initial datum in the full supercritical Obukhov--Corrsin regime with $p = \infty$ and $p^\circ =2$ with a completely different technique, based on ideas of balanced growth of Sobolev norms. 
 To clarify our result in this setting we recall the definition of the mean energy dissipation given in \cite[p. 20]{frisch}.
\begin{definition}
Let $d \in \N$ and $f \in L^2 ((0,1) \times \T^d)$ we define the mean energy dissipation  
$$\mathcal{D}_T [f] = -  \frac{1}{2} \frac{d}{dt} \int_{\T^d} |f(x,t)|^2 dx  \,,$$
which is a distribution $\mathcal{D}_T [f] \in \mathcal{D}' (0,1)$.
\end{definition}

We remark that there are several examples in the literature for the advection equation in which the mean energy dissipation is not identically zero and $L^1 (0,1)$ (see for instance \cite{Modena1,Modena2,Modena3,BDLC20,SG21,PS21,chesluo1,chesluo2,MB22}), relying on the convex integration technique, however none of these examples are proved to be vanishing diffusivity limits of a sequence $\{ \theta_\kappa \}_{\kappa >0}$ of (unique) solutions to \eqref{e:ADV-DIFF} which satisfies \eqref{diss_main_OC}.
 Inspired by \cite[Figure 3]{A78} and using a modification   of the velocity field constructed in \cite{CCS22}, we prove the following result.

 \begin{maintheorem}\label{t_anomalous_autonomous}
 Let  $\alpha\in [0,1)$ and $\beta >0$, then there exists an autonomous divergence-free velocity field $u \in C^\alpha (\T^3)$ and an initial datum $\theta_{\initial} \in C^\infty(\T^3)$ with $ \int_{\T^3} \theta_{\initial}=0$, $\| \theta_{\initial} \|_{L^2} =1$ such that the unique solutions $\theta_\kappa$ of the advection-diffusion equation \eqref{e:ADV-DIFF} with initial datum $\theta_{\initial}$  exhibit  anomalous  dissipation \eqref{diss_main_OC}
and there exists a measure $\mu_T \in \mathcal{M}(0,1)$ so that up to non-relabelled subsequences
\begin{equation} \label{eq:absolutely_continuous}
    \mathcal{D}_T [\theta_\kappa] \rightharpoonup \mu_T 
\end{equation}
weakly* in the sense of measures. In addition, $ \| \mu_{T}\|_{TV} \geq 1/4$,  the absolutely continuous part of the measure $\mu_T$ w.r.t. the Lebsegue measure $\mathcal{L}^1$ is non-trivial, its singular part $\mu_{T, \text{sing}}$ is such that $\| \mu_{T, \text{sing}} \|_{TV} \leq \beta$ and
\begin{align}\label{e:closeness-H-1}
    \| \mathcal{D}_T [\theta_0] - \mu_T \|_{H^{-1} (0,1)} \leq \beta \,,
\end{align} 
where $\theta_0$ is the unique {bounded} solution of \eqref{e:ADV-DIFF} with $\kappa =0$.

Furthermore, $\theta_\kappa \overset{*}{\rightharpoonup} \theta_0 $ weakly* in $L^\infty ((0,1) \times \T^3)$ and (up to not relabelled subsequences) we have the closeness $\| \theta_{\kappa} - \theta_0 \|_{L^\infty ((0,1); L^2(\T^3))} < \beta$,  and 
$$e(t) = \int_{\T^3} | \theta_0 (t, x )|^2 dx  $$ 
is smooth in  $[0,1]$  and it is such that $e(1) < e(0)$.
 \end{maintheorem}
 
\begin{remark}
We remark that the property that $\mu_T$ has non-trivial absolutely continuous part in the previous theorem implies anomalous dissipation \eqref{diss_main_OC}. 
Indeed, for any $\theta_{\kappa}$ solution to \eqref{e:ADV-DIFF}, we have
\[
 \mathcal{D}_T [\theta_{\kappa}] (t) = \kappa \int_{\T^d} |\nabla \theta_{\kappa}(t,x)|^2 \, dx
\]
and up to nonrelabelled subsequences $\mathcal{D}_T [\theta_\kappa] \rightharpoonup \mu_T$.
Hence, due to lower semicontinuity on open sets under weak-* convergence in $\mathcal{M} ((0,1) \times \T^3)$, up to such a nonrelabelled subsequence we have
\[
\liminf_{\kappa \to 0} \kappa \int_{t_1}^{t_2} \int_{\T^d} |\nabla \theta_{\kappa}(t,x)|^2 \, dx dt \geq \mu_T ((t_1, t_2)) \quad \text{as $\kappa \to 0$ for all $0 \leq t_1 < t_2 \leq 1$}
\]
and since $\mu_T$ has nontrivial absolutely continuous part, we deduce \eqref{diss_main_OC}.
\end{remark}

In this context we raise the following open problem.

\begin{OQ}
    Is it possible to construct a divergence--free autonomous velocity field $u \in L^\infty(\T^2)$ (or better in $C^\alpha (\T^2)$ for $\alpha \geq 0$) and an initial datum $\theta_{\initial} \in L^\infty$ such that the unique solutions $\theta_\kappa$ of \eqref{e:ADV-DIFF} satisfy \eqref{diss_main_OC}?
\end{OQ}

\subsection*{Plan of the paper} In Section \ref{section:notation} we introduce the notations. Then, in Section \ref{sec:Construction} we construct the initial datum and an autonomous velocity field $u : \T^3 \to \R^3$ that will be used for the proofs of Theorem \ref{t_Onsager} and Theorem \ref{t_anomalous_autonomous}.  Subsequently, in Section \ref{sec:PrelimForThmA} we add some well known preliminaries that will be used in the proofs. In Section \ref{sec:proof-thB} we prove Theorem \ref{t_anomalous_autonomous} and in Section \ref{sec:4D-NS-Proof} we prove Theorem \ref{t_Onsager}. Finally, in Section \ref{sec:duchon-robert-anomalous} we point out some differences between the Duchon--Robert distribution and the anomalous dissipation measure (see Proposition \ref{prop:duchon-anomalous}) and an application (see Corollary \ref{corollary:duchon-anomalous}).

\subsection*{Acknowledgments} MS and CJ are supported by the SNSF Grant 182565 and by the Swiss State Secretariat for Education, Research and Innovation (SERI) under contract number MB22.00034.
The authors are grateful to Elia Bru\`e, Gianluca Crippa, Maria Colombo, Camillo De Lellis and Luigi De Rosa for fruitful discussions and useful remarks on the problem.

\subsection*{Data availability statement}
 Data sharing not applicable to this article as no datasets were generated or analysed during the current study.

\subsection*{Declaration regarding conflict of interest} On behalf of all authors, the corresponding author states that
there is no conflict of interest.

\section{Notation} \label{section:notation}
In this section, we explain the notation throughout the paper. First we want to make it clear that we use $x$, $y$, $z$, $w$ as spatial coordinates (up to 4 dimensions). Sometimes, however, we use $x$ as a point in $\T^3$ or $\T^4$, when there is no risk of confusion.
For vector-valued functions $v \colon \T^d \to \R^m$ ($m \geq 2$) we denote the the $j$-th component as $v^{(j)}$. When we are interested in the function built by considering just a selection of components, we adopt the following convention: Let $\{ j_1, \ldots, j_k \} \subset \{ 1, \ldots, m\}$ be distinct integers. We define $v^{(j_1, \ldots, j_k)} \colon \T^d \to \R^m$ as
\[
 v^{(j_1, \ldots, j_k)}(x) \coloneqq
 \begin{pmatrix}
 v^{(j_1)}(x) \\
 \vdots \\
 v^{(j_k)}(x) \\ 
 \end{pmatrix}.
\]

In a general dimension $d \in \N$, we denote the restriction of the gradient to the first two components as $\nabla_{x,y} = (\partial_x \,, \partial_y)$ and the restriction of the gradient to the first three components as $\nabla_{x,y,z} = (\partial_x \,, \partial_y \,, \partial_z)$.
Let $\{ A_q \}_{q \geq 1}$ and $\{ B_q \}_{q \geq 1}$ be two sequences. We write $A_q \lesssim B_q$ if there is a universal constant $C$ such that for all $q \geq 1$, we have $A_q \leq C B_q$. We write $A_q \lesssimlarge B_q$ if there are universal constants $C$ and $Q$ such that for all $q \geq Q$, $A_q \leq C B_q$. 
We use the notation $\mathcal{M} (X)$ for the space of Radon measure on a locally compact metric space $X$ and $\| \mu  \|_{TV}$ as the total variation of the measure $\mu \in \mathcal{M}(X)$ (see for instance \cite{rudin}).
{Let $\rho \colon B_1 \to \R$ be an arbitrary non-negative function satisfying $\rho \in C^\infty_c (B_1) $, $\int_{\T^d} \rho \, dx = 1$ and $\| \nabla \rho \|_{L^{\infty}(\T^d)} \leq 2$. For any $0< r <1$, we define $\rho_r \colon B_r \to \R$ by $\rho_r(x) = r^{-d} \rho(\sfrac{x}{r})$. We extend it to $\T^d$ by letting it take the value $0$ outside $B_r$. For any function $f \colon \T^d \to \R$, we define $(f)_{r} = f \star \rho_{r}$ i.e.
\[
 (f)_{r}(x) = (f \star \rho_{r})(x) = \int_{\T^d} f(y) \rho_{r}(x-y) \, dy.
\]}

\section{Construction of the vector field and initial datum}\label{sec:Construction}
The goal of this section is to construct the vector field $u$ and the initial datum ${\theta}_{in}$ in Theorem~\ref{t_anomalous_autonomous} as well as the initial datum $v_{\initial}$ and the forces $\{ F_{\nu} \}_{\nu \geq 0}$ in Theorem~\ref{t_Onsager}. The construction relies on a result from \cite{CCS22}. 
 In Subsection~\ref{subsec:ParmeterDefinitions}, we introduce all the parameters we will need. In Subsection~\ref{subsec:ConstructionVectorField}, we construct $u$ and  in Subsection~\ref{subsec:ConstructionInitialDatum}, we build ${\theta}_{in}$. Then, in Subsection~\ref{subsec:ConstructionsForNavierStokes4D}, we build $v_{\initial}$ and $\{ F_{\nu} \}_{\nu \geq 0}$ for Theorem~\ref{t_Onsager}.
\subsection{Choice of parameters}\label{subsec:ParmeterDefinitions}
We consider  $\alpha  < 1$ as in the statement of Theorem \ref{t_anomalous_autonomous}  and 
\begin{enumerate}
 \myitem{P1} $1 - \alpha(1 + \eps \delta)(1 + \delta) - \frac{\delta}{4} > 0$, \label{item:EpsDeltaOne}
 \smallskip
 \myitem{P2} $\eps \leq \frac{\delta^3}{200}$, \label{item:EpsDeltaThree}
\end{enumerate}
We notice  that the existence of $\eps$ and $\delta$  sufficiently small satisfying \eqref{item:EpsDeltaOne} follows from the condition $\alpha <1$ and up to consider $\eps = \eps (\delta) >0$ smaller we can {take} $\eps $ and $\delta$ such that \eqref{item:EpsDeltaOne}, \eqref{item:EpsDeltaThree} are satisfied simultaneously.

Given $a \in (0,1)$ such that $a^{\sfrac{\eps \delta}{8}} +  a^{{\eps \delta^2}}  \leq \frac{1}{20}$, we select $a_0 \leq a$ (to be determined depending on $\beta >0$ and universal constants) and define
\begin{align}\label{d:a_q-sequence}
a_{q + 1}  =  a_q^{1 + \delta}
\end{align}
 for all $q \geq 0$.
  The sequence  $\{ a_{q } \}_{q \in \N}$ should be selected in such a way that we can guarantee that $\sfrac{a_{q}}{a_{q+1}}$ is a multiple of 4. This affects the proof only with fixed numerical constants in the estimates and would make the reading more technical, therefore we avoid it.
We fix the parameter
\begin{align} \label{d:gamma}
\gamma =  \frac{\delta}{8}.
\end{align}
 
We fix an integer $m \geq 1$ such that
\begin{equation}\label{eq:InequalityForTheConstantM}
 m - 1 \geq \frac{16}{\delta^2}
\end{equation}
and define the sequence of times
\[
 \left\{
 \begin{array}{ll}
  t_q  =  \frac{1}{4} a_q^{\gamma} & \text{for all $q \in \N$}; \\
  \smallskip
  \overline{t}_q  = \frac{1}{4} a_q^{\gamma - \gamma \delta} & \text{for all $q \in m\N \setminus \{ 0 \}$}; \\
  \smallskip
  \overline{t}_q  =  0 & \text{for all $q \not\in m\N$}. \\
 \end{array}
 \right.
\]
Define the sequence $\{ T_q \}_{q = 0}^{\infty}$ as 
\[
 T_q  =  \sum_{j \geq q} \overline{t}_j + 3 \sum_{j \geq q} t_j < \frac{1}{4} \quad \forall q \geq 0.
\]
Define the intervals
\begin{align*}
 &\mathcal{I}_{-1}  =  (\sfrac{1}{4}, \sfrac{1}{2} - T_0]; \\
 &\mathcal{I}_{q,0}  =  (\sfrac{1}{2} - T_q, \sfrac{1}{2} - T_q + \overline{t}_q] \quad \forall q \geq 0;\\
 &\mathcal{I}_{q,i}  =  (\sfrac{1}{2} - T_q + \overline{t}_q + (i-1)t_q, \sfrac{1}{2} - T_q + \overline{t}_q + i t_q] \quad \forall q \geq 0, \, \forall i = 1,2,3;
\end{align*}
and 
$$ \mathcal{I}_q = \bigcup_{i=0}^3 \mathcal{I}_{q,i} \,, \qquad \mathcal{I} = \bigcup_{q =0}^{\infty}  \bigcup_{i=1}^3 \mathcal{I}_{q,i}$$
Finally, we define the sequences of diffusivity parameters $\{ \kappa_q \}_{q=0}^{\infty}$ and the sequence of viscosity parameters $\{ \nu_q \}_{q=0}^{\infty}$ as
\begin{equation} \label{d:k_q}
 \kappa_q = \nu_q  =  a_q^{2 - \frac{\gamma}{1 + \delta} + 10 \eps}\,.
\end{equation}
 
  \subsection{Construction of the vector field}\label{subsec:ConstructionVectorField}
In this section we prove the existence of a velocity field, relying on \cite{CCS22}, with all the properties needed to prove Theorem \ref{t_anomalous_autonomous} and Theorem \ref{t_Onsager}. We recall that $\T^3 \cong \R^3 / \Z^3$ and we use the convenient notation of identifying the subsets $A \subset \T^3$ as $A \subset [0,1)^3$.

\begin{definition}
    We  say that $\theta: \T^2 \times [0,\infty) \to \R$ is a bounded stationary solution of the advection equation with a divergence free velocity field $u \in L^\infty ( \T^2 \times (0, \infty) )$ if 
    $$ \int_{\T^2 \times [0, \infty )} (u \theta) \cdot \nabla \varphi =0 \qquad \text{for any } \varphi \in C^\infty_c (  \T^2 \times (0,\infty))\,.$$
\end{definition}

 \begin{proposition}\label{prop:ResultFromPreviousPaper}
 \emph{(Existence of a vector field)} There exists an autonomous divergence-free velocity field ${u} \in C^{\infty}_{\text{loc}}( \T^2 \times (0,1) \setminus \{ \sfrac{1}{2} \} ) \cap C^\alpha (\T^3)$ with the following estimates:
 for any $k \in \N$ and $j \in \N$ there exists a constant $C>0$ such that
 \begin{align} \label{prop:estimate_u}
 \| \partial^k_z \nabla_{x,y}^j  u \|_{L^\infty (\T^2 \times \mathcal{I}_q)} \leq C a_q^{1- \gamma} a_q^{- k \gamma} a_{q+1}^{- j (1+ \varepsilon \delta ) } \,, \quad \supp \{ u^{(1,2)} \} \subset  \mathcal{I} \times \T^3
 \end{align}

Furthermore, extending $u$ to $\T^2 \times [0, \infty)$ 1-periodically with respect to the third variable, there exists  a bounded stationary solution $ {\theta}_{\stat}:  \T^2 \times [0,\infty) \to \R$ to the advection equation  with the following properties:
 \begin{enumerate}
 \myitem{S1} there exists a bounded function with zero-average  $\theta_{\chess} \colon \T^2 \to \R$ such that \\ $\| {\theta}_{\stat} (\cdot, \cdot , z)  - \theta_{\chess}((2a_q)^{-1} \cdot, (2a_q)^{-1} \cdot ) \|_{L^2(\T^2)}^2 < 40 a_0^{\eps \delta}$ 
  for all $z \in \mathcal{I}_{q,0} \cup \mathcal{I}_{q,1}$ for all $q \in \N$; \label{eq:SolutionTransportEquationProp1}
 \myitem{S2} $\| {\theta}_{\stat} (\cdot, \cdot , z) \|_{L^2(\T^2)}^2 = \| {\theta}_{\stat} (\cdot, \cdot , \sfrac{1}{4}) \|_{L^2(\T^2)}^2 > (1 - 40 a_0^{\eps \delta})$ for all $z \in [0, \sfrac{1}{2} - T_0)$; \label{eq:SolutionTransportEquationProp2} 
 \myitem{S3} there exists a constant $C>0$ such that for any $k, j \in \{0,1 \}$ we have 
 \begin{align} \label{eq:vartheta_stationary}
 \| \partial_z^{k} \nabla_{x,y}^{j} \theta_{\stat} \|_{ L^\infty (\T^2 \times \mathcal{I}_q)} \leq  C a_q^{- k \gamma}  a_{q+1}^{j(-1 - 3 \eps (1 + \delta))} \qquad \text{ for all } q \in \N\,.
 \end{align}
 \label{eq:SolutionTransportEquationProp3}
 \myitem{S4} \label{item:dissipation-theta_0}$\| \theta_{\stat}(\cdot, \cdot, z) \|_{L^2(\T^2)}^2 \leq a_0^{\sfrac{\eps \delta}{2}}$ for all $z > \sfrac{1}{2}$.
 \end{enumerate}
 \end{proposition} 

 \begin{proof}
 Let $\tilde u \in C_{loc}^{\infty}([0,1)  \times \T^2) \cap {C^{\alpha}([0,1] \times \T^2)}$ be the vector field constructed in \cite[Section 4.2]{CCS22} with the choice of parameters $\alpha \in (0,1)$ as in \eqref{item:EpsDeltaOne}, $\varepsilon , \delta $ as in \eqref{item:EpsDeltaThree} and $\beta =0$. We recall from \cite[Section 4.2]{CCS22} that $\tilde{u} (0, \cdot) \equiv \tilde{u} (1, \cdot) \equiv 0$. Then, we extend this velocity field $\tilde u $ in $[-1,3] \times \T^2$ defining $\tilde u (t, \cdot) \equiv 0$ for any $t \in (1, 3) \cup (-1, 0)$. It is now straightforward to check that  $\tilde u \in C^\alpha ([-1,3] \times \T^2 ;\R^2) \cap C^\infty_{\text{loc}} (([-1,3] \setminus \{ 1\}) \times \T^2 ; \R^2)$ { since $\tilde{u}$ restricted to $[0,1] \times \T^2$ is smooth around $t = 0$ and $C^{\alpha}$ around $t = 1$.} Extend this vector field to $[-1, \infty) \times \T^2$ so that it is 4-periodic in time. We now rescale and translate  the velocity field $\tilde u$ in time defining 
 $$  \overline{u} (t,x) = 4 \tilde u (4t -1, x) \qquad \text{ for all } t \in (0, \infty ) \, , x \in \T^2 \,.$$

Finally, we define
 the autonomous velocity field $u \in  C^{\infty}_{\text{loc}}((0,1) \setminus \{ \sfrac{1}{2} \} \times \T^2; \R^3)  \cap C^\alpha (\T^3 ;  \R^3)   $ as
 $$u (x,y,z) = \left( \begin{array}{c}
\overline{u}^{(1)} (z, x, y) \\
 \overline{u}^{(2)} (z, x, y) \\
 1
 \end{array} \right)  \qquad \text{ for all } (x,y,z ) \in \T^3  $$
 which verifies the estimate \eqref{prop:estimate_u} thanks to  \cite[Equations (4.20a), (4.20b), (4.20c), (4.22)]{CCS22}.
 We observe that $\overline{u}$ is bounded and 1-periodic in time, $\overline{u} \in C^\infty ([0,1/2) \times \T^2)$ and $\overline u (t, \cdot) \equiv 0$ for $t \in (1/2,1]$, therefore the forward integral curves of $\overline{u}$ are unique. Using the superposition principle (see \cite{A08}) we have that the bounded solutions of the advection equation are unique.

 Now, let $\overline{\theta} : [0, \infty) \times \T^2$ be the solution to the advection equation with velocity field $\overline{u}$ and initial datum $\overline{ \theta}_{\initial} (\cdot)= \theta_{\text{chess}, 0} \star \psi_0 (\cdot)$, where $\psi_0 (\cdot) = \psi ((2a_0)^{-1 - \epsilon \delta} \cdot)$ for a smooth and periodic function $\psi \in C^\infty (\T^2)$ and $\theta_{\text{chess}, 0}  (\cdot) = \theta_{\text{chess}} ((2a_0)^{-1} \cdot) $ and the chessboard function $\theta_{\text{chess}} \in L^\infty (\T^2) $ is defined as follows
 $$ \theta_{\text{chess}} (x,y) =\begin{cases}
1 \qquad \ \  x,y \in [0,1/2) \text{ or } x,y \in [1/2, 1);
\\
-1 \qquad \text{otherwise.}
\end{cases}$$ 
 
 We define $\theta_{\stat} \in L^\infty ((0, \infty) \times (\T^2 \times [0,\infty)) )$ as $\theta_{\stat} (t, x, y , z) = \overline{\theta} (z, x, y)$. It can be verified that $\theta_{\stat}$  is a (time-independent) distributional solution of the advection equation with velocity field $u : \T^2 \times [0, \infty ) \to \R^3$ and initial datum $\theta_{\stat}$.

 We observe that \eqref{eq:SolutionTransportEquationProp1} follows from \cite[Equations (4.18), (4.19), (4.25a) and (4.25b)]{CCS22} and the estimate
$$ \sum_{q=0}^\infty 20 a_q^{\epsilon \delta} \leq  20 \sum_{q=0}^\infty a_0^{\epsilon \delta + \epsilon \delta^2 q} \leq 40 a_0^{\epsilon \delta},$$
where the last inequality follows from the choice of $a_0$.
 Property \eqref{eq:SolutionTransportEquationProp2} follows from the choice of initial datum and 
 property \eqref{eq:SolutionTransportEquationProp3} is a consequence of \cite[Section 8]{CCS22}. Property \eqref{item:dissipation-theta_0} directly follows from Step 2 in \cite[Section 7]{CCS22}.
 \end{proof}

\subsection{Mollification of $\theta_{\chess}$} \label{rmk:AboutAMollifiedVersionOfTheChessFunction}
Since the function $\theta_{\chess}$ defined in Proposition \ref{prop:ResultFromPreviousPaper} is not smooth, we introduce a mollification of it, for future use.
We will consider the mollified function $(\theta_{\text{chess}})_{a_0^{\eps \delta}} = \theta_{\text{chess}} \star \rho_{a_0^{\eps \delta}} \colon \T^2 \to \R$.
 Notice that 
 \begin{equation*}
 \|   \theta_{\text{chess}} - (\theta_{\text{chess}})_{a_0^{\eps \delta}} \|_{L^2(\T^2)} \lesssim  {a_0^{\frac{\eps \delta}{2}}},
 \end{equation*}
 In virtue of the previous proposition, this means that
 \begin{equation}\label{eq:L2DifferenceBetweenChessboardAndItsMollification}
 \| {\theta}_{\stat} (\cdot, \cdot , z)  - (\theta_{\text{chess}})_{a_0^{\eps \delta}}({(2 a_q)^{-1}} \cdot, {(2 a_q)^{-1}} \cdot) \|_{L^2(\T^2)}^2 \lesssim  a_0^{\eps \delta}
  \text{ for all $z \in \mathcal{I}_{q,0} \cup \mathcal{I}_{q,1}$ for all $q \in \N$. }
 \end{equation}
 We also point out that
 \begin{equation}\label{eq:LInftyNormOfTheGradientOfTheMollifiedChessFunction}
  \| \nabla (\theta_{\text{chess}})_{a_0^{\eps \delta}} \|_{L^{\infty}(\T^2)} \leq {2 a_0^{- \eps \delta}}.
 \end{equation}
 
 \subsection{Construction of the initial datum ${\theta}_{in}$ and $\theta_0$ solution to \eqref{e:ADV-DIFF} with $\kappa =0$}\label{subsec:ConstructionInitialDatum}
 Let $\varphi_{\initial} \in C^{\infty}(\R)$ be such that
 \[
  \supp(\varphi_{\initial}) \subset (0,\sfrac{1}{4}), \quad 0 \leq \varphi_{in} \leq 5, \quad \| \varphi_{in} \|_{L^2(\R)} \in (1,2)
 \]
 in such a way that 
 the initial datum defined as
\begin{align}\label{eq:initialdatum_theta}
\theta_{\initial} (x,y,z) = \theta_{\stat} (x,y,z) \varphi_{\initial} (z),
\end{align}
where $\theta_{\stat}$ is the function given by Proposition \ref{prop:ResultFromPreviousPaper}, is such that  
 \begin{equation}\label{eq:L2NormOfTheInitialDatum}
  \| {\theta}_{\initial} \|_{L^2(\T^3)} =1 \,,
 \end{equation}
where we used \eqref{eq:SolutionTransportEquationProp2} of Proposition \ref{prop:ResultFromPreviousPaper}.

We notice that for all $t \in (0, 1/2)$ we have that ${\theta}_0(t,x,y,z) = \theta_{\stat}(x,y,z) \varphi_{\initial}(z-t)$ is a solution to the advection equation with initial datum $\theta_{\stat}(x,y,z) \varphi_{\initial}(z)$ and velocity field $u$. This will be used in the proof of Theorem~\ref{t_Onsager} and \ref{t_anomalous_autonomous}.

\subsection{Construction of the initial datum and the forces for Theorem~\ref{t_Onsager}}\label{subsec:ConstructionsForNavierStokes4D}
In this subsection, we define the objects relevant for the proof of Theorem~\ref{t_Onsager}. We begin by defining $u_0 = u$. 
For any $\nu \in ( \nu_{q+1},   \nu_q]$, we define 
\begin{align} \label{u_nu}
 u_{\nu} =
 \begin{pmatrix}
  u^{(1)} \mathbbm{1}_{\{ z < \sfrac{1}{2} - T_{q} \}} \\
  u^{(2)} \mathbbm{1}_{\{ z < \sfrac{1}{2} - T_{q} \}} \\
  u^{(3)} \\
 \end{pmatrix}
 =
 \begin{pmatrix}
  u^{(1)} \mathbbm{1}_{\{ z < \sfrac{1}{2} - T_{q} + \sfrac{a_q^{\gamma}}{4} \}} \\
  u^{(2)} \mathbbm{1}_{\{ z < \sfrac{1}{2} - T_{q} + \sfrac{a_q^{\gamma}}{4} \}} \\
  1 \\
 \end{pmatrix}.
\end{align}
and  $\theta_{\nu} \colon [0,1] \times \T^3 \to \R$ be the solution to the advection-diffusion equation
 \begin{equation}\label{eq:Special3DAdvDiff}
 \left\{
 \begin{array}{ll}
 \partial_t {\theta}_{\nu}+ u_{\nu} \cdot \nabla {\theta}_{\nu} = \nu \Delta {\theta}_{\nu}; \\
 {\theta}_{\nu} (0,x,y,z) = {\theta}_{\initial}(x,y,z). \\
 \end{array}
 \right.
\end{equation}
This is a slight abuse of notation since we have already defined $\theta_{\kappa}$ as solution to the advection-diffusion equation \eqref{e:ADV-DIFF} at the beginning of the paper. For this reason, we introduce the following convention: Whenever $\theta$ has a subscript with letter $\kappa$, it is a solution to \eqref{e:ADV-DIFF}, whereas if it has a subscript with letter $\nu$ it is a solution to \eqref{eq:Special3DAdvDiff}.
Now define $v_{\initial} \colon \T^4 \to \R^4$ for Theorem~\ref{t_Onsager} as
\begin{equation} \label{v-initial-nu}
 v_{\initial, \nu}(x,y,z,w)  = 
 \begin{pmatrix}
  u_{\nu}(x,y,z) \\
  \theta_{\initial}(x,y,z)
 \end{pmatrix}
 =
  \begin{pmatrix}
  u_{\nu}^{(1)}(x,y,z) \\
  u_{\nu}^{(2)}(x,y,z) \\
  u_{\nu}^{(3)}(x,y,z) \\
  \theta_{\initial}(x,y,z)
 \end{pmatrix}.
\end{equation}
Finally, we define the forces $\{ F_{\nu} \}_{\nu \geq 0}$ in Theorem~\ref{t_Onsager} as
\begin{equation}\label{eq:DefOfForces4d}
 F_{\nu}(x,y,z,w)  = 
  \begin{pmatrix}
  \partial_z u_{\nu}^{(1)} - \nu \Delta u_{\nu}^{(1)} \\
  \partial_z u_{\nu}^{(2)} - \nu \Delta u_{\nu}^{(2)} \\
  \partial_z u_{\nu}^{(3)} - \nu \Delta u_{\nu}^{(3)} \\
  0
 \end{pmatrix}
 =
   \begin{pmatrix}
  \partial_z u_{\nu}^{(1)} - \nu \Delta u_{\nu}^{(1)} \\
  \partial_z u_{\nu}^{(2)} - \nu \Delta u_{\nu}^{(2)} \\
  0 \\
  0
 \end{pmatrix}.
\end{equation}

\begin{lemma}\label{lemma:AboutTheCollectionOfBodyForces}
 The collection $\{ F_\nu \}_{\nu \in (0,1)}$ is uniformly bounded in $C^{\alpha}(\T^4;\R^4)$ and $F_{\nu} \to F_0$ in $C^{\alpha}(\T^4;\R^4)$.
\end{lemma}

\begin{proof}
{It suffices to prove that
\begin{equation}\label{e:regforce2}
\| \partial_z u^{(1,2)} \|_{C^{\alpha }(\{z \geq 1/2 - T_q + \sfrac{a_q^\gamma}{4} \} )} \to 0 
\qquad\text{ and } \qquad
\sup_{\nu \in ( \nu_{q+1},   \nu_q]} \| \nu \Delta u_{\nu} \|_{C^\alpha(\T^4  )} \to 0  \,,
\end{equation} 
as $q \to \infty$.
Indeed, this implies $F_{\nu} \to F_0$ in $C^{\alpha}(\T^4;\R^4)$ with $F_0$ given by Equation~\eqref{eq:DefOfForces4d} with $\nu = 0$.}
We estimate the first term thanks to \eqref{prop:estimate_u}  and the interpolation inequality 
 \begin{align*}
 \| \partial_z u \|_{C^{\alpha } (\{z \geq 1/2 - T_q + \sfrac{a_q^\gamma}{4} \} )} & \lesssim  a_q^{1- 2 \gamma} a_{q+1}^{- \alpha (1+ \varepsilon \delta) } = a_q^{1- 2 \gamma - \alpha  (1 + \delta )(1+ \varepsilon \delta)} \to 0
 \end{align*} 
 as $q \to \infty$, where in the last limit we used $\gamma = \delta /8$ and \eqref{item:EpsDeltaOne}. We now show the second property in~\eqref{e:regforce2} using $\nu \in (  \nu_{q+1},  \nu_q]$ and \eqref{prop:estimate_u}
 \begin{align*}
 \| \nu \Delta u_\nu \|_{C^{\alpha}(\T^4)} & \leq   \nu_q  \| \Delta u  \|_{C^{\alpha } ( \{ z < 1/2 - T_q + \sfrac{a_q^\gamma}{4} )} 
 \lesssim  a_q^{2 - \frac{\gamma}{1 + \delta} + 10 \varepsilon}  a_{q-1}^{1- \gamma} a_{q}^{-2 -2 \varepsilon \delta} a_{q}^{- \alpha  (1+ \epsilon \delta)}
 \to 0
 \end{align*}
 as $q \to \infty$, in the last we used that 
 $a_q^{2 \epsilon} a_q^{- 2 \epsilon \delta } \leq 1$ and \eqref{item:EpsDeltaOne}.

\end{proof}
 
 \section{Preliminaries }\label{sec:PrelimForThmA}
 In this section, we present some preliminary lemmas that will be used in the proofs of Theorem~\ref{t_anomalous_autonomous} and \ref{t_Onsager}.
 
\subsection{Existence and uniqueness for advection-diffusion}
\begin{lemma} \label{lemma:local-energy}
    Let $d \geq 2$, $u : [0,1] \times \T^d \to \R^d$ be a divergence free velocity field such that $u \in L^\infty((0,1); L^2 (\T^3))$ and a bounded initial datum $\theta_{\initial} \in L^\infty (\T^d)${. Then} for any $\kappa >0$ there exists a unique solution $\theta_\kappa : [0,1] \times \T^d \to \R$ to the advection-diffusion \eqref{e:ADV-DIFF} such that
    \begin{itemize}
        \item $\theta_\kappa \in L^\infty((0,1) \times \T^3) \cap L^2 ((0,1) ; H^1(\T^d))$,
        \item the local energy equality holds in the sense of distribution 
        $$ \frac{1}{2} \partial_t |\theta_\kappa |^2 +  u \cdot \nabla \theta_\kappa \theta_{\kappa} + \kappa | \nabla \theta_\kappa |^2 = \frac{\kappa}{2} \Delta |\theta_\kappa|^2 \,. $$
        \item the global energy equality holds for all $s \in [0,1]$ for all $t >s $
        \begin{equation} \label{e:energy-equality-global}
           \frac{1}{2} \int_{\T^d} |\theta_\kappa |^2 (t,x ) dx + \kappa  \int_s^t \int_{\T^d}  | \nabla \theta_\kappa |^2(u,x) \, dx du = \frac{1}{2} \int_{\T^d} |\theta_\kappa |^2 (s,x ) dx   \,. 
        \end{equation} 
    \end{itemize}
\end{lemma}
\begin{proof}[Sketch of the proof]
    We provide here a sketch of the proof since it is a well known result. Consider a mollifier $\rho_\varepsilon \in C^\infty (\T^d)$ and $\theta_{\initial , \varepsilon } = \theta_{\initial} \star \rho_\varepsilon$ and $u_\varepsilon = u \star \rho_\varepsilon$. Then there exists a unique smooth solution to 
    \begin{align*}
        \partial_t \theta_{\kappa, \varepsilon} + u_\varepsilon \cdot \nabla \theta_{\kappa, \varepsilon} = \kappa \Delta \theta_{\kappa, \varepsilon}
        \\
        \theta_{\kappa, \varepsilon } (0, \cdot) = \theta_{\initial , \varepsilon} (\cdot )
    \end{align*}
    Multiplying the previous identity by $\theta_{\kappa, \varepsilon}$ and integrating in space-time we get a uniform in $\varepsilon$ bound 
    $$ \int_{\T^3} | \theta_{\kappa, \varepsilon} (t,x)|^2 dx + 2 \kappa \int_{0}^1 \int_{\T^3} |\nabla \theta_{\kappa, \varepsilon} |^2 = \int_{\T^3} |\theta_{\initial, \varepsilon }|^2 \leq \int_{\T^3} |\theta_{\initial }|^2 \,.$$
{Moreover, $\| \theta_{\kappa, \eps}(t, \cdot) \|_{L^{\infty}(\T^d)} \leq \| \theta_{\initial, \eps} \|_{L^{\infty}(\T^d)} \leq \| \theta_{\initial} \|_{L^{\infty}(\T^d)}$ by a comparison argument.}
    Therefore,  up to subsequences, there exists $\theta_{\kappa} \in L^\infty((0,1) \times \T^3) \cap L^2((0,1); H^1(\T^3))$ such that $\theta_{\kappa, \varepsilon} \rightharpoonup^\star \theta_{\kappa}$ as $\varepsilon \to 0$. By the linearity of the equation it is straightforward to check that $\theta_{\kappa} $ is a solution to the advection-diffusion equation with velocity field $u$ and initial datum $\theta_{\initial}$. {We now prove the local energy equality by } mollifying the equation with a space-time {mollifier $\psi_{\delta} \in C^\infty ((0,1) \times \T^3)$ and multiplying the equation by $\theta_{\kappa, \delta} = \theta_{\kappa} \star \psi_{\delta} $
    $$ \partial_t |\theta_{\kappa, \delta}|^2 + 2 u_\delta \cdot \nabla \theta_{\kappa, \delta} \theta_{\kappa, \delta}  = - 2 \kappa  | \nabla \theta_{\kappa, \delta}|^2 + \kappa \Delta |\theta_{\kappa, \delta}|^2 + R_{\delta} \theta_{\kappa, \delta} \,,$$
    where $R_{\delta} =  2 u_\delta \cdot \nabla \theta_{\kappa, \delta} - 2 ( u \cdot \nabla \theta_{\kappa} )_\delta$ where $( u \cdot \nabla \theta_{\kappa} )_\delta = ( u \cdot \nabla \theta_{\kappa} ) \star \psi_{\delta}$. It is straightforward to check that $\| R_{\delta} \|_{L^1((0,1) \times \T^d)} \to 0$ and that all the other terms pass into the limit as $\delta \to 0$.}  Finally, suppose that there are two solutions $\theta_{\kappa,1}, \theta_{\kappa, 2} \in L^\infty((0,1) \times \T^3) \cap L^2 ((0,1); H^1(\T^d))$ then {by heat equation regularity we observe that $\partial_t \theta_{\kappa, 1}, \partial_t \theta_{\kappa,2} \in  L^2 ((0,1); L^2 (\T^d))$}, therefore by subtracting the two equations, multiplying the equation $\theta_{\kappa, 1} - \theta_{\kappa, 2}$ and integrating in space-time we get
    $$ \frac{d}{dt} \int_{\T^d} | \theta_{\kappa, 1}(t,x) - \theta_{\kappa,2}(t,x)|^2  dx \leq 0  \,,$$
    in the sense of distributions.
    {Finally, the fact that $\partial_t \theta_\kappa \in L^2((0,1) \times \T^d)$ implies that $\theta_\kappa \in C([0,1]; L^2 (\T^d))$ so that the global energy equality follows from integrating the local energy equality in space-time.}
\end{proof}

 \subsection{Uniqueness result for the advection equation}
\begin{lemma} \label{lemma:uniqueness}
Let $u : \T^3 \to \R^3$ be a bounded, divergence-free velocity field such that $u \in C^{\infty} (\T^2 \times [0,1/2); \R^3)$, {$u^{(3)} \equiv 1$ in $\T^3$} and {$u^{(1)} \equiv u^{(2)} \equiv 0$ in $\T^2 \times (1/2, 1)$}. Then the bounded solutions of the advection equation are unique.
\end{lemma} 
\begin{proof}
We start by proving that integral curves of $u$ are unique. Let $(x,y,z) \in \T^3$ be arbitrary. We will show that solutions $\gamma \colon [0, \infty) \to \T^3$ of
 \begin{equation}\label{eq:IntegralCurvesUniqueness}
  \left\{
  \begin{array}{l}
   \dot{\gamma}(t) = u(\gamma(t)); \\
   \gamma(0) = (x,y,z);
  \end{array}
  \right.
\end{equation}
are unique for all $(x,y,z) \in \T^3$. 
It is enough to prove that such solutions are unique in the time interval $[0, \sfrac{1}{2}]$. Indeed, if this is known, iterating this fact allows us to conclude that we have uniqueness in the time interval $[0, \infty)$. We divide the proof in two distinct cases. \\
\textbf{Case 1: Assume that \pmb{$0 \leq z < \sfrac{1}{2}$}.}
Since $u \in W^{1, \infty}_{loc}(\T^2 \times [0, \sfrac{1}{2}); \R^3)$, solutions to \eqref{eq:IntegralCurvesUniqueness} are unique for $t \in [0, \sfrac{1}{2} - z)$ {since $u^{(3)} \equiv 1$ in $\T^3$}. We prove that the limit of $\gamma(t)$ as $t \to \sfrac{1}{2} - z$ exists. Since $u \in L^{\infty}(\T^3)$, $\| \dot{\gamma} \|_{L^{\infty}((0, \sfrac{1}{2} - z))} \leq \| u \|_{L^{\infty}(\T^3)} < \infty$ and therefore $\lim_{t \to \sfrac{1}{2} - z} \gamma(t)$ exists. Denote this limit as $\gamma(\sfrac{1}{2} - z)$. 
Since {$u \equiv (0,0,1)$ in $\T^2 \times (1/2, 1)$}, we must have $\gamma(t) = \gamma(\sfrac{1}{2} - z) + (0,0,t)$ for all $t \in (\sfrac{1}{2} - z, 1 - z)$. This proves uniqueness for $t \in [0, 1 - z) \supset [0, \sfrac{1}{2}]$ as wished. \\
\textbf{Case 2: Assume that \pmb{$\sfrac{1}{2} \leq z < 1$}.}
Since {$u \equiv (0,0,1)$ in $\T^2 \times (1/2, 1)$}, it is clear that the solution is unique for $t \in [0,1-z)$ and $\gamma(t) = (x,y,z + t)$ for all $t \in [0, 1 - z]$. In the time interval $[1 - z, \sfrac{3}{2} - z]$, Case 1 above applies and yields the desired uniqueness for this second case.

Finally, the integral curves of $u$ being unique combined with the superposition principle (see \cite{A08}), implies that non-negative solutions of the advection equation are unique. Since the equation is linear, also bounded solutions are unique.
\end{proof}
 
 \subsection{Stability between advection  equations and advection-diffusion equations}
 We present a result which gives an upper bound on the $L^2$-distance between a solution to the advection-diffusion equation and the advection equation.
 \begin{lemma}\label{lemma:AdvectionDiffusionAndTransport}
 Let $d \in \N$, $0 < T < \infty$ and $\kappa > 0$. Let  $u: [0,T] \times \T^d \to \R^d$ be a bounded velocity field such that there exists an initial datum $\theta_{\initial} \in L^{{\infty}} (\T^d)$ for which there is a unique {bounded} solution  $\theta \in L^2 ((0,T); H^1 ( \T^d)  ) \cap C({[}0,T); L^2 (\T^d))$ of the advection equation.  Let also ${\theta}_{\kappa} $  be the unique solution in the class $ L^2 ((0,T); H^1 (\T^d)) \cap C({[}0,T); L^2 (\T^d))$ of the advection diffusion equation with velocity field $u$, diffusion parameter $\kappa $ and initial datum $\theta_{\initial}$. Namely, we have
 \[
  \left\{
  \begin{array}{l}
   \partial_t {\theta}_{\kappa} + u \cdot \nabla {\theta}_{\kappa} = \kappa \Delta {\theta}_{\kappa}; \\
   \partial_t {\theta} + u \cdot \nabla {\theta} = 0; \\
   {\theta}_{\kappa}(0, \cdot ) = {\theta} (0, \cdot ) = {\theta}_{\initial} (\cdot ).
  \end{array}
  \right.
 \]
 Then 
 \begin{equation*}
 \| {\theta}_{\kappa}(t, \cdot ) - {\theta} (t, \cdot) \|_{L^2(\T^d)}^2 \leq \kappa \int_0^t \| \nabla {\theta} (s, \cdot ) \|_{L^2(\T^d)}^2 \, ds \quad \text{for all} \quad t \in [0,T].
 \end{equation*}
 \end{lemma}
 \begin{proof}
The existence and uniqueness of solutions of the advection diffusion equation with $L^\infty$ velocity field and $L^2$ initial datum is classical (see for instance \cite{LSU68}). From the regularity of $\theta_\kappa$ and $\theta$ we notice that the energy equality holds
\begin{align*}
 \int_{\T^d} & | \theta_\kappa (t, x) - \theta (t, x) |^2 dx 
 \\
&  = - 2 \kappa \int_0^t \int_{\T^d} | \nabla (\theta_\kappa (s,x) - \theta (s,x) )|^2 dx ds  -2 \kappa  \int_0^t \int_{\T^d} \nabla \theta (s,x) \cdot \nabla (\theta_\kappa (s,x) - \theta (s,x))  dx ds 
\\
& \leq - 2 \kappa \int_0^t \int_{\T^d} | \nabla (\theta_\kappa (s,x) - \theta (s,x) )|^2 dx ds
\\
& \quad + \kappa  \int_0^t \int_{\T^d} | \nabla \theta (s,x)|^2 dx ds + \kappa \int_0^t \int_{\T^d} | \nabla (\theta_\kappa (s,x) - \theta (s,x) ) |^2 dx ds   
\\
& \leq \kappa  \int_0^t \int_{\T^d} | \nabla \theta (s,x)|^2 dx ds  \,.
\end{align*}
 \end{proof} 
 
 \begin{remark}
There are examples for which there are no solutions $\theta \in L^2 ((0,T); H^1 (\T^d))$ to the advection equation with a divergence free velocity field in $L^\infty ((0,T) ;W^{1,p} (\T^d))$ for any $p <\infty$ (see \cite[Theorem 1]{ACM-loss}). Therefore, the existence assumption is non trivial in the previous statement. 
 \end{remark}
 \subsection{Enhanced diffusion}
We present a result which gives a quick $L^2$-norm decaying of the solution of the heat equation  thanks to the high frequency of the initial datum.
 \begin{lemma}\label{lemma:EnhancedDiffusion}
 Let $v_{in} \in L^2(\T^d)$ with zero-average, let $\lambda \in \N$ and define $v_{\lambda, in}$ as $v_{\lambda , \initial}(x)  =  v_{\initial}(\lambda x)$ for all $x \in \T^2$. Then let $v_{\lambda} \in L^{\infty}((0, \infty); L^2(\T^d))$ be the solution to the heat equation
 \begin{align} \label{rescaled_heat_equation}
 \left\{
  \begin{array}{l}
   \partial_t v_{\lambda} = \kappa \Delta v_{\lambda}; \\
   v_{\lambda}(0,x) = v_{\lambda, \initial}(x). \\
  \end{array}
  \right.
 \end{align}
 Then the following holds
 \begin{equation*}
  \| v_{\lambda}(t, \cdot) \|_{L^2(\T^d)}^2 \leq \e^{-  \kappa \lambda^2 t} \| v_{\lambda, \initial} \|_{L^2(\T^d)}^2 \qquad \text{ for all } t \in [0, \infty) \,.
 \end{equation*}
  \end{lemma}
  \begin{proof}
  From the heat equation we have that the solution $v_{1} \in L^\infty( (0, \infty); L^2 (\T^d))$ of 
  \[
  \left\{
  \begin{array}{l}
   \partial_t v_{1} = \kappa \Delta v_{1}; \\
   v_{1}(0,x) = v_{1, \initial}(x). \\
  \end{array}
  \right.
 \]
  is such that 
  $$ \| v_{1} (t, \cdot) \|_{L^2}^2 \leq \e^{-  \kappa t} \| v_{1, \initial} (\cdot) \|_{L^2(\T^d)}^2 \qquad \text{ for all } t \in [0, \infty) \,.$$
  Observing that the unique solution in $L^\infty ((0,\infty); L^2(\T^d))$ of \eqref{rescaled_heat_equation} is $v_{\lambda} (t, x) = v_1 (\lambda^2 t, \lambda x)$, we conclude
  $$ \| v_{\lambda}(t, \cdot) \|_{L^2(\T^d)}^2 = \| v_{1}(\lambda^2 t,  \cdot) \|_{L^2(\T^d)}^2  \leq \e^{-  \kappa \lambda^2 t} \| v_{1, \initial} \|_{L^2(\T^d)}^2 = \e^{-  \kappa \lambda^2 t} \| v_{\lambda, \initial} \|_{L^2(\T^d)}^2  \qquad \text{ for all } t \in [0, \infty) \,. $$
  \end{proof}
 
 \subsection{Tail estimates for the advection-diffusion equation}
We present a quantified estimate on the tail of the solutions to the advection diffusion equation with constant coefficient in 1D starting with a compactly supported initial datum.
\begin{lemma}\label{lemma:TailEstimates1DAdvectionDiffusion}
Let $T>0$, {$\kappa > 0$} and ${\psi} \colon [0,T] \times \T \to \R$ be the unique bounded solution to
  \begin{align} \label{eqn:lemma_tail}
  \begin{cases}
  \partial_t {\psi} + \partial_x {\psi} = \kappa \partial_{xx} {\psi}; \\
   \psi(0,x) = \psi_{\initial}(x). \\
  \end{cases}
  \end{align}
 where $\psi_{in} \in L^\infty ( \T)$. Let us assume that  $\supp \psi_{in} \subset (a,b)$.  Then, defining 
 $A_c(t) =  ( a-c +t, b+c+ t)^c \subset \T$,
   we have 
 \begin{equation}\label{eq:TailEstimate1DAdvcetionDiffusion}
 \| {\psi}(t, \cdot) \|_{L^{\infty}(A_c(t))} \leq {2} \| \psi_{in} \|_{L^{\infty}(\T)} { \e^{- \frac{c^2}{4 \kappa t} }}.
 \end{equation}
\end{lemma}
\begin{proof}
 Let $\tilde \psi \colon [0, T] \times \T \to \R$ be the unique bounded solution to
  \[
  \left\{
  \begin{array}{l}
   \partial_t \tilde \psi = \kappa \partial_{xx} \tilde  \psi; \\
   \tilde \psi|_{t = 0} = \psi_{\initial}. \\
  \end{array}
  \right.
 \]
  We slightly abuse the notation extending $\psi_{\initial}$ to $\R$, namely $\psi_{\initial} (x) = \psi_{\initial}(x - \lfloor x \rfloor) $ for all $x \in \R$. Then the solution to the heat equation on $\R$  is given by
 \[
 \tilde  \psi(t,x) = \int_{- \infty}^{+ \infty} \frac{1}{\sqrt{4 \pi \kappa t}} \e^{- \frac{|x - \xi|^2}{4 \kappa t}} \psi_{\initial}(\xi) \, d \xi.
 \]
 We observe that this solution is $1$-periodic and therefore is also the solution on $\T$. { We also observe that \begin{align*}
 \star &:= \int_{\{ \xi :  |\xi| > c \}} \frac{1}{\sqrt{4 \pi \kappa t}} e^{- \frac{|\xi|^2}{4 \kappa t}} \, d \xi = 2 \int_{c}^{\infty} \frac{1}{\sqrt{4 \pi \kappa t}} e^{- \frac{\xi^2}{4 \kappa t}} \, d \xi = \frac{2}{\sqrt{\pi}} \int_{\frac{c}{\sqrt{4 \kappa t}}}^{\infty} e^{- \eta^2} \, d \eta \\
 &\leq  \frac{\sqrt{4 \kappa t}}{c \sqrt{\pi}} \int_{\frac{c}{\sqrt{4 \kappa t}}}^{\infty} 2 \eta e^{- \eta^2} \, d \eta = \frac{\sqrt{4 \kappa t}}{c \sqrt{\pi}} \int_{\frac{c^2}{4 \kappa t}}^{\infty} e^{- y} \, d y \leq \frac{\sqrt{4 \kappa t}}{c \sqrt{\pi}} e^{- \frac{c^2}{4 \kappa t}} \,,
\end{align*}
hence if $\frac{\sqrt{4 \kappa t}}{c \sqrt{\pi}} \leq 1$ then $\star \leq e^{- \frac{c^2}{4 \kappa t}}$. Otherwise, we have $\star \leq 1 \leq e^{\frac{1}{\pi}} e^{- \frac{c^2}{4 \kappa t}} \leq 2 e^{- \frac{c^2}{4 \kappa t}}$, where we used 
$e^{- \frac{1}{\pi}} \leq e^{- \frac{c^2}{4 \kappa t}}$ thanks to  $\frac{\sqrt{4 \kappa t}}{c \sqrt{\pi}} \geq 1$.
As a consequence, we have
\[
 \star \leq 2 e^{- \frac{c^2}{4 \kappa t}} \,.
\]
}
  Then,
 for any $x \in (a - c, b + c)^c$
 \begin{align*}
 \tilde \psi(t,x) &= \int_{\{ \xi :  |x - \xi| > c \}} \frac{1}{\sqrt{4 \pi \kappa t}} \e^{- \frac{|x - \xi|^2}{4 \kappa t}} \psi_{in}(\xi) \, d \xi \\
 &\leq \| \psi_{\initial} \|_{L^{\infty}(\T)} \int_{\{ \xi :  |\xi| > c \}} \frac{1}{\sqrt{4 \pi \kappa t}} \e^{- \frac{|\xi|^2}{4 \kappa t}} \, d \xi \\
 &\leq {2} \| \psi_{\initial} \|_{L^{\infty}(\T)} { \e^{- \frac{c^2}{4 \kappa t} }}.\\
 \end{align*}
 
 Let us now define ${\psi}(t,x) = \tilde \psi(t, x - t)$ for all $x \in \T$ and $t \in [0,T]$,  then ${\psi}$ is the unique bounded solution of \eqref{eqn:lemma_tail} and 
 by the previous computation we have
  $$\| {\psi}(t, \cdot) \|_{L^{\infty}(A_c(t))} = \| \tilde{\psi}(t, \cdot ) \|_{L^{\infty}((a-c, b+c)^c)} \leq {2} \| \psi_{in} \|_{L^{\infty}(\T)} {\e^{- \frac{c^2}{4 \kappa t} }}\,.$$
 \end{proof}
 Finally, we deduce the following result which applies to the special advection-diffusion equation that we are considering in this paper.
 \begin{corollary}\label{lemma:TailEstimates3DAdvectionDiffusion}
Let $T>0$ and $u : [0,T] \times \T^3 \to \R^3$ be a  bounded velocity field such that $u^{(3)} \equiv 1$. 
 Let $\theta_\kappa \colon [0,T] \times \T^3 \to \R$ be the unique bounded solution to the advection-diffusion equation with velocity field  $u$  and diffusion parameter $\kappa > 0$ and initial datum $\theta_{\initial} \in L^\infty(\T^3)$.
 If we have that $\supp(\theta_{\initial}) \subset \T^2 \times (a,b) \subset \T^3$, then
 \begin{equation*}
 \| \theta_\kappa (t, \cdot ) \|_{L^{\infty}(\T^2 \times A_c(t) )} \leq {2} \| \theta_{\initial} \|_{L^{\infty}(\T^3)} { \e^{- \frac{c^2}{4 \kappa t} }}, \qquad \text{for all } t \in [0,T]\,,
 \end{equation*}
 where $A_c(t) = (a + t - c, b + t + c)^c \subset \T$.
 \end{corollary}
 \begin{proof}
 {Define $M \colon \T \to \R$ as $M(z) = \| \theta_{\initial} \|_{L^{\infty}(\T^3)} \mathbbm{1}_{(a,b)} (z)$.} Let ${\psi} \colon [0,T] \times \T \to \R$ be the unique bounded solution of
  \begin{equation}
  \left\{
  \begin{array}{l}
   \partial_t {\psi} + \partial_z {\psi} = \kappa \partial_{zz} {\psi}; \\
   \psi(0,z) = M(z). \\
  \end{array}
  \right.
  \end{equation}
  By defining $\phi \colon [0,T] \times \T^3 \to \R$ as $\phi(t,x,y,z) = {\psi}(t,z)$ we have a solution to
  \begin{equation}
  \left\{
  \begin{array}{l}
   \partial_t \phi + u \cdot \nabla \phi = \kappa \Delta \phi; \\
   \phi(0,x,y,z) = M(z). \\
  \end{array}
  \right.
  \end{equation}
  By the comparison principle for the advection-diffusion equation (see for instance \cite[Proposition 4.4]{F08}) $|\theta(t,x,y,z)| \leq \phi(t,x,y,z)$ for all $t \in [0,T]$, $(x,y,z) \in \T^3$. Using Lemma~\ref{lemma:TailEstimates1DAdvectionDiffusion} we get
  \begin{align*}
   \| \theta(t) \|_{L^{\infty}(\T^2 \times A_c(t))} & \leq \| \phi(t) \|_{L^{\infty}(\T^2 \times A_c (t) )} = \| {\psi}(t) \|_{L^{\infty}(A_c(t))} \leq {2} \| \theta_{\initial} \|_{L^{\infty}(\T)} { \e^{- \frac{c^2}{4 \kappa t} }}.
  \end{align*}
 \end{proof}
 
 \section{Proof of Theorem~\ref{t_anomalous_autonomous}} \label{sec:proof-thB}
 The objective of this section is to prove Theorem~\ref{t_anomalous_autonomous}. In Subsection~\ref{subsec:PlanOfProofA}, we describe the plan of the proof. In the subsequent subsection we describe a partition of unity that is central in our proof. In Subsection~\ref{subsec:ProofOfTheoremAWithoutProofOfLemmas}, we give a proof of Theorem~\ref{t_anomalous_autonomous}. In this proof, we mention several lemmas which we choose not to prove immediately. Instead the proof of these lemmas are postponed to the last subsection, Subsection~\ref{subsec:ProofOfTheLemmasOfProofA}.
 Throughout this section, let $\alpha \in [0,1)$  as in Theorem~\ref{t_anomalous_autonomous} and $u : \T^3 \to \R^3$ be the autonomous velocity field constructed as in Subsection~\ref{subsec:ConstructionVectorField}. Let us fix the smooth initial datum $\theta_{\initial}$ as in \eqref{eq:initialdatum_theta}. 
 
 \subsection{Plan of the proof}\label{subsec:PlanOfProofA}
We begin by fixing some arbitrary $q \in m \N$. The main objective is to prove that 
\begin{equation}\label{eq:LTwoNormAlmostZero}
\| {\theta}_{\kappa_q}(1, \cdot ) \|_{L^{2}(\T^3)}^2 \lesssim a_0^{\eps \delta}.
\end{equation}
By selecting $a_0$ sufficiently small combined with the energy balance
$$\int_{\T^3} | \theta (t, x)|^2 \, dx + 2 \kappa_q \int_0^t \int_{\T^3} | \nabla \theta_{\kappa_q} (t, x) |^2 \, dx \, dt =  \int_{\T^3} | \theta_{\initial} (x)|^2 \, dx $$
 we obtain that 
\begin{equation} \label{eq:total-diss}
 \kappa_q \int_0^1 \| \nabla \theta_{\kappa_q}(t) \|_{L^2(\T^3)}^2 \, dt > \frac{1}{4},
\end{equation}
for any $q \in m \N$,
which proves \eqref{diss_main_OC}.  All the other properties of Theorem \ref{t_anomalous_autonomous} (and similarly for Theorem \ref{t_Onsager}) will follow from a stability estimate of the type $\| \theta_{\kappa_q} - \theta_0 \|_{L^\infty_t L^2_x}^2 \leq \beta  $ (where $\beta >0 $ is a fixed constant defined in Theorem \ref{t_Onsager}) and proving that $\kappa_q |\nabla \theta_{\kappa_q}|^2$ weakly* converges in the sense of measures to a measure $\mu \in \mathcal{M}((0,1) \times \T^4)$ such that $\mu_T  = \pi_{\#} \mu $ has a non trivial absolutely continuous part. The last property is given by trying to characterize as much as possible the sequence $\kappa_q |\nabla \theta_{\kappa_q}|^2$. 

More precisely, the proof is divided into 6 steps.
In the first step, we decompose the initial datum  on the z-axis as ${\theta}_{\initial} = \sum_{j=1}^{N_q} \theta_{\initial, j}$ such that $\supp (\theta_{\initial , j} ) \cap \supp (\theta_{\initial , k})  = \emptyset $ for any $j, k \in \{ 1,..., N_q \}$ with $|j-k| > 1$.
\begin{figure}[htp]
  \includegraphics[clip,width=11cm]{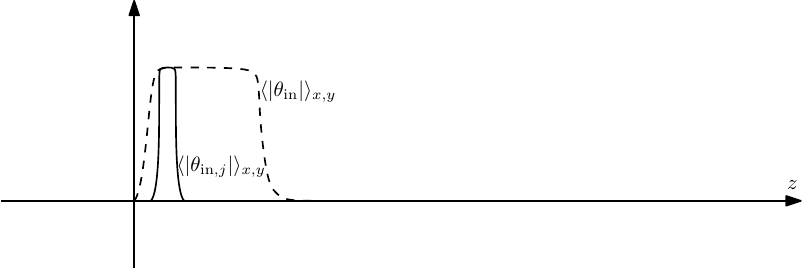}%
  \caption{The figure shows the profile $\theta_{\initial}$ averaged in $x,y$ depending only on $z$, more precisely on the vertical axis we represent the value $\langle | \theta_{\initial} | \rangle_{x,y} (z) = \int_{\T^2 } |\theta_{\initial}(x,y,z)| dx dy$ and on the horizontal axis  we have the $z$-variable.
}
\end{figure}

In Step 2, we show that the solution to the advection-diffusion equation 
$\theta_{\kappa_q, j}$ with initial datum $\theta_{\initial , j}$ and velocity field $u$
 is close to the solution to the advection equation $\theta_{0,j}$ with initial datum $\theta_{\initial , j}$ and velocity field $u$ up to a certain time. 
 
 \begin{figure}[htp]
\includegraphics[clip,width=11cm]{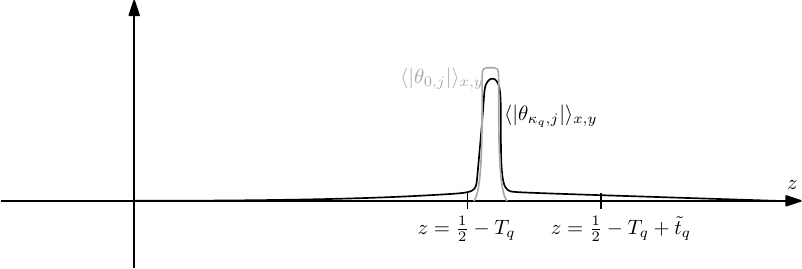}
\caption{We represent the functions 
$\langle |\theta_{0,j} | \rangle_{x,y}  $ and $\langle |\theta_{\kappa_q ,j}\rangle_{x,y}   $ at a precise time ($t= t_{q,j}$ that will be defined in \eqref{d:time_t_j}) depending only on $z$ and representing the stability between the two functions at these times $t_{q,j}$, more precisely $\langle |\theta_{0,j} | \rangle_{x,y}  (z) = \int_{\T^2} | \theta_{0,j} (t_{q,i} , x,y,z)|  dx dy$ and similarly $\langle |\theta_{\kappa_q,j} | \rangle_{x,y}  (z) = \int_{\T^2} | \theta_{\kappa ,j} (t_{q,i} , x,y,z)|  dx dy$.
}
\end{figure}

 In Step 3, we prove that  the $L^2$ norm of $\theta_{\kappa_q, j}$ decays rapidly in a certain time interval. 
 \begin{figure}[htp]
\includegraphics[clip,width=11cm]{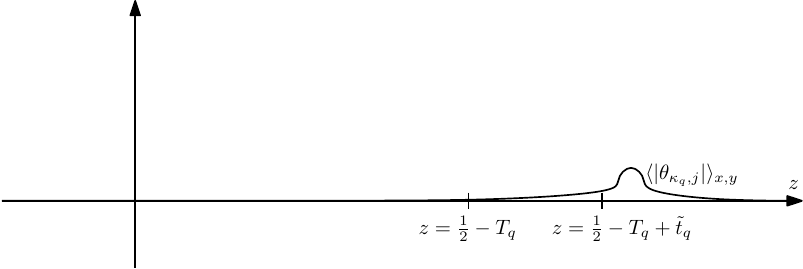}
\caption{The figure represents the phenomena of ``anomalous dissipation'' (loss of most of the $L^2$ norm) of the function $\theta_{\kappa_q, i}$. In the figure we draw the function $\langle |\theta_{\kappa_q, i}| \rangle_{x,y} (z)= \int_{\T^2} | \theta_{0,i} (t_{q,i} + \Tilde{t}_q , x,y,z)|  dx dy $  depending only on $z$, the horizontal axis, at a certain time ($t= t_{q,i} + \tilde{t}_q$ we will define later in \eqref{d:time_t_j}).}
\end{figure}

In Step 4, we use all the informations gathered until now, to prove that the solution to the advection-diffusion equation starting from ${\theta}_{\initial}$ decays significantly from time $t=0$ to time $t=1$, as in \eqref{eq:LTwoNormAlmostZero}. In Step 5, we prove that $\theta_\kappa$ weak{$^{\ast}$} converges to $\theta_0$ in $L^{\infty}((0,1) \times \T^3)$ as $\kappa \to 0$ and  $e(t) = \int_{\T^3} | \theta_0 (t, x) |^2 dx$ is a smooth, non-increasing function such that $e(1) < e(0)$. 
 \begin{figure}[htp]
\includegraphics[clip,width=11cm]{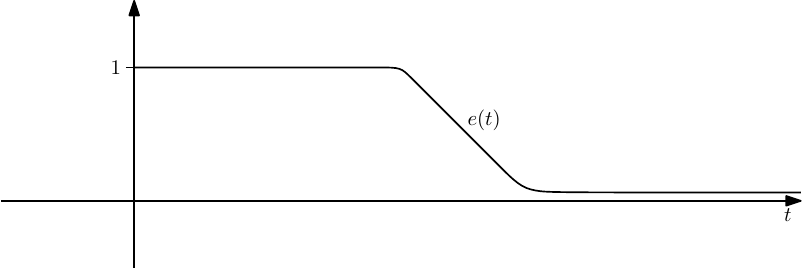}
\caption{The smooth energy profile of the limiting solution: $e(t) = \int_{\T^3} |\theta_{0} (t, x,y,z)|^2 dx dy dz$.}
\end{figure}

 Finally, in the sixth and final step, we prove that $\mu_T$ has a non-trivial absolutely continuous part.
 
\subsection{Partition of unity in the $z$ variable and corresponding time sequence}\label{subsec:PartitionOfUnityInZ}
In this section we introduce a partition of unity (used later for the $z$ variable) and a corresponding time sequence that will be useful to  describe locally in $z$ the solution $\theta_\kappa$.
 
There exists a universal constant $C > 0$ such that for every $q \in \N$ there are a sequences of smooth functions $\{ \chi_j \}_{j=1}^{N_q} \subset C^\infty (\T )$ and a sequence of  points $\{ z_j\}_{j=1}^{N_q} \subset \T$ with the following properties:
 \begin{enumerate}
 \myitem{D1} $0 \leq \chi_j \leq 1$ for all $j = 1,2, \ldots, N_q$;
 \myitem{D2} \label{support:chi_j} $\supp(\chi_j) \subset (z_j, z_j + \sfrac{a_q^{\gamma}}{8})$ for all $j = 1, 2, \ldots, N_q$;
 \myitem{D3} \label{chi:estimate-C-1} $\| \chi_j \|_{C^1(\T)} \leq C a_0^{-1} a_q^{- \gamma}$ for all $j = 1, \ldots, N_q$;
 \myitem{D4} $\sum_{i = 1}^{N_q} \chi_j \equiv 1$ on $\T$;
 \myitem{D5} \label{item:BoundOnTheNumberOfCutoffs} $N_q \leq C a_{q}^{- \gamma}$;
 \myitem{D6} we have
 \begin{equation}\label{eq:AboutAdjacentSupports}
  |\supp(\chi_i) \cap \supp(\chi_j)| \leq a_0 a_q^{\gamma}
 \end{equation} 
 and
 \begin{equation}\label{eq:DisjointnessOfTheSupportsOfCutoffs}
  \inf_{|i-j| >1} \dist \left ( \supp(\chi_{j}) \,, \supp(\chi_i) \right )  \geq \frac{ a_q^{\gamma}}{16} \,.
 \end{equation} 
 \end{enumerate}
 
For future reference, we also define 
\begin{equation} \label{d:time_t_j}
t_{q,j}= \frac{1}{2} - T_q - z_j + \frac{a_q^{\gamma}}{8} \quad \text{and} \quad \tilde{t}_q = 4a_0^{- \eps \delta} a_q^{\frac{\gamma}{1 + \delta} - 10 \eps}.
\end{equation}
Heuristically, $t_{q,j}$ represents the time up to which $\theta_{\kappa_q, j}$ remains close to $\theta_{0, j}$ solution to the advection equation starting from the same initial datum. 
After $t_{q,j}$, $\| \theta_{\kappa_q, j}(t, \cdot) \|_{L^2(\T^3)}$ starts decaying rapidly. In fact, heuristically $\tilde{t}_q$ represents the time needed from $t_{q,j}$ for $\| \theta_{\kappa_q, j}(t, \cdot) \|_{L^2(\T^3)}$ to become small.
\subsection{Proof of Theorem~\ref{t_anomalous_autonomous}}\label{subsec:ProofOfTheoremAWithoutProofOfLemmas}
We point out that the lemmas appearing in this proof are proved in Subsection~\ref{subsec:ProofOfTheLemmasOfProofA}.
Throughout this entire proof , we assume wothout loss of generality that $q$ is sufficiently large in order to have $\tilde{t}_q \leq \overline{t}_q$ thanks to \eqref{item:EpsDeltaThree}.

\textbf{Step 1: Decomposition.}

The purpose of this step is to decompose the solution ${\theta}_{\kappa_q}$ in a suitable way. Using the partition of unity introduced above in the $z$-variable we cut the support of the initial datum ${\theta}_{in}$ into smaller parts. For each $j = 1, 2,  \ldots, N_q$, we define ${\theta}_{\initial, j} \colon \T^3 \to \R$ as ${\theta}_{\initial , j}(x,y,z)  =  {\theta}_{\initial}(x,y,z) \chi_j(z)$.
Then we define ${\theta}_{\kappa_q , j}$ as the solution to the advection-diffusion equation
\begin{equation}
 \left\{
 \begin{array}{ll}
 \partial_t {\theta}_{\kappa_q, j} + u \cdot \nabla {\theta}_{\kappa_q, j} = \kappa_q \Delta {\theta}_{\kappa_q , j}; \\
 {\theta}_{\kappa_q , j} (0,x,y,z) = {\theta}_{\initial , j}(x,y,z). \\
 \end{array}
 \right.
\end{equation}
Thanks to the linearity of the equation we have ${\theta}_{\kappa_q} = \sum_{j = 1}^{N_q} {\theta}_{\kappa_q , j}$, because $\theta_{\initial } = \sum_{j = 1}^{N_q} {\theta}_{\initial , j}$. Accordingly, we also specify that ${\theta}_{0,j } \colon [0,1/2] \times \T^3 \to \R$ defined as ${\theta}_{0,j}(t,x,y,z) = \theta_{\stat }(x,y,z) \varphi_{\initial} (z-t) \chi_j(z - t)$  is the unique solution  to the advection equation (see Lemma \ref{lemma:uniqueness})
\begin{equation} \label{d:theta_0j}
 \left\{
 \begin{array}{ll}
 \partial_t {\theta}_{0,j} + u \cdot \nabla {\theta}_{0,j} = 0; \\
 {\theta}_{0,j} (0,x,y,z) = {\theta}_{\initial, j}(x,y,z). \\
 \end{array}
 \right.
\end{equation}
\\
\textbf{Step 2: $L^2$-stability between ${\theta}_{\kappa_q, j}$ and ${\theta}_{0,j}$ until time $t_{q,j} < 1/2$.}
The goal of this step is to prove that 
\begin{equation}\label{stability:theta_kq-theta}
\| {\theta}_{\kappa_q, j}(t , \cdot ) - {\theta}_{0,j}(t, \cdot ) \|_{L^{2}(\T^3)}^2 \lesssim a_q^{\gamma + \eps} \qquad \text{for all }  t \in [0,t_{q,j}] \quad \text{for all } j =1, 2, \ldots , N_q \,.
\end{equation}
By Lemma~\ref{lemma:AdvectionDiffusionAndTransport}, we get
\begin{equation}\label{eq:StabilityEstimateApplicationInThmA}
 \| {\theta}_{\kappa_q, j}(t, \cdot ) - {\theta}_{0,j}(t, \cdot ) \|_{L^{2}(\T^3)}^2 \leq \kappa_q \int_0^t \| \nabla {\theta}_{0,j} (s, \cdot ) \|_{L^2(\T^3)}^2 \, ds \,,
\end{equation}
for all $t \in [0, t_{q,j}]$.
We will prove the following estimate:
\begin{lemma}\label{lemma:EstimateOnTheGradientOfTheSolutionToTheTransportEquation}
 There exists a universal constant $C > 0$  such that for all $ q \in m \N$ and $j = 1, \ldots, N_q$ we have
 \begin{equation}\label{eq:EstimateOnTheGradientOfTheSolutionToTheTransportEquation}
  \kappa_q \int_0^{t_{q,j}} \| \nabla {\theta}_{0,j} (s, \cdot ) \|_{L^2(\T^3)}^2 \, ds \leq C a_q^{\gamma + \eps}.
 \end{equation}
\end{lemma}
Hence, we conclude that \eqref{stability:theta_kq-theta} holds.
\\
\textbf{Step 3: Decay of $\| {\theta}_{\kappa_q, j}(t, \cdot) \|_{L^2 (\T^3)}$ in $[t_{q,j}, t_{q,j} + \sfrac{\tilde{t}_q}{2}]$.}
The goal of this step is to prove that
\begin{equation}\label{eq:AutonomousNormalGoalOfStep3}
\| {\theta}_{\kappa_q,j}(t_{q, j} + \sfrac{\tilde{t}_q}{2}, \cdot ) \|_{L^2(\T^3)}^2 \lesssim a_0^{\eps \delta} a_q^{\gamma} \,.
\end{equation}
This step is divided into 4 substeps. In Substep 3.1. we prove that ${\theta}_{\kappa_q, j} (t_{q,j} , \cdot)$ is $a_q$-periodic in the $x,y$-variables up to small $L^2$ errors. Thanks to this high periodicity structure, in the remainder of this step, via an approximation of $\theta_{\kappa_q, j}$, we prove that $\| {\theta}_{\kappa_q , j}(t, \cdot ) \|_{L^2(\T^3)}$ decays in the time interval $[t_{q,j}, t_{q,j} + \sfrac{\tilde{t}_q}{2}]$. In Substep 3.2 we introduce an approximation of $\theta_{\kappa_q, j}$ denoted by $\tilde{\theta}_{\kappa_q, j}$. Subsequently, in Substep 3.3, we prove that $\| \tilde{\theta}_{\kappa_q , j}(t, \cdot ) \|_{L^2(\T^3)}$ decays in $[t_{q,j}, t_{q,j} + \sfrac{\tilde{t}_q}{2}]$. Finally, in Substep 3.4, we use this to show that $\| {\theta}_{\kappa_q , j}(t, \cdot ) \|_{L^2(\T^3)}$ decays in $[t_{q,j}, t_{q,j} + \sfrac{\tilde{t}_q}{2}]$ too.
\\
\textbf{Substep 3.1: Almost periodicity of ${\theta}_{\kappa_q, j}(t_{q,j} , \cdot, \cdot )$.}
By Section~\ref{rmk:AboutAMollifiedVersionOfTheChessFunction} there exists $(\theta_{\text{chess}})_{a_0^{\eps \delta}} \in L^\infty (\T^2)$ such that 
\begin{align*}
 &\int_{\T^2} |{\theta}_{0,j}(t_{q,j},x,y,z) - \varphi_{\initial}( z- t_{q,j}) \chi_j(z - t_{q,j}) (\theta_{\text{chess}})_{a_0^{\eps \delta}}((2a_q)^{-1} x, (2a_q)^{-1} y)|^2 \, dx dy \\
 &= |\varphi_{\initial}( z- t_{q,j}) \chi_j (z - t_{q,j})|^2 \int_{\T^2} |\theta_{\stat}(x,y,z) - (\theta_{\text{chess}})_{a_0^{\eps \delta}}((2a_q)^{-1} x, (2a_q)^{-1} y)|^2 \, dx dy \\
 &\overset{\eqref{eq:L2DifferenceBetweenChessboardAndItsMollification}}{\lesssim }  a_0^{\eps \delta} |\varphi_{\initial}( z- t_{q,j}) \chi_j(z - t_{q,j})|^2
\end{align*}
which implies
\begin{align*}
 \| {\theta}_{0,j}(t_{q,j}, \cdot, \cdot, \cdot) & - \varphi_{\initial}( \cdot - t_{q,j})  \chi_j(\cdot - t_{q, j}) (\theta_{\text{chess}})_{a_0^{\eps \delta}}((2a_q)^{-1} \cdot, (2a_{q})^{-1} \cdot) \|_{L^2(\T^3)}^2 \lesssim  a_0^{\eps \delta} a_q^{\gamma} \| \varphi_{\initial} \|_{L^{\infty}(\T)}^2 \lesssim a_0^{\eps \delta} a_q^{\gamma}\,.
\end{align*}
Thus, adding and subtracting $\theta_{0,j}$ and using \eqref{stability:theta_kq-theta}, we find
\begin{equation}\label{eq:SymmetriesL2Distance}
 \| {\theta}_{\kappa_q ,j}(t_{q,j}, \cdot, \cdot, \cdot) - \varphi_{\initial}( \cdot - t_{q,j})\chi_j(\cdot - t_{q, j}) (\theta_{\text{chess}})_{a_0^{\eps \delta}}((2a_q)^{-1} \cdot, (2a_{q})^{-1} \cdot) \|_{L^2(\T^3)}^2 \lesssim a_{q}^{\gamma + \eps} + a_0^{\eps \delta} a_q^{\gamma} \lesssim a_0^{\eps \delta} a_q^{\gamma} \,.
\end{equation}
 \\
\textbf{Substep 3.2: Approximation of ${\theta}_{\kappa_q, j}$ on $[t_{q,j}, t_{q,j} + \sfrac{\tilde{t}_q}{2}]$.}
In this substep, we study the solution $\tilde{\theta}_{\kappa_q, j} \colon [t_{q,j}, t_{q,j} + \sfrac{\tilde{t}_q}{2}] \times \T^3 \to \R$ to
\begin{equation}\label{eq:FirstApproxPDEFirstNotNS}
\left\{
\begin{array}{l}
\partial_t \tilde{\theta}_{\kappa_q, j} + \partial_z \tilde{\theta}_{\kappa_q ,j} = \kappa_q \Delta \tilde{\theta}_{\kappa_q ,j} \qquad \text{ on } [t_{q,j}, t_{q,j} + \sfrac{\tilde{t}_q}{2}] \times \T^3;
\\
\tilde{\theta}_{\kappa_q , j}(t_{q, j}, x, y, z) = \varphi_{\initial}( z - t_{q,j})\chi_j (z - t_{q,j}) (\theta_{\text{chess}})_{a_0^{\eps \delta}}( (2 a_q)^{-1}  x,  (2 a_q)^{-1}  y);
\end{array}
\right.
\end{equation}
We can characterize the solution to \eqref{eq:FirstApproxPDEFirstNotNS}. Let $f_{\kappa_q,j} \colon [t_{q, j}, t_{q, j} + \sfrac{\tilde{t}_q}{2}] \times \T^2 \to \R$ be the solution to the heat equation
\begin{equation}\label{eq:2DHeatEquation}
\left\{
 \begin{array}{l}
  \partial_t f_{\kappa_q, j} = \kappa_q \Delta f_{\kappa_q, j}  \qquad \text{ on } [t_{q,j}, t_{q,j} + \sfrac{\tilde{t}_{q}}{2}]  \times \T^2;
   \\
  f_{\kappa_q, j}(t_{q,j},x,y) = (\theta_{\text{chess}})_{a_0^{\eps \delta}}( (2 a_q)^{-1}  x,  (2 a_q)^{-1}  y).
 \end{array}
\right.
\end{equation}
and $\psi_{\kappa_q,j} \colon [t_{q,j}, t_{q,j} + \sfrac{\tilde{t}_q}{2}] \times \T \to \R$ be the solution to
\begin{equation}\label{eq:1DAdvectionDiffusionEquationInStep3}
\left\{
 \begin{array}{l}
  \partial_t \psi_{\kappa_q,j} + \partial_z \psi_{\kappa_q,j} = \kappa_q \partial_{zz} \psi_{\kappa_q,j} \qquad \text{on }  [t_{q,j}, t_{q,j} + \sfrac{\tilde{t}_{q}}{2}]  \times \T; \\
  \psi_{\kappa_q,j}(t_{q,j},z) = \varphi_{\initial}(z- t_{q,j}) \chi_j (z-t_{q,j}).
 \end{array}
\right.
\end{equation}
Define $\tilde{\theta}_{\kappa_q, j} \colon [t_{q,j}, t_{q,j} + \sfrac{\tilde{t}_q}{2}] \times \T^3 \to \R$ as $\tilde{\theta}_{\kappa_q, j}(t,x,y,z) = \psi_{\kappa_q,j}(t,z)f_{\kappa_q, j}(t,x,y)$. By standard computations, $\tilde{\theta}_{\kappa_q, j}$ solves \eqref{eq:FirstApproxPDEFirstNotNS}.
We notice that
\[
 \partial_t ({\theta}_{\kappa_q, j} - \tilde{\theta}_{\kappa_q,j}) + u^{(1,2)} \cdot \nabla_{x,y} {\theta}_{\kappa_q, j} + \partial_z
({\theta}_{\kappa_q, j} - \tilde{\theta}_{\kappa_q,j}) = \kappa_q \Delta ({\theta}_{\kappa_q, j} - \tilde{\theta}_{\kappa_q,j}).
\]
Therefore, for any $t_{q,j} \leq t \leq t_{q,j} + \sfrac{\tilde{t}_q}{2}$, multiplying by $(\theta_{\kappa_q, j} - \tilde{\theta}_{\kappa_q , j})$ and integrating in space
\begin{align*}
 \frac{1}{2} \frac{d}{dt} \| {\theta}_{\kappa_q, j} &- \tilde{\theta}_{\kappa_q,j}\|_{L^2(\T^3)}^2 + \frac{\kappa_q}{2} \| \nabla ({\theta}_{\kappa_q, j} - \tilde{\theta}_{\kappa_q,j}) \|_{L^2(\T^2)}^2 
 \\
 &= - \int_{\T^3} (u^{(1,2)} \cdot \nabla_{x,y} {\theta}_{\kappa_q, j}) ({\theta}_{\kappa_q, j} - \tilde{\theta}_{\kappa_q,j}) \, dx dy dz - \frac{\kappa_q}{2} \| \nabla ({\theta}_{\kappa_q, j} - \tilde{\theta}_{\kappa_q,j})  \|_{L^2(\T^3)}^2 \\
 &\leq \| u^{(1,2)} \cdot \nabla_{x,y} \tilde{\theta}_{\kappa_q,j} \|_{L^2(\T^3)} \| {\theta}_{\kappa_q, j} - \tilde{\theta}_{\kappa_q,j} \|_{L^2(\T^3)} - \frac{\kappa_q}{2} \| \nabla ({\theta}_{\kappa_q, j} - \tilde{\theta}_{\kappa_q,j})  \|_{L^2(\T^3)}^2 \\
 &\leq \frac{C_P}{2 \kappa_q} \| u^{(1,2)} \cdot \nabla_{x,y} \tilde{\theta}_{\kappa_q,j} \|_{L^2(\T^3)}^2 + \frac{\kappa_q}{2 C_P} \| {\theta}_{\kappa_q, j} - \tilde{\theta}_{\kappa_q,j} \|_{L^2(\T^3)}^2 - \frac{\kappa_q}{2 C_P} \| {\theta}_{\kappa_q, j} - \tilde{\theta}_{\kappa_q,j} \|_{L^2(\T^3)}^2 \\
 &\lesssim \frac{1}{\kappa_q} \| u^{(1,2)} \cdot \nabla_{x,y} \tilde{\theta}_{\kappa_q,j}\|_{L^2(\T^3)}^2 \,,
\end{align*}
where $C_P$ is the constant in the Poincar\'{e} inequality.
Thus, integrating in time we get
\begin{align}\label{eq:DifferenceTildeHat}
 \| {\theta}_{\kappa_q, j}(t, \cdot ) - \tilde{\theta}_{\kappa_q,j}(t, \cdot ) & \|_{L^2(\T^3)}^2 + 2 \kappa_q \int_{t_{q,j}}^{t} \| \nabla ({\theta}_{\kappa_q, j} - \tilde{\theta}_{\kappa_q, j})(s, \cdot)\|_{L^2(\T^3)}^2 \, ds \notag 
 \\
 & \lesssim \frac{1}{\kappa_q} \int_{t_{q,j}}^{t} \| u^{(1,2)} \cdot \nabla_{x,y} \tilde{\theta}_{\kappa_q,j}(s, \cdot) \|_{L^2(\T^3)}^2 \, ds + \| {\theta}_{\kappa_q, j}(t_{q,j}, \cdot ) - \tilde{\theta}_{\kappa_q,j}(t_{q,j}, \cdot ) \|_{L^2(\T^3)}^2 \,
\end{align}
for any $t_{q,j} \leq t \leq t_{q,j} + \sfrac{\tilde{t}_q}{2}$.
We now estimate for any $s \in [t_{q,j}, t_{q,j} + \sfrac{\tilde{t}_q}{2}]$
\begin{align*}
 \| u^{(1,2)} \cdot \nabla_{x,y} & \tilde{\theta}_{\kappa_q,j}(s, \cdot ) \|_{L^2(\T^3)}^2 
 \\
 &= \int_{\T^3} |u^{(1,2)}(x,y,z) \cdot \nabla_{x,y} \tilde{\theta}_{\kappa_q,j}(s,x,y,z)|^2 \, dx dy dz \\
 &\overset{\eqref{prop:estimate_u}}{=} \int_{\T^3 \cap \{ z \in (1/2 - T_q, 1/2 - T_q + {\overline{t}_q} )^c \} } |u^{(1,2)}(x,y,z) \cdot \psi_{\kappa_q,j}(s,z) \nabla_{x,y} f_{\kappa_q, j} (s,x,y)|^2 \, dx dy dz \\
 &\leq \| u^{(1,2)} \|_{L^{\infty}(\T^3)}^2 \| \psi_{\kappa_q,j}(s,\cdot) \|_{L^{\infty}( (1/2 - T_q, 1/2 - T_q + {\overline{t}_q} )^c )}^2 \| \nabla_{x,y} f_{\kappa_q, j} (s) \|_{L^{2}(\T^2)}^2 
\end{align*}
First we notice that by \eqref{support:chi_j} and \eqref{d:time_t_j} we have that $\supp (\chi_j  (\cdot - t_{q,j}) ) \subset (1/2 - T_q + \sfrac{a_q^\gamma}{8} , 1/2 - T_q + \sfrac{a_q^\gamma}{4} )$.
We now apply Lemma~\ref{lemma:TailEstimates1DAdvectionDiffusion} with $a= 1/2 -T_q + \sfrac{a_q^\gamma}{8}$, $b = 1/2 -T_q + \sfrac{a_q^\gamma}{4}$ and $c = a_q^\gamma /8$ to estimate 
$$\| \psi_{\kappa_q,j}(s,\cdot) \|_{L^{\infty}( (1/2 - T_q, 1/2 - T_q + {\overline{t}_q} )^c )}^2\leq {2} \| \varphi_{\initial } \|_{L^\infty(\T)} \| \chi_j \|_{L^\infty(\T)} \exp \left( - \frac{a_q^{2 \gamma}}{ {256} \kappa_q (s- t_{q,j})} \right)\,$$
for any  $s \in [t_{q,j}, t_{q,j}  + \tilde{t}_q/2]$,
where we used that  
$$ (1/2 - T_q, 1/2 - T_q + {\overline{t}_q} )^c  \subset (1/2 - T_q +s, 1/2 - T_q +\sfrac{a_q^\gamma}{8} + \sfrac{a_q^\gamma}{4} + s )^c \qquad \text{for any } s \in [0, {\sfrac{\overline{t}_q}{2}}] \,.$$
Therefore, by standard estimates of the exponential function, we get
\begin{align*}
\| u^{(1,2)} \|_{L^{\infty}(\T^3)}^2 \| \psi_{\kappa_q,j}(s,\cdot) \|_{L^{\infty}( (1/2 - T_q, 1/2 - T_q + {\overline{t}_q} )^c )}^2 \lesssim \kappa_q^3 \tilde{t}_q^3 a_q^{- 6 \gamma}
\end{align*}
In virtue of the previous computation, \eqref{eq:SymmetriesL2Distance} and \eqref{eq:DifferenceTildeHat}, we get
\begin{align} 
 &\| {\theta}_{\kappa_q, j}(t, \cdot ) - \tilde{\theta}_{\kappa_q,j}(t, \cdot ) \|_{L^2(\T^3)}^2 + 2 \kappa_q \int_{t_{q,j}}^{t} \| \nabla ({\theta}_{\kappa_q, j} - \tilde{\theta}_{\kappa_q, j})(s, \cdot)\|_{L^2(\T^3)}^2 \, ds \notag 
 \\
 &\lesssim a_0^{\varepsilon \delta} a_q^\gamma + \frac{1}{\kappa_q} \kappa_q^3 \tilde{t}_q^3 a_q^{- 6 \gamma} \int_{t_{q,j}}^t \| \nabla_{x,y} f_{\kappa_q, j} (s, \cdot) \|_{L^{2}(\T^2)}^2 \, ds  \label{eq:TildeApproxOfHatL2}
 \\
 & \overset{\eqref{e:energy-equality-global}}{\leq} a_0^{\varepsilon \delta} a_q^\gamma + \frac{1}{2} \kappa_q \tilde{t}_q^3 a_q^{- 6 \gamma} \| f_{\kappa_q, j}(t_{q,j}, \cdot) \|_{L^2(\T^2)}^2 \lesssim a_0^{\varepsilon \delta} a_q^\gamma \,,  \notag
\end{align}
for any
$t \in [t_{q, j}, t_{q, j} + \sfrac{{\overline{t}_q}}{2}]$, where we used \eqref{d:gamma}.
\\
\textbf{Substep 3.3: Dissipation of $\tilde{\theta}_{\kappa_q,j}(t, \cdot )$ on $[t_{q, j}, t_{q, j} + \sfrac{\tilde{t}_q}{2}]$.} Recall $f_{\kappa_q, j}$ the solution to \eqref{eq:2DHeatEquation} and $\psi_{\kappa_q,j}$ the solution to \eqref{eq:1DAdvectionDiffusionEquationInStep3}.
By Lemma~\ref{lemma:EnhancedDiffusion}, 
\begin{equation}
\| f_{\kappa_q, j}(t_{q,j} + \sfrac{\tilde{t}_q}{2}) \|_{L^2(\T^2)}^2 \leq \e^{-  \frac{\kappa_q  \tilde{t}_q}{4 a_q^2}} \leq \e^{- a_0^{- \eps \delta}} \leq a_0^{\eps \delta},
\end{equation}
where we used that  $\frac{\kappa_q  \tilde{t}_q}{4 a_q^2} = a_0^{- \eps \delta}$ (see \eqref{d:time_t_j} and \eqref{d:k_q}).
Then
\begin{equation}\label{eq:DecayForThetaTilde}
\begin{split}
 \| \tilde{\theta}_{\kappa_q, j}(t_{q,j} + \sfrac{\tilde{t}_q}{2}, \cdot ) \|_{L^2(\T^3)}^2 &= \int_{\T^3} |\psi_{\kappa_q,j}(t_{q,j} + \sfrac{\tilde{t}_q}{2},z)|^2 |f_{\kappa_q, j}(t_{q,j} + \sfrac{\tilde{t}_q}{2},x,y)|^2 \, dx dy dz \\
 & \int_{\T} |\psi_{\kappa_q,j}(t_{q,j} + \sfrac{\tilde{t}_q}{2},z)|^2\, dz \int_{\T^2} |f_{\kappa_q, j}(t_{q,j} + \sfrac{\tilde{t}_q}{2},x,y)|^2 \, dx dy 
  \\
 &= \| \psi_{\kappa_q,j}(t_{q,j} + \sfrac{\tilde{t}_q}{2}, \cdot ) \|_{L^2(\T)}^2 \| f_{\kappa_q, j}(t_{q,j} + \sfrac{\tilde{t}_q}{2}) \|_{L^2(\T^2)}^2
  \\
 &\leq \| \psi_{\kappa_q,j}(t_{q,j} + \sfrac{\tilde{t}_q}{2}, \cdot ) \|_{L^2(\T)}^2 a_0^{\eps \delta} \leq a_0^{\eps \delta} a_q^{\gamma}.
\end{split}
\end{equation}
\\
\textbf{Substep 3.4: Dissipation of ${\theta}_{\kappa_q,j}(t, \cdot )$  on $[t_{q, j}, t_{q, j} + \sfrac{\tilde{t}_q}{2}]$.}
 By  \eqref{eq:TildeApproxOfHatL2} and \eqref{eq:DecayForThetaTilde} we have
\begin{align} \label{eq:dissipation-theta_kappa-j}
 \| {\theta}_{\kappa_q,j}(t_{q, j} + \sfrac{\tilde{t}_q}{2}) \|_{L^2(\T^3)}^2 &\leq 2 \| {\theta}_{\kappa_q,j}(t_{q, j} + \sfrac{\tilde{t}_q}{2}) - \tilde{\theta}_{\kappa_q, j}(t_{q, j} + \sfrac{\tilde{t}_q}{2}) \|_{L^2(\T^3)}^2  +
  2 \| \tilde{\theta}_{\kappa_q, j}(t_{q, j} + \sfrac{\tilde{t}_q}{2}) \|_{L^2(\T^3)}^2 
 \notag
 \\
 &\lesssim a_0^{\eps \delta} a_q^{\gamma}  + a_0^{\eps \delta} a_q^{\gamma} \lesssim a_0^{\eps \delta} a_q^{\gamma}
\end{align}
for $q \in m \N$.
\\

\textbf{Step 4: Decay of  $\| {\theta}_{\kappa_q} (t, \cdot) \|_{L^2 (\T^3)}$ for $ t \in (0,1)$.}
In this step, we prove \eqref{eq:LTwoNormAlmostZero}.
For each $j = 0, \ldots, N_q$, we define the sets $\tilde A_{q,j} \subset \T$ as
\[ 
\tilde A_{q,j} = \supp (\chi_j (1- \cdot )) \subset [1+ z_j , 1+ z_j + a_q^\gamma] \simeq_{\T }  [ z_j ,  z_j + \sfrac{a_q^\gamma}{8}]
\]
and the $\frac{a_q^\gamma}{32}$-neighbourhood 
$$A_{q,j} = \left \{ z \in \T : \dist (z, \tilde A_{q,j}) < \frac{a_q^\gamma}{32} \right \}\,. $$

By Corollary~\ref{lemma:TailEstimates3DAdvectionDiffusion}
\begin{equation}\label{eq:ApplicationOfTailEstimatesInThmA}
 \| {\theta}_{\kappa_q,j}(1 , \cdot ) \|_{L^{\infty}(\T^2 \times A_{q,j}^c )} \leq {4} \exp \left( - \frac{(\sfrac{ a_q^{\gamma}}{32})^2}{{4} \kappa_q } \right) = {4} \exp \left( - \frac{ a_q^{2 \gamma}}{{4} \cdot 32^2 \kappa_q} \right). \\
\end{equation}
This will allow us to conclude the argument later. Now, we observe that
\begin{equation}\label{eq:L2NormOfTheFullSolutionFirstIneq}
 \| {\theta}_{\kappa_q}(1, \cdot ) \|_{L^2(\T^3)}^2 = \left\| \sum_{j = 0}^{N_q} {\theta}_{\kappa_q,j}(1, \cdot ) \right\|_{L^2(\T^3)}^2 \leq 2 \left\| \sum_{j \text{ odd}} {\theta}_{\kappa_q,j}(1, \cdot ) \right\|_{L^2(\T^3)}^2 + 2 \left\| \sum_{j \text{ even}} {\theta}_{\kappa_q,j}(1, \cdot ) \right\|_{L^2(\T^3)}^2
\end{equation}
and
\begin{align*}
 \left\| \sum_{j \text{ odd}} {\theta}_{\kappa_q,j}(1, \cdot) \right\|_{L^2(\T^3)}^2 &= \left\| \sum_{j \text{ odd}} {\theta}_{\kappa_q,j}(1, \cdot ) \mathbbm 1_{\T^2 \times A_{q,j}} (\cdot ) + \sum_{j \text{ odd}} {\theta}_{\kappa_q,j}(1, \cdot ) \mathbbm 1_{\T^2 \times A_{q,j}^c} (\cdot ) \right\|_{L^2(\T^3)}^2 \\
 &\leq 2 \left\| \sum_{j \text{ odd}} {\theta}_{\kappa_q,j}(1, \cdot ) \mathbbm 1_{\T^2 \times A_{q,j}} (\cdot ) \right\|_{L^2(\T^3)}^2 + 2 \left\| \sum_{j \text{ odd}} {\theta}_{\kappa_q,j}(1, \cdot) \mathbbm 1_{\T^2 \times A_{q,j}^c} (\cdot) \right\|_{L^2(\T^3)}^2.
\end{align*}
The collection of open sets $\{ A_{q,j} \}_{ \{ j  \ \text{odd}  \} }$ is a collection of mutually disjoint sets due to \eqref{eq:DisjointnessOfTheSupportsOfCutoffs}. This follows from the definition of the quantities $z_j$. Therefore
\begin{align*}
\left\| \sum_{j \text{ odd}} {\theta}_{\kappa_q,j}(1, \cdot ) 1_{\T^2 \times A_{q,j}} (\cdot ) \right\|_{L^2(\T^3)}^2 &= \sum_{j \text{ odd}} \left\| {\theta}_{\kappa_q,j}(1, \cdot ) 1_{\T^2 \times A_{q,j}} \right\|_{L^2(\T^3)}^2 \leq \sum_{j \text{ odd}} \left\| {\theta}_{\kappa_q,j}(1, \cdot ) \right\|_{L^2(\T^3)}^2 \\
&\leq \sum_{j \text{ odd}} \left\| {\theta}_{\kappa_q,j}(t_{q, j} + \sfrac{\overline{t}_q}{2}, \cdot ) \right\|_{L^2(\T^3)}^2 \lesssim N_q a_0^{\eps \delta} a_q^{\gamma} \lesssim a_0^{\eps \delta}.\\
\end{align*}
By \eqref{eq:ApplicationOfTailEstimatesInThmA},
\begin{align*}
\left\| \sum_{j \text{ odd}} {\theta}_{\kappa_q,j}(1, \cdot ) \mathbbm 1_{\T^2 \times A_{q,j}^c} (\cdot ) \right\|_{L^2(\T^3)} &\leq \left\| \sum_{j \text{ odd}} {\theta}_{\kappa_q,j}(1, \cdot ) 1_{\T^2 \times A_{q,j}^c} \right\|_{L^{\infty}(\T^3)} \leq \sum_{j \text{ odd}} \left\| {\theta}_{\kappa_q,j}(1, \cdot ) \right\|_{L^{\infty}(\T^2 \times A_{q,i}^c)} \\
&\leq \sum_{j \text{ odd}} {4} \exp \left( - \frac{ a_q^{2 \gamma}}{{4} \cdot 32^2 \kappa_q} \right)
 \lesssim \sum_{j \text{ odd}}  \frac{{4} \cdot 32^2 \kappa_q}{ a_q^{2 \gamma}} \lesssim a_q^{2 - 3\gamma - \frac{\gamma}{1 + \delta} + 10 \eps}  \leq a_q \,.
\end{align*}
Thus, $\left\| \sum_{j \text{ odd}} {\theta}_{\kappa_q,j}(1, \cdot ) \right\|_{L^2(\T^3)}^2 \lesssim a_0^{\eps \delta} + a_q \lesssim a_0^{\eps \delta}$.
The same bound can be proved for the last term in \eqref{eq:L2NormOfTheFullSolutionFirstIneq}. Therefore, we conclude \eqref{eq:LTwoNormAlmostZero}. 

\noindent \textbf{Step 5: Smoothness of the energy $e$ and $L^\infty_{t} L^2_x$ closeness.}
We observe that, by Lemma~\ref{lemma:uniqueness}, there exists a unique {bounded} solution $\theta_0 \in L^\infty ((0,1) \times \T^3)$ of the advection equation, with velocity field $u: \T^3 \to \R^3$ and initial datum $\theta_{\initial} \in L^\infty (\T^3)$ defined in \eqref{eq:initialdatum_theta}. We observe also that any $L^\infty ((0,1) \times \T^3)$-weakly* converging subsequence of $ \{ \theta_{\kappa} \}_{\kappa >0}$ (solutions of the advection diffusion equation with initial datum $\theta_{\initial}$) converges to a solution of the advection equation with initial datum $\theta_{\initial}$. Therefore we conclude that $\theta_\kappa \overset{*}{\rightharpoonup} \theta_0$, as $\kappa \to 0$.
By definition of $\theta_0$ in Subsection~\ref{subsec:ConstructionInitialDatum}, it is clear that 
$$e(t) =\int_{\T^3} | \theta_{0} (t, x) |^2 dx $$
is smooth in {$[0,1/2)$}.   Thanks to Lemma \ref{lemma:uniqueness} we can uniquely find $\theta_0 : [0,1 ] \times \T^3 \to \R$ a solution to the advection equation with velocity field $u$ and initial datum $\theta_{\initial}$. We also notice that, thanks to $u^{(3)} \equiv 1$ and $u (x,y,z) \equiv (0,0,1)$ for $z \geq 1/2$ and $\supp (\theta_{\initial}) \subset  \T^2 \times (0, 1/4) $ (see \eqref{eq:initialdatum_theta}), there exists $\tau >0$ such that $\int_{\T^3} | \theta_0 (t,x) |^2 dx$ is constant in {$(1/2 - \tau , 1]$}, therefore $e(t) = \int_{\T^3} | \theta_{0} (t, x) |^2 dx$   is smooth in {$[0,1]$}.

Finally, we need to prove that for any $\kappa_q $ such that $q \in m \N$ we have the $L^\infty_{t} L^2_x$ closeness of $\theta_{\kappa_q}$ to $\theta_0$, more precisely we claim that there exists a universal constant $C>0$ such that 
\begin{equation}\label{eq:L2ClosenessInTermsOfLimsup}
\limsup_{q \in m\N, \, q \to \infty} \| \theta_{\kappa_q} - \theta_0 \|_{L^\infty ((0,1); L^2(\T^3))}^2 \leq C a_0^{\epsilon \delta}.
\end{equation}
To prove the inequality above we estimate

\begin{align}
\begin{split}\label{eq:ComputationInOrderToComputeL-Inf-L-2}
    \int_{\T^3} |\theta_{\kappa_q} (x ,t) & -\theta_0 (x ,t)|^2 dx   
     =  \int_{\T^3} \left| \sum_{j=0}^{N_q} (\theta_{k_q, j} - \theta_{0,j}) \right|^2
    \\
    & \leq   \underbrace{\sum_{j=0 }^{N_q} \int_{\T^3} |\theta_{\kappa_q, j} (x ,t) -\theta_{0,j} (x ,t)|^2 }_{= I (t)} + \underbrace{ \sum_{i \neq j}^{N_q} \int_{\T^3} |\theta_{\kappa_q, j} (x ,t) -\theta_{0,j} (x ,t)| |\theta_{\kappa_q, i} (x ,t) -\theta_{0,i} (x ,t)|}_{= II(t)}
\end{split}
\end{align}
We start by estimating $\| II \|_{L^{\infty}((0,1))}$.
Using that $| \theta_{\kappa_q, j} (x,y,z, t)| \leq \psi_{\kappa_q, j} (z,t ) $ and $|\theta_{0,j} (x,y,z,t)| \leq \chi_{j} (z-t) $ for any $(x,y, z) \in \T^3$ and $t \in (0,1)$,
we estimate $II(t)$ 
\begin{align*}
    II(t) & \leq \sum_{i \neq j}^{N_q} \int_{\T^3} (\psi_{\kappa_q , j} (z,t) + \chi_{j} (z- t)) (\psi_{\kappa_q , i} (z,t) + \chi_{i} (z- t)).
\end{align*}
Let $S_{q,i}(t)$ be the time-dependent interval defined by $S_{q,i}(t) = (z_i - a_0 a_q^{\gamma} + t, z_i + \sfrac{a_q^{\gamma}}{8} + a_0 a_q^{\gamma} + t)$. By Lemma~\ref{lemma:TailEstimates1DAdvectionDiffusion}, we find that for all $t \in [0,1]$
\begin{equation}\label{eq:PsiOutsideTheSetS-q-i}
 \| \psi_{\kappa_q, i}(t, \cdot) \|_{L^{\infty}(S_{q,i}(t)^c)} \lesssim \exp\left( \frac{(a_0 a_q^{\gamma})^2}{{4} \kappa_q t} \right) \leq \frac{{4} \kappa_q t}{a_0^2 a_q^{2 \gamma}} \lesssim \frac{a_q}{a_0^2} \lesssimlarge a_0 a_q^{1/2}.
\end{equation}
Note that due to \eqref{eq:AboutAdjacentSupports} and \eqref{eq:DisjointnessOfTheSupportsOfCutoffs}, we have
\begin{equation}\label{eq:AboutHowTheSetsS-q-i-Overlap}
S_{q,i}(t) \cap S_{q,j}(t) = \emptyset \text{ whenever $|i-j| > 1$ and } |S_{q,i}(t) \cap S_{q,i + 1}(t)| \leq 3 a_0 a_q^{\gamma} \text{ for all $t \in [0,1]$.}
\end{equation}
Therefore, for all $t \in [0,1]$, we have
\begin{align*}\label{eq:AboutTwoAdjacentPsis}
\begin{split}
 \int_{\T} |\psi_{\kappa_q, i}| |\psi_{\kappa_q, i+1}|(t,z) \, dz &\leq \int_{S_{q,i}(t) \cap S_{q,i+1}(t)} |\psi_{\kappa_q, i}| |\psi_{\kappa_q, i+1}|(t,z) \, dz 
 \\
 & \quad + \int_{\T \setminus (S_{q,i}(t) \cap S_{q,i+1}(t))} |\psi_{\kappa_q, i}| |\psi_{\kappa_q, i+1}|(t,z) \, dz 
 \lesssimlarge a_0 a_q^{\gamma} + a_0 a_q^{\gamma} \lesssimlarge a_0 a_q^{\gamma}.
\end{split}
\end{align*} 
Also notice that 
$\supp (\chi_{j} (\cdot - t)) \subset S_{q,j} (t)$ for any $t \in (0,1)$ and $j =0,\ldots,N_q$ from which we get (using the convention that $\psi_{\kappa_q, N_q + 1} = \psi_{\kappa_q, 1}$ and $\chi_{N_q + 1} = \chi_1$)
\begin{align*}
    II(t) & \leq    \sum_{|i- j| > 1 }^{N_q} \int_{\T^3} (\psi_{\kappa_q , j} (z,t) + \chi_{j} (z- t)) (\psi_{\kappa_q , i} (z,t) + \chi_{i} (z- t))
    \\
    & \quad + \sum_{j=0}^{N_q} \int_{\T^3} (\psi_{\kappa_q , j} (z,t) + \chi_{j} (z- t)) (\psi_{\kappa_q , j+1} (z,t) + \chi_{j+1} (z- t))
    \\
    & \lesssimlarge  N_q^2 a_0^2 a_q + N_q a_0 a_q^{\gamma} \lesssimlarge a_0
\end{align*}
where we used \eqref{item:BoundOnTheNumberOfCutoffs}, \eqref{eq:PsiOutsideTheSetS-q-i} and \eqref{eq:AboutHowTheSetsS-q-i-Overlap}. Hence $\| II \|_{L^{\infty}((0,1))} \lesssimlarge a_0$.
\\
Now, we estimate $\| I \|_{L^{\infty}((0,1))}$.
Splitting $I$ into 3 parts and using the $L^\infty$ bound on $\theta_{\kappa_q}$ yields
\begin{align*}
    I(t)= \sum_{j=0}^{N_q} \int_{\T^3} |\theta_{\kappa_q,j} (x ,t) -\theta_{0,j} (x ,t)|^2 dx  
    & \leq  10 T_q + \underbrace{ \sum_{j=0}^{N_q} \int_{\T^2 \times \{ z < 1/2 - T_q \}} |\theta_{\kappa_q,j} (x,t) - \theta_{0,j} (x,t)|^2 dx }_{= I_1(t)} 
    \\
    & \quad  + \underbrace{ \sum_{j=0}^{N_q}  \int_{\T^2 \times \{ z > 1/2\}} |\theta_{\kappa_q,j} (x,t) - \theta_{0,j} (x,t)|^2 dx  }_{= I_2 (t)}.
\end{align*}
Observing that $T_q \to 0$ as $q \to \infty$, we only have to estimate $\| I_1 \|_{L^{\infty}((0,1))}$ and $\| I_2 \|_{L^{\infty}((0,1))}$.
We now estimate each one of the terms in $I_1$ given by
\[
 I_{1, j}(t) = \int_{\T^2 \times \{ z < 1/2 - T_q \}} |\theta_{\kappa_q,j} (x,t) - \theta_{0,j} (x,t)|^2 dx \text{ for $j = 1, \ldots, N_q$.}
\] 
If $t < t_{q,j}$ (defined in \eqref{d:time_t_j}), we use \eqref{stability:theta_kq-theta} and \eqref{item:BoundOnTheNumberOfCutoffs} to get that $|I_{1,j}(t)| \lesssim a_q^{\gamma + \eps}$.
If $t \geq t_{q,j}$, then by definition of $t_{q,j}$ and $\theta_{0,j}$ (as in \eqref{d:theta_0j}) and the fact that the third component of $u$ is constantly $1$ we have that $\supp (\theta_{0,j} (t, \cdot )) \cap \{ (x,y,z ) \in \T^3 : z < 1/2 - T_q \} = \emptyset$. Furthermore, by defining $\tilde{\psi}_{\kappa, j} : [0,1] \times \T \to \R$
\begin{equation}
\left\{
 \begin{array}{l}
  \partial_t \tilde{\psi}_{\kappa_q,j} + \partial_z \tilde{\psi}_{\kappa_q,j} = \kappa_q \partial_{zz} \tilde{\psi}_{\kappa_q,j} \text{ on } [0,1 ] \times \T
  \\
  \tilde{\psi}_{\kappa_q,j}(0,z) = \varphi_{\initial}(z) \chi_j (z).
 \end{array}
\right.
\end{equation}
we have by comparison principle that  $\theta_{\kappa, j} (t,x,y,z) \leq  \tilde{\psi}_{\kappa, j} (t,z)$ for any $(t,x,y,z ) \in [0,1] \times \T^3$ and using Lemma \ref{lemma:TailEstimates1DAdvectionDiffusion} with the set $A(t) = (z_j + t - \sfrac{a_q^{\gamma}}{8} , z_j +t + 2 \sfrac{a_q^{\gamma}}{8})$ we have
\[
\| \theta_{\kappa_q, j} (t, \cdot ) \|_{L^\infty (\T^2 \times \{ z < 1/2 - T_q\})} \leq  \| \tilde{\psi}_{\kappa_q, j} (t, \cdot ) \|_{L^\infty (A(t)^c)} \lesssim a_q^{\gamma + \epsilon}
\]
for any $t \geq t_{q,j}$, where we used 
the definition of $\kappa_q$ \eqref{d:k_q}, $t_{q,j}$ \eqref{d:time_t_j}, $z_j$ for the supports of $\chi_j$ \eqref{support:chi_j} and standard estimates on the exponential function. Therefore, we have $\| I_{1,j} \|_{L^{\infty}((0,1))} \lesssim a_q^{\gamma + \eps}$. Hence
$$ \| I_1 \|_{L^{\infty}((0,1))}  \leq \sum_{j = 1}^{N_q} \| I_{1,j} \|_{L^{\infty}((0,1))} \leq N_q a_q^{\gamma + \eps} \lesssim a_q^{\eps}.$$
We now estimate each one of the terms of $I_2$ given by
\[
 I_{2, j}(t) = \int_{\T^2 \times \{ z > 1/2\}} |\theta_{\kappa_q,j} (x,t) - \theta_{0,j} (x,t)|^2 dx.
\]
It is straightforward to see that
$$ |I_{2, j}(t)| \leq 
\underbrace{
\int_{\T^2 \times \{ z > 1/2\}} 2 |\theta_{\kappa_q,j} (x,t)|^2  dx }_{= I_{2,j}^{\prime}(t)} +
\underbrace{\int_{\T^2 \times \{ z > 1/2\}} 2   
| \theta_{0,j} (x,t)|^2 dx}_{= I_{2,j}^{\prime \prime}(t)} \,.
$$
If $t < t_{q,j} + \overline{t}_q$, then by Lemma~\ref{lemma:TailEstimates1DAdvectionDiffusion} with $\tilde{\psi}_{\kappa, j}$ and $A(t)$ as before we have
$$ I_{2,j}^{\prime}(t) \leq 2 \| \theta_{\kappa , j} (t, \cdot ) \|_{L^\infty (\T^2 \times \{ z >1/2 \})}^2 \lesssim \| \tilde{\psi}_{\kappa, j} (t, \cdot ) \|_{L^\infty (A(t)^c)}^2 \lesssim a_q^{\gamma + \epsilon}.$$
If $t \geq t_{q,j} + \overline{t}_q$, using \eqref{eq:dissipation-theta_kappa-j} we have 
$$ I_{2,j}^{\prime}(t) \leq \| \theta_{\kappa, j} (t, \cdot) \|_{L^2}^2 \leq \| \theta_{\kappa, j} (t_{q,j} + \overline{t}_q, \cdot) \|_{L^2}^2 \leq a_0^{\epsilon \delta } a_q^{\gamma}.$$
for any $q \in m \N$. Hence $\| I_{2,j}^{\prime} \|_{L^{\infty}((0,1))} \lesssim a_0^{\eps \delta} a_q^{\gamma}$.
Using \eqref{item:dissipation-theta_0} we have that 
$$ \| I_{2,j}^{\prime \prime} \|_{L^{\infty}((0,1))} \lesssim a_0^{\epsilon \delta} a_q^{\gamma}.$$
As a consequence 
$$\| I_2 \|_{L^{\infty}((0,1))} \leq \sum_{j = 1}^{N_q} \| I_{2,j}^{\prime} \|_{L^{\infty}((0,1))} + \sum_{j = 1}^{N_q} \| I_{2,j}^{\prime \prime} \|_{L^{\infty}((0,1))} \lesssim a_0^{\eps \delta}.$$
Thus, we conclude $\| I \|_{L^{\infty}((0,1))} \leq \| I_1 \|_{L^{\infty}((0,1))} + \| I_2 \|_{L^{\infty}((0,1))} \lesssim a_0^{\eps \delta}$. Due to \eqref{eq:ComputationInOrderToComputeL-Inf-L-2}, we conclude that \eqref{eq:L2ClosenessInTermsOfLimsup} holds, as wished.
 Thanks to \eqref{eq:L2ClosenessInTermsOfLimsup}, we proved that there exists $f_{\err} \in L^\infty ((0,1) \times \T^3)$ such that $\| f_{\err} \|_{L^\infty ((0,1); L^2 (\T^3))}^2 \leq C a_0^{\epsilon \delta }$ and
 $$  |\theta_{\kappa_q}|^2 \rightharpoonup^*   |\theta_0|^2 +  f_{\err}$$
 weakly* in $L^\infty$, up to subsequences. Indeed 
 \begin{align*}
     \int^{1}_0   \int_{ \T^3 } \left | | \theta_{\kappa_q}|^2 - |\theta_0|^2 \right |^2 \leq ( \| \theta_{\kappa_q} \|_{L^\infty}
     + \| \theta_0 \|_{L^\infty} )^2 \int_0^1 \int_{\T^3} |\theta_{\kappa_q} - \theta_0 |^2 \leq C \| \theta_{\kappa_q} - \theta_0 \|_{L^\infty((0,1); L^2(\T^3))}^2 \,.
 \end{align*}
   We notice that, thanks to the energy balance we have that $\{ \kappa_{q} | \nabla \theta_{\kappa_{q}} |^2\}_{q \in m \N}$ is uniformly bounded in $L^1((0,1) \times \T^3)$, and hence there exists $\mu \in \mathcal{M}([0,1] \times \T^3)$ such that, up to non relabelled subsequences we have 
$$ \kappa_{q} | \nabla \theta_{\kappa_{q}} |^2 \weak \mu \geq 0 \,,$$
in the sense of measure as $q \to \infty$ and we define $\mu_T \in \mathcal{M}(0,1)$ as $\mu_T  = \pi_{\#} \mu $. 
 Therefore, passing into the limit in the sense of distributions as $\kappa_q \to 0$  (along the sequence $\{ \kappa_{q} \}_{q \in m \N} $)  the distributional identity of the local energy equality of Lemma \ref{lemma:local-energy}  we have the following identity in the sense of distribution
 $$ \frac{1}{2} \partial_t |\theta_0|^2 + \frac{1}{2} \partial_t f_{\err} + \frac{1}{2} \diver (u |\theta_0|^2) + \frac{1}{2} \diver (u f_{\err}) + \mu =0 \,,$$
from which, integrating in space, we get  the $H^{-1}_{t}$ closeness 
 $$ \left \| \partial_t \frac{1}{2} \int_{\T^3} |\theta_0|^2  + \mu_T \right \|_{H^{-1}_{t}} = \left \| \frac{1}{2} \partial_t \int_{\T^3} f_{\err}  \right \|_{H^{-1}_{t}} \leq C a_0^{\epsilon \delta } < \beta \,,$$
where
 $\mu_T  = \pi_{\#} \mu $,
  from which we get  \eqref{e:closeness-H-1} up to choose $a_0$ sufficiently small in terms of $\beta$ in  Subsection \ref{subsec:ParmeterDefinitions}.

\noindent \textbf{Step 6: Convergence behaviour of $-\frac{1}{2}  e_{\kappa_q}^{\prime}$ for $q \in m \N$.}
All lemmas stated in the proof will be proved at the end of Step 6.
In order to finish the proof of the theorem, we examine the behaviour of 
\[
-\frac{1}{2} e_{\kappa_q}^{\prime}(t) =  \kappa_q \int_{\T^3} |\nabla \theta_{\kappa_q}(t,x)|^2 \, dx.
\]
As observed above, by the energy balance it is clear that up to subsequences $\{ -\frac{1}{2} e_{\kappa_q}^{\prime} \}_{q  \in m \N}$ weak-$\ast$ converges to a non-negative measure $\mu_T \geq 0$. This measure can be decomposed into an absolutely continuous part and a singular part. The goal of this step is to show that the total variation of the singular part is small compared to the one of the absolutely continuous part. We define $A_{\kappa_q} \colon [0,1] \to \R$ as
\[
 A_{\kappa_q}(t) \coloneqq  \kappa_q \sum_{j = 1}^{N_q} \int_{\T^3} |\nabla \tilde{\theta}_{\kappa_q, j}(t,x)|^2 \mathbbm{1}_{[t_{q,j}, t_{q,j} + \tilde{t}_q]} (t) \, dx
\]
and $S_{\kappa_q} \colon [0,1] \to \R$ as $S_{\kappa_q}(t) = -\frac{1}{2} e_{\kappa_q}^{\prime}(t) - A_{\kappa_q}(t)$. We will show that 
\begin{equation}\label{eq:TheTwoPropertiesThatWeWant}
\limsup_{q \to \infty} \| S_{\kappa_q} \|_{L^1((0,1))} \lesssim a_0^{\frac{\eps \delta}{2}}\,, \quad  \sup_{q} \| A_{\kappa_q} \|_{L^{\infty}((0,1))} < \infty \,, \quad \| \mu_T \|_{TV} \geq 1/4 \,.
\end{equation}
This proves that up to subsequences $A_{\kappa_q}$ weak-$\ast$ converges in $L^{\infty}$ to some $\mathcal{A} \in L^{\infty}((0,1))$ and $S_{\kappa_q}$ weak-$\ast$ converges in the sense of measures to $\mathcal{S} \in \mathcal{M}((0,1))$ (which potentially has a singular part) such that $\| \mathcal{S} \|_{TV} \lesssim a_0^{\varepsilon \delta /2} < \beta < 1/4$, by choosing $a_0$ sufficiently small depending on $\beta$ and universal constants which proves that the absolutely continuous part of the measure $\mu_T = \mathcal{A} + \mathcal{S}$ is non-trivial. The proof is divided into 4 substeps. The first 3 substeps are dedicated to prove, via a sequence of approximations that $S_{\kappa_q} = -\frac{1}{2} e_{\kappa_q}^{\prime} - A_{\kappa_q}$ is uniformly small in $L^1$. In the last step, we prove that $\sup_{q} \| A_{\kappa_q} \|_{L^{\infty}((0,1))} < \infty$ and $\| \mu_T \|_{TV} \geq 1/4$. \\
\textbf{Substep 6.1:}
We have
\[
 -\frac{1}{2} e_{\kappa_q}^{\prime}(t) = \underbrace{ \kappa_q \sum_{i = 1}^{N_q} \int_{\T^3} |\nabla \theta_{\kappa_q, i}(t,x)|^2 \, dx}_{= A_{\kappa_q}^{(1)}(t)} + \underbrace{ \kappa_q \sum_{i,j = 1, i \neq j}^{N_q} \int_{\T^3} \nabla \theta_{\kappa_q, i}(t,x) \cdot \nabla \theta_{\kappa_q, j}(t,x) \, dx}_{= S_{\kappa_q}^{(1)}(t)}.
\]
The goal of this substep is to show that $\| S_{\kappa_q}^{(1)} \|_{L^1((0,1))} \lesssimlarge a_0^{\eps \delta}$. Decompose $S_{\kappa_q}^{(1)}$ into two parts
\[
 S_{\kappa_q}^{(1)}(t) =  \underbrace{\kappa_q \sum_{i = 1}^{N_q} \int_{\T^3} \nabla \theta_{\kappa_q, i}(t,x) \cdot \nabla \theta_{\kappa_q, i+1}(t,x) \, dx}_{= I(t)} + \underbrace{ \kappa_q \sum_{|i - j| > 1} \int_{\T^3} \nabla \theta_{\kappa_q, i}(t,x) \cdot \nabla \theta_{\kappa_q, j}(t,x) \, dx}_{= II(t)}
\]
with the convention that $\theta_{\kappa_q, N_q+1}(t,x) = \theta_{\kappa_q, 1}(t,x)$. 
We will prove the two following results at the end of this section:
\begin{lemma}\label{lemma:EstimateOfQuantitiy-II}
There exists a constant $Q \in \N$ and a  universal constant $C>0$ such that for all $q \geq Q$ we have
\begin{equation}\label{eq:EstimateOfQuantitiy-II}
 \| II \|_{L^1((0,1))} \leq C a_q^{\sfrac{1}{8}}.
\end{equation}
\end{lemma}
\begin{lemma}\label{lemma:EstimateOfQuantitiy-I}
There exists a constant $Q \in \N$ and a universal constants $C>0$ such that for all $q \in m \N$ such that $q \geq Q$ we have
\begin{equation}\label{eq:EstimateOfQuantitiy-I}
 \| I \|_{L^1((0,1))} \leq C a_0^{\sfrac{\eps \delta}{2}}.
\end{equation}
\end{lemma}

As a consequence, we obtain
\begin{equation}\label{eq:EstimateOfQuantityE1}
 \| S_{\kappa_q}^{(1)} \|_{L^1((0,1))} \leq \| I \|_{L^1((0,1))} + \| II \|_{L^1((0,1))} \lesssimlarge a_0^{\frac{\eps \delta}{2}}.
\end{equation}
\textbf{Substep 6.2: }
Note that 
\begin{align*}
 A_q^{(1)}(t) &=  \kappa_q \sum_{i = 1}^{N_q} \int_{\T^3} |\nabla \theta_{\kappa_q, i}(t,x)|^2 \, dx = \underbrace{ \kappa_q \sum_{i = 1}^{N_q} \int_{\T^3} |\nabla \theta_{\kappa_q, i}(t,x)|^2 \, dx \mathbbm{1}_{[t_{q,i}, t_{q,i} + \tilde{t}_q]}(t)}_{= A_{\kappa_q}^{(2)}(t)} \\
 & \quad + \underbrace{ \kappa_q \sum_{i = 1}^{N_q} \int_{\T^3} |\nabla \theta_{\kappa_q, i}(t,x)|^2 \, dx \mathbbm{1}_{[0, t_{q,i})}(t)}_{= S_{\kappa_q}^{(2)}(t)} + \underbrace{ \kappa_q \sum_{i = 1}^{N_q} \int_{\T^3} |\nabla \theta_{\kappa_q, i}(t,x)|^2 \, dx \mathbbm{1}_{(t_{q,i} + \tilde{t}_q, 1]}(t)}_{= S_{\kappa_q}^{(3)}(t)}.
\end{align*}
\begin{lemma}\label{lemma:EstimatesOfTheQuantities-E2-E3}   There exists a universal constant $C>0$  such that for all $ q \in m \N$  we have
 \begin{equation}\label{eq:EstimatesOfTheQuantities-E2-E3}
  \| S_{\kappa_q}^{(2)} \|_{L^1((0,1))} + \| S_{\kappa_q}^{(3)} \|_{L^1((0,1))} \leq C a_0^{\eps \delta}.
 \end{equation}
\end{lemma}

\textbf{Substep 6.3: }
For $q \in m\N$ and all $i = 1, \ldots, N_q$, define $\mathcal{J}_{q,i} = [t_{q,i}, t_{q,i} + \tilde{t}_q]$ and observe that 
\begin{align*}
 A_{\kappa_q}^{(2)}(t) &=  \kappa_q \sum_{i = 1}^{N_q} \int_{\T^3} |\nabla {\theta}_{\kappa_q, i}(t,x)|^2 \, dx \mathbbm{1}_{\mathcal{J}_{q,i}}(t) = \underbrace{ \kappa_q \sum_{i = 1}^{N_q} \int_{\T^3} |\nabla \tilde{\theta}_{\kappa_q, i}(t,x)|^2 \, dx \mathbbm{1}_{\mathcal{J}_{q,i}}(t)}_{= A_{\kappa_q}(t)} \\
 &\quad + \underbrace{ \kappa_q \sum_{i = 1}^{N_q} \int_{\T^3} \nabla {\theta}_{\kappa_q, i}(t,x) \cdot \nabla ({\theta}_{\kappa_q, i} - \tilde{\theta}_{\kappa_q, i})(t,x)\, dx \mathbbm{1}_{\mathcal{J}_{q,i}}(t)}_{= S_{\kappa_q}^{(4)}(t)} \\
 &\quad + \underbrace{ \kappa_q \sum_{i = 1}^{N_q} \int_{\T^3} \nabla ({\theta}_{\kappa_q, i} - \tilde{\theta}_{\kappa_q, i})(t,x) \cdot \nabla \tilde{\theta}_{\kappa_q, i}(t,x)\, dx \mathbbm{1}_{\mathcal{J}_{q,i}}(t)}_{= S_{\kappa_q}^{(5)}(t)} \\
\end{align*}
\begin{lemma}\label{lemma:EstimateOfQuatities-E6-E7}
There exist universal constant $C>0$  such that for all $ q \in m \N$  we have 
 \begin{equation}\label{eq:EstimateOfQuatities-E6-E7}
  \| S_{\kappa_q}^{(4)} \|_{L^1((0,1))} + \| S_{\kappa_q}^{(5)} \|_{L^1((0,1))} \leq C a_0^{\sfrac{\eps \delta }{2}}\,.
 \end{equation}
\end{lemma}

Notice that $S_{\kappa_q} = \sum_{\ell = 1}^{5} S_{\kappa_q}^{(\ell)}$ and therefore in virtue of \eqref{eq:EstimateOfQuantityE1}, \eqref{eq:EstimatesOfTheQuantities-E2-E3},  and \eqref{eq:EstimateOfQuatities-E6-E7}
\begin{equation}\label{eq:L1BoundOnTheError}
\| S_{\kappa_q} \|_{L^1((0,1))} \lesssimlarge a_0^{\frac{\eps \delta}{2}}. 
\end{equation}
This proves the first inequality in \eqref{eq:TheTwoPropertiesThatWeWant}.

\textbf{Substep 6.4: } In this step we prove that  $\| \mu_T \|_{TV} \geq 1/4$ and $\sup_{q} \| A_{\kappa_q} \|_{L^{\infty}((0,1))} < \infty$.
Firstly, we notice that $\mu_T$ is the weak* limit up to subsequence (in the sense of measure) of the non-negative sequence $\{ \kappa_q \int_{\T^3} |\nabla \theta_{\kappa_q} |^2 \}_{q \in m \N }$ that is such that \eqref{eq:total-diss} holds, therefore $\| \mu_T \|_{TV} \geq 1/4$ holds.

We now prove the second property.
Recall that $\tilde{\theta}_{\kappa_q, i}$ solves \eqref{eq:FirstApproxPDEFirstNotNS} and $\mathcal{J}_{q,i} = [t_{q,i}, t_{q,i} + \tilde{t}_q]$. Define the function $\overline{N}_q \colon [0,1] \to \N$ as
\[
 \overline{N}_q(t) = \# \{ j \in \{ 1, \ldots, N_q \} : \mathbbm{1}_{\mathcal{J}_{q,i}}(t) \neq 0 \}
\]
It follows from straightforward computations that $\| \overline{N}_q(t) \|_{L^{\infty}((0,1))} \lesssim a_0^{- \eps \delta} a_q^{\frac{\gamma}{1 + \delta} - 10 \eps - \gamma}$.
For all $t \in \mathcal{J}_{q,i}$, 
\[
\| \nabla \tilde{\theta}_{\kappa_q, i}(t, \cdot) \|_{L^2(\T^3)}^2 \leq \| \nabla \tilde{\theta}_{\kappa_q, i}(t_{q,i}, \cdot) \|_{L^2(\T^3)}^2 \overset{\eqref{eq:LInftyNormOfTheGradientOfTheMollifiedChessFunction}}{\lesssim} a_0^{-6 \eps \delta} a_q^{-2 + \gamma}.
\]
From this we get the $L^\infty$ bound on $A_{\kappa_q}$ independent on $q$:
\begin{align*}
 A_{\kappa_q}(t) &\leq  \kappa_q \sum_{i = 1}^{N_q} \int_{\T^3} |\nabla \tilde{\theta}_{\kappa_q, i}(t,x)|^2 \, dx \mathbbm{1}_{\mathcal{J}_{q,i}} (t) =  \kappa_q \sum_{i = 1}^{N_q} \| \nabla \tilde{\theta}_{\kappa_q, i}(t,\cdot) \|_{L^2(\T^3)}^2 \mathbbm{1}_{\mathcal{J}_{q,i}} (t) \\
 &\leq \kappa_q \overline{N}_q(t) a_0^{-6 \eps \delta} a_q^{-2 + \gamma} \lesssim a_0^{-7 \eps \delta} \quad \forall t \in (0,1)\,,
\end{align*}
which concludes the proof of the second property. Therefore, up to subsequences
\[
  A_{\kappa_q} \overset{*}{\rightharpoonup} \mathcal{A} \in L^{\infty}((0,1)), \quad 
  S_{\kappa_q} \overset{*}{\rightharpoonup} \mathcal{S} \in \mathcal{M}([0,1]) \quad \text{and} \quad - \frac{1}{2}  e_{\kappa_q}^{\prime} \overset{*}{\rightharpoonup} \mu_T \in \mathcal{M}([0,1])  \,,
 \]
 where the first convergence is weakly* in $L^\infty$ and the other two are weakly* in $\mathcal{M}(0,1)$.
 In virtue of \eqref{eq:L1BoundOnTheError}, we find that $\| \mu_T - \mathcal{A} \|_{TV} = \| \mathcal{S} \|_{TV} \lesssim a_0^{\frac{\eps \delta}{2}} < \beta$ where the last inequality holds up to choosing $a_0$ sufficiently small depending on universal constants and $\beta $  in  Subsection \ref{subsec:ParmeterDefinitions}. Thanks to the two properties proved in this step we conclude the proof of Theorem~\ref{t_anomalous_autonomous}.

\subsection{Proof of Lemmas~\ref{lemma:EstimateOnTheGradientOfTheSolutionToTheTransportEquation}, \ref{lemma:EstimateOfQuantitiy-II}, \ref{lemma:EstimateOfQuantitiy-I}, \ref{lemma:EstimatesOfTheQuantities-E2-E3},  and \ref{lemma:EstimateOfQuatities-E6-E7}}\label{subsec:ProofOfTheLemmasOfProofA}

\begin{proof}[Proof of Lemma~\ref{lemma:EstimateOnTheGradientOfTheSolutionToTheTransportEquation}]
In order to estimate the right-hand side, we estimate the following three quantities:
\begin{enumerate}
 \item $\displaystyle \int_{0}^{t_{q,j}} \int_{\T^3} |\partial_x {\theta}_{0,j}(t,x,y,z)|^2 \, dx dy dz \, dt$ \label{item:QuantityToEstimate1} \\
 \item $\displaystyle \int_{0}^{t_{q,j}} \int_{\T^3} |\partial_y {\theta}_{0,j}(t,x,y,z)|^2 \, dx dy dz \, dt$ \label{item:QuantityToEstimate2}\\
 \item $\displaystyle \int_{0}^{t_{q,j}} \int_{\T^3} |\partial_z {\theta}_{0,j}(t,x,y,z)|^2 \, dx dy dz \, dt$ 
\label{item:QuantityToEstimate3}\\
\end{enumerate}
\fbox{Estimate of \eqref{item:QuantityToEstimate1}:}
We notice that $|\chi_j (z-t)| = 0$ for any $z \in (z_j +t, z_j + t + \sfrac{a_q^\gamma}{8})^c$, from which we estimate
\begin{align*}
 &\int_{0}^{t_{q,j}} \int_{\T^3} |\partial_x {\theta}_{0,j}(t,x,y,z)|^2 \, dx dy dz \, dt = \int_{0}^{t_{q,j}} \int_{\T} |\varphi_{\initial}(z-t)|^2 |\chi_j(z-t)|^2 \int_{\T^2} |\partial_x \theta_{\stat}(x,y,z)|^2 \, dx dy \, dz \, dt \\ 
 &\leq \| \varphi_{\initial} \|_{L^{\infty}(\T)}^2 \int_{0}^{t_{q,j}} \int_{z_j + t}^{z_j + t + \sfrac{a_q^\gamma}{8}} \int_{\T^2} |\partial_x \theta_{\stat}(x,y,z)|^2 \, dx dy \, dz \, dt \\
  &= \| \varphi_{\initial} \|_{L^{\infty}(\T)}^2 \int_{z_j}^{\sfrac{1}{2} - T_q + \sfrac{a_q^{\gamma}}{4}} \int_{\T^2} \int_{z - z_j - \sfrac{a_q^{\gamma}}{8}}^{z - z_j}  |\partial_x \theta_{\stat}(x,y,z)|^2  \, dt \, dx dy  \, dz \\
  &\lesssim a_q^{\gamma} \int_{0}^{\sfrac{1}{2} - T_q + \sfrac{a_q^{\gamma}}{4}} \int_{\T^2} |\partial_x \theta_{\stat}(x,y,z)|^2 \, dx dy \, dz 
  \end{align*}
  Using \eqref{eq:vartheta_stationary}  we have
  \begin{align*}
  &= a_q^{\gamma} \Bigg[ \int_{0}^{\sfrac{1}{2} - T_0} \underbrace{\int_{\T^2} |\partial_x \theta_{\stat}(x,y,z)|^2 \, dx dy}_{\lesssim a_0^{- 2(1 + 3 \eps (1 + \delta))}} \, dz + \sum_{j = 0}^{q-1} \int_{\sfrac{1}{2} - T_j}^{\sfrac{1}{2} - T_{j+1}} \underbrace{\int_{\T^2} |\partial_x \theta_{\stat}(x,y,z)|^2 \, dx dy}_{\lesssim a_{j+1}^{-2(1 + 3 \eps (1 + \delta))}} \, dz \\
  &\qquad + \int_{\sfrac{1}{2} - T_q}^{\sfrac{1}{2} - T_{q} + a_q^{\gamma }} \underbrace{\int_{\T^2} |\partial_x \theta_{\stat}(x,y,z)|^2 \, dx dy}_{\lesssim a_q^{-2(1 + 3 \eps(1 + \delta))}} \, dz \Bigg] \\
  &\lesssim a_q^{\gamma} \Bigg[ a_0^{- 2(1 + 3 \eps (1 + \delta))} + \sum_{j = 0, j \in m \N }^{q-m}  a_j^{\frac{\gamma}{1 + \delta} - 10 \eps} a_{j+1}^{-2(1 + 3 \eps (1 + \delta))} + \sum_{j = 0, j \notin m \N }^{q-1}  a_j^\gamma a_{j+1}^{-2(1 + 3 \eps (1 + \delta))}  + a_q^{\gamma} a_q^{-2(1 + 3 \eps (1 + \delta))} \Bigg]
  \\
  & \lesssim a_q^{\gamma} \left[ a_0^{- 2(1 + 3 \eps (1 + \delta))} + q a_{q-m}^{\frac{\gamma}{1 + \delta} - 10 \eps - 2(1 + 3 \eps (1 + \delta))(1 + \delta)} +  q a_{q-1}^{\gamma - 2(1 + 3 \eps (1 + \delta))(1 + \delta)} +  a_q^{\gamma -2(1 + 3 \eps (1 + \delta))} \right]\,
  \end{align*}
  where we have used that $q \in m \N$ and $a_{j+1} = a_j^{1+ \delta}$ for any $j$.
  Using  also \eqref{eq:InequalityForTheConstantM}  we get
  $$ a_{q-m +1}^{\frac{\gamma (1- \delta)}{(1+\delta)^2} - \frac{10 \eps}{1 + \delta} - 2 (1+ 3 \varepsilon (1+ \delta))} \leq a_q^{\frac{\gamma}{1+\delta} - 2 (1+ 3 \varepsilon (1+ \delta))},$$
concluding that 
$$ \int_{0}^{t_{q,j}} \int_{\T^3} |\partial_x {\theta}_{0,j}(t,x,y,z)|^2 \, dx dy dz \, dt \lesssim q a_{q}^{\gamma + \frac{\gamma}{1 + \delta} - 2(1 + 3 \eps (1 + \delta))}\,.$$
\fbox{Estimate of \eqref{item:QuantityToEstimate2}:} In the same way of \eqref{item:QuantityToEstimate1}.
\\
\fbox{Estimate of \eqref{item:QuantityToEstimate3}:}
We note that 
\[
 \partial_z {\theta}_{0,j} = \partial_z (\varphi_{\initial}(z-t) \chi_j(z - t)) \theta_{\stat}(x,y,z) + \varphi_{\initial}(z-t) \chi_j(z - t) \partial_z \theta_{\stat}(x,y,z)\,.
\]
Using again \eqref{eq:vartheta_stationary}, \eqref{chi:estimate-C-1} and $\| \varphi_{\initial} \|_{C^1} \lesssim 1$ 
we estimate
$$ \| \partial_z {\theta}_{0,j} \|_{L^\infty} \lesssim a_0^{-1} a_q^{- \gamma} \lesssim  a_{q}^{\gamma + \frac{\gamma}{1 + \delta} - 2(1 + 3 \eps (1 + \delta))} \,,$$
where in the last we used $\gamma = \sfrac{\delta}{8} <\frac{1}{8}$ in the last inequality. 
Finally, from the estimates of \eqref{item:QuantityToEstimate1}, \eqref{item:QuantityToEstimate2} and \eqref{item:QuantityToEstimate3} and \eqref{d:k_q}, we find that
\[
\| {\theta}_{\kappa_q, j}(t) - {\theta}_{0,j}(t) \|_{L^2(\T^3)}^2 \lesssim  q a_q^{2 - \frac{\gamma}{1 + \delta} + 10 \eps + \frac{\gamma}{1 + \delta} + \gamma - 2(1 + 3 \eps (1 + \delta))}  \lesssim a_q^{\gamma + \eps} \quad \text{for all} \quad  t \in [0, t_{q,j}],
\]
where we also used that $q a_q^{\varepsilon} \leq 1$ for any $q \in \N$.
\end{proof}

\begin{proof}[Proof of Lemma~\ref{lemma:EstimateOfQuantitiy-II}]
We estimate the terms of $II$ given by
\[
 II_{(i,j)}(t) =  \kappa_q \int_{\T^3} \nabla \theta_{\kappa_q, i}(t,x) \cdot \nabla \theta_{\kappa_q, j}(t,x) \, dx \quad \text{with $i,j \in \{ 1, \ldots, N_q\}$ and $|i - j| > 1$}
\]
in $L^1((0,1))$.
For each $q \geq 1$, let $\{ \overline{\chi}_{j} \}_{j = 1}^{N_q} \subset C^{\infty}_0(\T)$ be a sequence of cutoff functions such that 
\begin{enumerate}
 \myitem{$\overline{\mbox{D1}}$} $0 \leq \overline{\chi}_{j} \leq 1$ for all $j = 1, \ldots, N_q$;
 \myitem{$\overline{\mbox{D2}}$} $\overline{\chi}_{j} = 1$ on $\supp(\chi_{j})$ for all $j = 1, \ldots, N_q$;
 \myitem{$\overline{\mbox{D3}}$} $\| \overline{\chi}_{j} \|_{C^1(\T)} \lesssim a_0^{-1} a_q^{- \gamma}$ for all $j = 1, \ldots, N_q$; \label{item:AboutTheC1NormOfTheSecondCollectionsOfCutoffs}
 \myitem{$\overline{\mbox{D4}}$} we have
 \begin{equation}
  |\supp(\overline{\chi}_{i}) \cap \supp(\overline{\chi}_{j})| \leq 3 a_0 a_q^{\gamma} \quad \forall i,j \in \{ 1, \ldots, N_q\};
 \end{equation}
 and
 \begin{equation}\label{eq:DistanceBetweenTwoNonAdjacentsCutoffsSecondCollection}
  \inf_{|i-j| > 1} \dist ( \supp(\overline{\chi}_{i}), \supp(\overline{\chi}_{j}) ) \geq \frac{a_q^{\gamma}}{16}.
 \end{equation}
\end{enumerate}
For each $j = 1, \ldots, N_q$, let $\overline{\psi}_{\kappa_q, j} \colon [0,1] \times \T \to \R$ be the solution to the advection-diffusion equation
 \begin{equation}\label{eq:1D-Adv-Diff-Cutoffs}
 \left\{
 \begin{array}{ll}
 \partial_t \overline{\psi}_{\kappa_q, j} + \partial_z \overline{\psi}_{\kappa_q, j} = \kappa_q \partial_{zz} \overline{\psi}_{\kappa_q, j}; \\
 \overline{\psi}_{\kappa_q, j}(0, \cdot) = \overline{\chi}_j(\cdot). \\
 \end{array}
 \right.
\end{equation}
Note that due to Lemma~\ref{lemma:TailEstimates1DAdvectionDiffusion} and \eqref{eq:DistanceBetweenTwoNonAdjacentsCutoffsSecondCollection}, whenever we have two integers $i$ and $j$ such that $|i - j| > 1$, it holds that
\begin{equation}\label{eq:TwoNonAdjacentsCutoffsRemainFarAppartSomehow}
 \| \overline{\psi}_{\kappa_q, i} \overline{\psi}_{\kappa_q, j} \|_{L^{\infty}((0,1) \times \T)} \lesssim \kappa_q a_q^{- 2 \gamma}.
\end{equation}
For each $j$, define $\overline{\theta}_{\kappa_q, j}$ as $\overline{\theta}_{\kappa_q, j}(t,x,y,z) = {\theta}_{\kappa_q, j}(t,x,y,z) \overline{\psi}_{\kappa_q, j}(t,z)$. By standard computations, we note that 
 \begin{equation}\label{eq:1D-Adv-Diff-Cutoffs-Prod-With-Solution}
 \left\{
 \begin{array}{ll}
 \partial_t \overline{\theta}_{\kappa_q, j} + u \cdot \nabla \overline{\theta}_{\kappa_q, j} = \kappa_q \Delta \overline{\theta}_{\kappa_q, j} - 2 \kappa_q \nabla {\theta}_{\kappa_q, j} \cdot \nabla \overline{\psi}_{\kappa_q, j}; \\
 \overline{\theta}_{\kappa_q, j}(0, \cdot) = {\theta}_{\kappa_q, j}(0, \cdot). \\
 \end{array}
 \right.
\end{equation}
By energy estimates, we find
\begin{align*}
 &\| ({\theta}_{\kappa_q, j} - \overline{\theta}_{\kappa_q, j})(t, \cdot) \|_{L^2(\T^3)}^2 + 2 \kappa_q \int_0^t \int_{\T^3} |\nabla({\theta}_{\kappa_q, j} - \overline{\theta}_{\kappa_q, j})|^2 \, dx ds \\
 &= 4 \kappa_q \int_0^t \int_{\T^3} (\nabla {\theta}_{\kappa_q, j} \cdot \nabla \overline{\psi}_{\kappa_q, j})({\theta}_{\kappa_q, j} - \overline{\theta}_{\kappa_q, j}) \, dx ds \\
 &\leq 8 \kappa_q \int_0^t \int_{\T^3} | \nabla {\theta}_{\kappa_q, j} \cdot \nabla \overline{\psi}_{\kappa_q, j}| \, dx ds \\
 &\leq 4 \left( 2 \kappa_q \int_0^t \int_{\T^3} |\nabla {\theta}_{\kappa_q, j}|^2 \, dx ds \right)^{\sfrac{1}{2}} \left( 2 \kappa_q \int_0^t \int_{\T^3} |\nabla \overline{\psi}_{\kappa_q, j}|^2 \, dx ds \right)^{\sfrac{1}{2}} \\
 &\lesssim \kappa_q^{\sfrac{1}{2}} \| \theta_{\kappa_q, j}(0, \cdot) \|_{L^2(\T^3)} \| \partial_z \overline{\chi}_{j} \|_{L^{\infty}(\T)} \\
 &\overset{\eqref{item:AboutTheC1NormOfTheSecondCollectionsOfCutoffs}}{\lesssim} \kappa_q^{\sfrac{1}{2}} a_0^{-1} a_q^{- \gamma} a_q^{\sfrac{\gamma}{2}} = \kappa_q^{\sfrac{1}{2}} a_0^{-1} a_q^{- \sfrac{\gamma}{2}} \lesssimlarge a_q^{\sfrac{1}{2}}.
\end{align*}
In particular,
\begin{equation}\label{eq:Useful-Stability-Between-Theta-And-Approximation}
2 \kappa_q \int_0^t \int_{\T^3} |\nabla({\theta}_{\kappa_q, j} - \overline{\theta}_{\kappa_q, j})|^2 \, dx ds \lesssimlarge a_q^{\sfrac{1}{2}}.
\end{equation}
Now, instead of estimating the terms of $II$ in $L^{1}((0,1))$, we will estimate
\[
 \overline{II}_{(i,j)}(t) =  \kappa_q \int_{\T^3} \nabla \overline{\theta}_{\kappa_q, i}(t,x) \cdot \nabla \overline{\theta}_{\kappa_q, j}(t,x) \, dx \quad \text{with $i,j \in \{ 1, \ldots, N_q\}$ and $|i - j| > 1$}
\]
in $L^{1}((0,1))$ for $i$ and $j$ such that $|i - j| > 1$. 
We will then use \eqref{eq:Useful-Stability-Between-Theta-And-Approximation} to show that this is a good approximation of the terms of $II$ in $L^{1}((0,1))$. We observe that
\begin{align*}
\| \overline{II}_{(i,j)} & \|_{L^1((0,1))} = \left\|  \kappa_q \int_{\T^3} \nabla \overline{\theta}_{\kappa_q, i}(\cdot,x) \cdot \nabla \overline{\theta}_{\kappa_q, j}(\cdot,x) \, dx \right\|_{L^1((0,1))} =  \kappa_q \int_0^1 \left| \int_{\T^3} \nabla \overline{\theta}_{\kappa_q, i} \cdot \nabla \overline{\theta}_{\kappa_q, j} \, dx \right| \, dt 
\\
&\leq \underbrace{ \kappa_q \int_0^1 \left| \int_{\T^3} (\nabla \overline{\psi}_{\kappa_q, i} \cdot \nabla \overline{\psi}_{\kappa_q, j}) ({\theta}_{\kappa_q, i} {\theta}_{\kappa_q, j}) \, dx \right| \, dt}_{= B_1} 
 + \underbrace{ \kappa_q \int_0^1 \left| \int_{\T^3} (\nabla \overline{\psi}_{\kappa_q, i} \cdot \nabla \theta_{\kappa_q, j}) ({\theta}_{\kappa_q, i} \overline{\psi}_{\kappa_q, j}) \, dx \right| \, dt}_{= B_2} \\
&\quad + \underbrace{ \kappa_q \int_0^1 \left| \int_{\T^3} (\nabla \theta_{\kappa_q, i} \cdot \nabla \overline{\psi}_{\kappa_q, j}) (\overline{\psi}_{\kappa_q, i} {\theta}_{\kappa_q, j}) \, dx \right| \, dt}_{= B_3} 
+ \underbrace{ \kappa_q \int_0^1 \left| \int_{\T^3} (\nabla \theta_{\kappa_q, i} \cdot \nabla \theta_{\kappa_q, j}) ( \overline{\psi}_{\kappa_q, i} \overline{\psi}_{\kappa_q, j}) \, dx \right| \, dt}_{= B_4} \,.
\end{align*}

Now, we estimate each one of the quantities $B_1$, $B_2$, $B_3$ and $B_4$.
\begin{align} \label{eq:EstimateOfA_1}
B_1 &=  \kappa_q \int_0^1 \left| \int_{\T^3} (\nabla \overline{\psi}_{\kappa_q, i} \cdot \nabla \overline{\psi}_{\kappa_q, j}) ({\theta}_{\kappa_q, i} {\theta}_{\kappa_q, j}) \, dx \right| \, dt \notag 
\\
&\leq  \kappa_q \int_0^1 \| \nabla \overline{\psi}_{\kappa_q, i} (t, \cdot) \|_{L^{\infty}(\T)} \| \nabla \overline{\psi}_{\kappa_q, j} (t, \cdot) \|_{L^{\infty}(\T)} \| \theta_{\kappa_q, i}(t, \cdot) \|_{L^{\infty}(\T^3)} \| \theta_{\kappa_q, j}(t, \cdot) \|_{L^{\infty}(\T^3)} \, dt 
\\
&\lesssim \kappa_q \| \nabla \overline{\psi}_{\kappa_q, i} (0, \cdot) \|_{L^{\infty}(\T)} \| \nabla \overline{\psi}_{\kappa_q, j} (0, \cdot) \|_{L^{\infty}(\T)} \overset{\eqref{item:AboutTheC1NormOfTheSecondCollectionsOfCutoffs}}{\lesssim} \kappa_q a_0^{-2} a_q^{- 2 \gamma} \lesssimlarge a_q^{\sfrac{1}{2}} \,, \notag
\end{align} 
\begin{align} \label{eq:EstimateOfA_2}
B_2 &=  \kappa_q \int_0^1 \left| \int_{\T^3} (\nabla \overline{\psi}_{\kappa_q, i} \cdot \nabla \theta_{\kappa_q, j}) ({\theta}_{\kappa_q, i} \overline{\psi}_{\kappa_q, j}) \, dx \right| \, dt \notag
\\
&\leq \left(  \kappa_q \int_0^1 \int_{\T^3} |\nabla \theta_{\kappa_q, j}|^2 \, dx dt \right)^{\sfrac{1}{2}} \left(  \kappa_q \int_0^1 \int_{\T^3} |\nabla \overline{\psi}_{\kappa_q, i}|^2 \, dx dt \right)^{\sfrac{1}{2}} \| {\theta}_{\kappa_q, i} \|_{L^{\infty}(\T^3)} \| \overline{\psi}_{\kappa_q, j} \|_{L^{\infty}(\T)} \\
&\lesssim \kappa_q^{\sfrac{1}{2}} \| \theta_{\initial, j} \|_{L^2(\T^2)} \| \nabla \overline{\psi}_{\kappa_q, i}(0, \cdot) \|_{L^{\infty}(\T)} \overset{\eqref{item:AboutTheC1NormOfTheSecondCollectionsOfCutoffs}}{\lesssim} \kappa_q^{\sfrac{1}{2}} a_0^{-1} a_q^{- \sfrac{\gamma}{2}} \lesssimlarge a_q^{\sfrac{1}{2}} \,,  \notag
\end{align}
\begin{align} \label{eq:EstimateOfA_4}
B_4 &=  \kappa_q \int_0^1 \left| \int_{\T^3} (\nabla \theta_{\kappa_q, i} \cdot \nabla \theta_{\kappa_q, j}) (\overline{\psi}_{\kappa_q, i} \overline{\psi}_{\kappa_q, j}) \, dx \right| \, dt \notag 
\\
&\leq \left(  \kappa_q \int_0^1 \int_{\T^3} |\nabla \theta_{\kappa_q, i}|^2 \, dx dt \right)^{\sfrac{1}{2}} \left( \kappa_q \int_0^1 \int_{\T^3} |\nabla \theta_{\kappa_q, j}|^2 \, dx dt \right)^{\sfrac{1}{2}} \| \overline{\psi}_{\kappa_q, i} \overline{\psi}_{\kappa_q, j} \|_{L^{\infty}((0,1) \times \T)}
\\
&\leq \| \theta_{\initial, i}(0, \cdot) \|_{L^2(\T^3)} \| \theta_{\initial, j}(0, \cdot) \|_{L^2(\T^3)} \| \overline{\psi}_{\kappa_q, i} \overline{\psi}_{\kappa_q, j} \|_{L^{\infty}((0,1) \times \T)} \overset{\eqref{eq:TwoNonAdjacentsCutoffsRemainFarAppartSomehow}}{\lesssim} a_q^{\gamma} \kappa_q a_q^{- 2 \gamma} \leq a_q \,. \notag 
\end{align}
The upper bound on $B_3$ is the same of $B_2$,
hence, due to \eqref{eq:EstimateOfA_1}, \eqref{eq:EstimateOfA_2} and \eqref{eq:EstimateOfA_4} we find that
\begin{equation}\label{eq:BoundOnTheQuatitityOverlineII-ij}
 \left\| \overline{II}_{(i,j)} \right\|_{L^1((0,1))} \leq B_1 + B_2 + B_3 + B_4 \lesssimlarge a_q^{\sfrac{1}{2}}.
\end{equation}
Now, in order to use this estimate to bound the terms of $II$ in $L^1((0,1))$, we will estimate $\| {II}_{(i,j)} - \overline{II}_{(i,j)} \|_{L^1((0,1))}$.
With this goal in mind, we first observe that
\begin{equation}\label{eq:EstimateOfTheAnomalousDissipationForThetaBar}
 \left(  \kappa_{q} \int_0^1 \int_{\T^3} \left| \nabla \overline{\theta}_{\kappa_q, j} \right|^2 \, dx \, dt \right)^{\sfrac{1}{2}} \overset{\eqref{eq:Useful-Stability-Between-Theta-And-Approximation}}{\lesssimlarge} a_q^{\sfrac{1}{4}} + \left(  \kappa_{q} \int_0^1 \int_{\T^3} \left| \nabla {\theta}_{\kappa_q, j} \right|^2 \, dx \, dt \right)^{\sfrac{1}{2}} \lesssimlarge a_q^{\sfrac{1}{4}} + a_q^{\sfrac{\gamma}{2}} \lesssimlarge a_q^{\sfrac{\gamma}{2}}. 
\end{equation}
Thus,
\begin{align*}
\| {II}_{(i,j)} & - \overline{II}_{(i,j)} \|_{L^1((0,1))} =  \kappa_q \int_0^1 \left| \int_{\T^3} \nabla {\theta}_{\kappa_q, i} \cdot \nabla {\theta}_{\kappa_q, j} - \nabla \overline{\theta}_{\kappa_q, i} \cdot \nabla \overline{\theta}_{\kappa_q, j} \, dx \right| \, dt \\
&\leq  \kappa_q \int_0^1 \int_{\T^3} \left| \nabla {\theta}_{\kappa_q, i} \cdot \nabla ({\theta}_{\kappa_q, j} - \overline{\theta}_{\kappa_q, j}) \right| \, dx \, dt
+  \kappa_q \int_0^1 \int_{\T^3} \left| \nabla ({\theta}_{\kappa_q, i} - \overline{\theta}_{\kappa_q, i}) \cdot \nabla \overline{\theta}_{\kappa_q, j} \right| \, dx \, dt
 \\
&\leq \max_{i,j} \left(  \kappa_{q} \int_0^1 \int_{\T^3} \left| \nabla {\theta}_{\kappa_q, i} \right|^2 + \left| \nabla {\overline \theta}_{\kappa_q, i} \right|^2 \, dx \, dt \right)^{\sfrac{1}{2}} \left(  \kappa_{q} \int_0^1 \int_{\T^3} \left| \nabla ({\theta}_{\kappa_q, j} - \overline{\theta}_{\kappa_q, j}) \right|^2 \, dx \, dt \right)^{\sfrac{1}{2}}  
\\
&\overset{\eqref{eq:Useful-Stability-Between-Theta-And-Approximation}, \eqref{eq:EstimateOfTheAnomalousDissipationForThetaBar}}{\lesssimlarge} a_q^{\sfrac{\gamma}{2}} a_q^{\sfrac{1}{4}} + a_q^{\sfrac{\gamma}{2}} a_q^{\sfrac{1}{4}} \lesssimlarge a_q^{\sfrac{1}{4} + \sfrac{\gamma}{2}}.
\end{align*}
Hence
\[
 \left\| {II}_{(i,j)} - \overline{II}_{(i,j)} \right\|_{L^1((0,1))} \lesssimlarge a_q^{\sfrac{1}{4}}
\]
and thus, combined with \eqref{eq:BoundOnTheQuatitityOverlineII-ij}, we find
\begin{equation}\label{eq:BoundOnL1NormOfII(i,j)}
 \left\| {II}_{(i,j)} \right\|_{L^1((0,1))} \lesssimlarge a_q^{\sfrac{1}{4}}.
\end{equation}
For this last inequality, we deduce that
\begin{equation}
 \left\| II \right\|_{L^1((0,1))} \leq \sum_{|i - j| > 1} \left\| {II}_{(i,j)} \right\|_{L^1((0,1))} \overset{\eqref{eq:BoundOnL1NormOfII(i,j)}}{\lesssimlarge} N_q^2 a_q^{\sfrac{1}{4}} \overset{\eqref{item:BoundOnTheNumberOfCutoffs}}{\lesssimlarge} a_q^{\sfrac{1}{4} - 2\gamma} \lesssimlarge a_q^{\sfrac{1}{8}}.
\end{equation}
This ends the proof of Lemma~\ref{lemma:EstimateOfQuantitiy-II}.
\end{proof}

\begin{proof}[Proof of Lemma~\ref{lemma:EstimateOfQuantitiy-I}]
We estimate the terms of $I$ given by
\[
 I (t ) = \sum_{i=1}^{N_q} I_{i}(t) = \sum_{i=1}^{N_q} 2 \kappa_q \int_{\T^3} \nabla \theta_{\kappa_q, i}(t,x) \cdot \nabla \theta_{\kappa_q, i+1}(t,x) \, dx
\]
in $L^1((0,1))$.
Then
\begin{align*}
 \left\| I_i \right\|_{L^1((0,1))} &= \underbrace{ \| I_i(\cdot ) \mathbbm{1}_{[0, t_{q,i}]}(\cdot ) \|_{L^1}}_{= \tilde{B}_1}
 + \underbrace{ \| I_i  (\cdot ) \mathbbm{1}_{[t_{q,i}, t_{q,i+1} + \tilde{t}_q]}(\cdot ) \|_{L^1}}_{= \tilde{B}_2} + 
 + \underbrace{ \| I_i  (\cdot ) \mathbbm{1}_{[t_{q,i+1} + \tilde{t}_q,1]}(\cdot ) \|_{L^1}}_{= \tilde{B}_3}
\end{align*}
Now we estimate the quantities $\tilde{B}_1$ and $\tilde{B}_3$ which we expect to be very small compared to $a_q^{\gamma}$. With that goal in mind, we first observe that
\begin{equation}\label{eq:AboutDissipationBeforeCriticalTime}
\begin{split}
 2 \kappa_q \int_0^{t_{q,i}} \int_{\T^3} |\nabla \theta_{\kappa_q, i}|^2 \, dx dt &= \| \theta_{\initial, i} \|_{L^2(\T^3)}^2 - \| \theta_{\kappa_q, i}(t_{q,i}, \cdot) \|_{L^2(\T^3)}^2 \\
   &= \left( \| \theta_{\initial, i} \|_{L^2(\T^3)} + \| \theta_{\kappa_q, i}(t_{q,i}, \cdot) \|_{L^2(\T^3)} \right) \left( \| \theta_{\initial, i} \|_{L^2(\T^3)} - \| \theta_{\kappa_q, i}(t_{q,i}, \cdot) \|_{L^2(\T^3)} \right)\\
   &= \left( \| \theta_{\initial, i} \|_{L^2(\T^3)} + \| \theta_{\kappa_q, i}(t_{q,i}, \cdot) \|_{L^2(\T^3)} \right) \left( \| \theta_{0, i}(t_{q,i}, \cdot) \|_{L^2(\T^3)} - \| \theta_{\kappa_q, i}(t_{q,i}, \cdot) \|_{L^2(\T^3)} \right)\\
     &\leq \left( \| \theta_{\initial, i} \|_{L^2(\T^3)} + \| \theta_{\kappa_q, i}(t_{q,i}, \cdot) \|_{L^2(\T^3)} \right) \| (\theta_{0, i} - \theta_{\kappa_q, i})(t_{q,i}, \cdot) \|_{L^2(\T^3)} \overset{\eqref{stability:theta_kq-theta}}{\lesssim} a_q^{\sfrac{\gamma}{2}} a_q^{\frac{\gamma + \eps}{2}}. \\
\end{split}
\end{equation}
It then follows that
\begin{align*}
 \tilde{B}_1 &= 2 \kappa_q \int_0^{t_{q,i}} \left| \int_{\T^3} \nabla \theta_{\kappa_q, i}(t,x) \cdot \nabla \theta_{\kappa_q, i+1}(t,x) \, dx \right| \, dt \\
  &\leq  \left( 2 \kappa_q \int_0^{t_{q,i}} \int_{\T^3} |\nabla \theta_{\kappa_q, i}|^2 \, dx dt \right)^{\sfrac{1}{2}} \left( 2 \kappa_q \int_0^{t_{q,i}} \int_{\T^3} |\nabla \theta_{\kappa_q, i+1}|^2 \, dx dt \right)^{\sfrac{1}{2}}\\
  &\overset{\eqref{eq:AboutDissipationBeforeCriticalTime}}{\lesssim} a_q^{\sfrac{\gamma}{4}} a_q^{\frac{\gamma + \eps}{4}} a_q^{\sfrac{\gamma}{2}} = a_q^{\gamma + \sfrac{\eps}{4}};\\
  \tilde{B}_3 &= 2 \kappa_q \int_{t_{q,i+1} + \tilde{t}_q}^1 \left| \int_{\T^3} \nabla \theta_{\kappa_q, i}(t,x) \cdot \nabla \theta_{\kappa_q, i+1}(t,x) \, dx \right| \, dt \\
  &\leq  \left( 2 \kappa_q \int_{t_{q,i+1} + \tilde{t}_q}^1 \int_{\T^3} |\nabla \theta_{\kappa_q, i}|^2 \, dx dt \right)^{\sfrac{1}{2}} \left( 2 \kappa_q \int_{t_{q,i+1} + \tilde{t}_q}^1 \int_{\T^3} |\nabla \theta_{\kappa_q, i+1}|^2 \, dx dt \right)^{\sfrac{1}{2}}
  \\
  &\leq  \| \theta_{\initial, i} \|_{L^2(\T^3)} \| \theta_{\kappa_q, i+1} (t_{q,i+1} + \tilde{t}_q, \cdot) \|_{L^2(\T^3)} \overset{\eqref{eq:AutonomousNormalGoalOfStep3}}{\lesssim} a_0^{\frac{\eps \delta}{2}} a_q^{\gamma}.
\end{align*}
To estimate $\tilde{B}_2$, we proceed by a sequence of approximations. We will reuse the functions $\tilde{\theta}_{\kappa_q, j}$,  $f_{\kappa_q, j}$ and $\psi_{\kappa_q, j}$ introduced in Step 3 of the proof (see \eqref{eq:FirstApproxPDEFirstNotNS}, \eqref{eq:2DHeatEquation} and \eqref{eq:1DAdvectionDiffusionEquationInStep3}). We note that
\begin{align*}
\tilde{B}_2 &= 2 \kappa_q \int_{t_{q,i}}^{t_{q,i+1} + \tilde{t}_q} \left| \int_{\T^3} \nabla \theta_{\kappa_q, i} \cdot \nabla \theta_{\kappa_q, i+1} \, dx \right| \, dt 
\leq 2 \kappa_q \int_{t_{q,i}}^{t_{q,i+1} + \tilde{t}_q} \left| \int_{\T^3} \nabla (\theta_{\kappa_q, i} - \tilde{\theta}_{\kappa_q, i})\cdot \nabla \theta_{\kappa_q, i+1} \, dx \right| \, dt 
\\
& \quad + 2 \kappa_q \int_{t_{q,i}}^{t_{q,i+1} + \tilde{t}_q} \left| \int_{\T^3} \nabla \tilde{\theta}_{\kappa_q, i} \cdot \nabla ( \theta_{\kappa_q, i+1} - \tilde{\theta}_{\kappa_q, i+1}) \, dx \right| \, dt   + 2 \kappa_q \int_{t_{q,i}}^{t_{q,i+1} + \tilde{t}_q} \left| \int_{\T^3} \nabla \tilde{\theta}_{\kappa_q, i} \cdot \nabla \tilde{\theta}_{\kappa_q, i+1} \, dx \right| \, dt \,.
\end{align*}
We now estimate the 3 summands. The first term is estimated observing that \eqref{e:energy-equality-global} implies
\[
 2 \kappa_q \int_{0}^{1} \int_{\T^3} |\nabla \theta_{\kappa_q, i+1}|^2 \, dx dt \leq \| \theta_{\initial, i+1} \|_{L^2(\T^3)}^2 \lesssim a_q^\gamma
\]
for all $i$ and therefore
\begin{align*}
  & 2 \kappa_q \int_{t_{q,i}}^{t_{q,i+1} + \tilde{t}_q} \left| \int_{\T^3} \nabla (\theta_{\kappa_q, i} - \tilde{\theta}_{\kappa_q, i})\cdot \nabla \theta_{\kappa_q, i+1} \, dx \right| \, dt \\
 &\leq  \left( 2 \kappa_q \int_{t_{q,i}}^{t_{q,i +1} + \tilde{t}_q} \int_{\T^3} |\nabla (\theta_{\kappa_q, i} - \tilde{\theta}_{\kappa_q, i})|^2 \, dx dt \right)^{\sfrac{1}{2}} \left( 2 \kappa_q \int_{0}^{1} \int_{\T^3} |\nabla \theta_{\kappa_q, i+1}|^2 \, dx dt \right)^{\sfrac{1}{2}} \overset{\eqref{eq:TildeApproxOfHatL2}}{\lesssimlarge} a_0^{\frac{\eps \delta}{2}} a_q^{\sfrac{\gamma}{2}} a_q^{\sfrac{\gamma}{2}} \,.
\end{align*} 
 where the last holds for $q$ large enough in order to have $ t_{q, i+1}  + \tilde{t}_q \leq  t_{q, i} + \sfrac{\overline{t}_q}{2}$ (see Section \ref{subsec:ParmeterDefinitions} and \eqref{d:time_t_j}).
Similarly, we can estimate the second term.
Finally, recalling that $\tilde{\theta}_{\kappa_q, i}$, we have $\tilde{\theta}_{\kappa_q, i} = f_{\kappa_q, i} \psi_{\kappa_q, i}$, we estimate the last summand as follows 
\begin{align*}
 2 \kappa_q \int_{t_{q,i}}^{t_{q,i+1} + \tilde{t}_q} \left| \int_{\T^3} \nabla \tilde{\theta}_{\kappa_q, i} \cdot \nabla \tilde{\theta}_{\kappa_q, i+1} \, dx \right| \, dt 
& \leq \underbrace{2 \kappa_q \int_{t_{q,i}}^{t_{q,i+1} + \tilde{t}_q} \int_{\T^3} |\nabla f_{\kappa_q, i}| |\nabla f_{\kappa_q, i+1}| |\psi_{\kappa_q, i}| |\psi_{\kappa_q, i+1}| \, dx \, dt}_{= \tilde{B}_{2,1}}\\
&+ \underbrace{2 \kappa_q \int_{t_{q,i}}^{t_{q,i+1} + \tilde{t}_q} \int_{\T^3} |\nabla f_{\kappa_q, i}| |f_{\kappa_q, i+1}| |\psi_{\kappa_q, i}| |\nabla \psi_{\kappa_q, i+1}| \, dx \, dt}_{= \tilde{B}_{2,2}} \\
&+\underbrace{2 \kappa_q \int_{t_{q,i}}^{t_{q,i+1} + \tilde{t}_q} \int_{\T^3} |f_{\kappa_q, i}| |\nabla f_{\kappa_q, i+1}| |\nabla \psi_{\kappa_q, i}| |\psi_{\kappa_q, i+1}| \, dx \, dt}_{= \tilde{B}_{2,3}} \\
&+\underbrace{2 \kappa_q \int_{t_{q,i}}^{t_{q,i+1} + \tilde{t}_q} \int_{\T^3} |f_{\kappa_q, i}| |f_{\kappa_q, i+1}| |\nabla \psi_{\kappa_q, i}| |\nabla \psi_{\kappa_q, i+1}| \, dx \, dt}_{= \tilde{B}_{2,4}}.
\end{align*}
Before estimating each one of the new quantities defined above, we make the following observation. Let $S_{q,i}(t)$ be the time-dependent interval defined by $S_{q,i}(t) = (z_i + t - a_0 a_q^{\gamma}, z_i + t + \sfrac{a_q^{\gamma}}{8} + a_0 a_q^{\gamma})$. By Lemma~\ref{lemma:TailEstimates1DAdvectionDiffusion}, we find that for all $t \in [0,1]$
\[
 \| \psi_{\kappa_q, i}(t, \cdot) \|_{L^{\infty}(S_{q,i}(t)^c)} \lesssim \exp\left( \frac{(a_0 a_q^{\gamma})^2}{{4} \kappa_q t} \right) \leq \frac{{4} \kappa_q t}{a_0^2 a_q^{2 \gamma}} \lesssim \frac{a_q}{a_0^2} \lesssimlarge a_0 a_q^{\gamma}.
\]
Notice that by \eqref{eq:AboutAdjacentSupports} $|S_{q,i}(t) \cap S_{q, i + 1}(t)| \leq 3 a_0 a_q^{\gamma}$ for all $t \in [0,1]$.
Hence for all $t \in [0,1]$
\begin{equation}\label{eq:AboutTwoAdjacentPsis}
\begin{split}
 \int_{\T} |\psi_{\kappa_q, i}| |\psi_{\kappa_q, i+1}|(t,z) \, dz &\leq \int_{S_{q,i}(t) \cap S_{q, i+1}(t)} |\psi_{\kappa_q, i}| |\psi_{\kappa_q, i+1}|(t,z) \, dz 
 \\
  & \quad + \int_{\T \setminus (S_{q,i}(t) \cap S_{q, i+1}(t))} |\psi_{\kappa_q, i}| |\psi_{\kappa_q, i+1}|(t,z) \, dz \lesssimlarge a_0 a_q^{\gamma} + a_0 a_q^{\gamma} \,.
\end{split}
\end{equation}
First, we estimate 
\begin{align*}
    \tilde{B}_{2,1} &= 2 \kappa_q \int_{t_{q,i}}^{t_{q,i+1} + \tilde{t}_q} \int_{\T^3} |\nabla f_{\kappa_q, i}| |\nabla f_{\kappa_q, i+1}| |\psi_{\kappa_q, i}| |\psi_{\kappa_q, i+1}| \, dx \, dt \\
 &\overset{\eqref{eq:LInftyNormOfTheGradientOfTheMollifiedChessFunction}}{\lesssim} \kappa_q a_0^{-6 \eps \delta} a_q^{-2} \int_{t_{q,i}}^{t_{q,i+1} + \tilde{t}_q} \int_{\T^3} |\psi_{\kappa_q, i}| |\psi_{\kappa_q, i+1}| \, dx \, dt \overset{\eqref{eq:AboutTwoAdjacentPsis}}{\lesssimlarge} \kappa_q a_0^{-6 \eps \delta} a_q^{-2} \tilde{t}_q a_0 a_q^{\gamma} \lesssimlarge a_0^{1 - 7 \eps \delta} a_q^{\gamma}\,.
\end{align*}
Finally, using that $\| \nabla f_{\kappa_q, i} \|_{L^\infty(\T^3)} \overset{\eqref{eq:LInftyNormOfTheGradientOfTheMollifiedChessFunction}}{\lesssim} a_0^{-3 \eps \delta} a_q^{-1}$ and $\| \nabla \psi_{\kappa_q, i} \|_{L^\infty(\T)} \lesssim a_0^{-1} a_q^{- \gamma}$ we have
\begin{align*} 
\tilde{B}_{2,2}  + \tilde{B}_{2,3} + \tilde{B}_{2,4} \lesssim \kappa_q a_0^{- (1 + 3 \eps \delta)} a_q^{- (1 + \gamma)} + \kappa_q a_0^{- 2} a_q^{- 2 \gamma}\overset{\eqref{d:k_q} }{\lesssimlarge} a_q^{1/2} \,.
\end{align*}

Hence,
\[
 \| I_i \|_{L^1((0,1))} \leq \tilde{B}_1 + \tilde{B}_2 + \tilde{B}_3 \lesssimlarge a_0^{\frac{\eps \delta}{2}} a_q^{\gamma}
\]
and therefore 
\[
 \| I \|_{L^1((0,1))} \leq \sum_{i = 1}^{N_q} \| I_i \|_{L^1((0,1))} \overset{\eqref{item:BoundOnTheNumberOfCutoffs}}{\lesssimlarge} a_0^{\frac{\eps \delta}{2}}\,.
\] 
\end{proof} 

\begin{proof}[Proof of Lemma~\ref{lemma:EstimatesOfTheQuantities-E2-E3}]
By \eqref{eq:AboutDissipationBeforeCriticalTime} in the proof of Lemma~\ref{lemma:EstimateOfQuantitiy-I}
\begin{align*}
 \| S_{\kappa_q}^{(2)} \|_{L^1((0,1))} &\leq  \kappa_q \sum_{i = 1}^{N_q} \int_0^{t_{q,i}} \int_{\T^3} |\nabla \theta_{\kappa_q, i}|^2 \, dx \, dt \overset{\eqref{eq:AboutDissipationBeforeCriticalTime}}{\lesssim} N_q a_q^{\gamma + \sfrac{\eps}{2}} \lesssim a_q^{\sfrac{\eps}{2}}.
\end{align*}
By \eqref{eq:AutonomousNormalGoalOfStep3},
\begin{align*}
\| S_{\kappa_q}^{(3)} \|_{L^1((0,1))} &\leq  \kappa_q \sum_{i = 1}^{N_q} \int_{t_{q,i} + \tilde{t}_q}^1 \int_{\T^3} |\nabla \theta_{\kappa_q, i}(t,x)|^2 \, dx \, dt \\
&\leq \sum_{i = 1}^{N_q} \| \theta_{\kappa_q,i}(t_{q,i} + \tilde{t}_q, \cdot) \|_{L^2(\T^3)}^2 \overset{\eqref{eq:AutonomousNormalGoalOfStep3}}{\lesssim} N_q a_0^{\eps \delta} a_q^{\gamma} \lesssim a_0^{\eps \delta} \,.
\end{align*}
\end{proof}

\begin{proof}[Proof of Lemma~\ref{lemma:EstimateOfQuatities-E6-E7}]
Recall that $\mathcal{J}_{q,i} = [t_{q,i}, t_{q,i} + \tilde{t}_q]$ and note that due to \eqref{e:energy-equality-global}, we have
\[
 2 \kappa_q  \int_{\mathcal{J}_{q,i}} \int_{\T^3} |\nabla {\theta}_{\kappa_q, i}|^2 \, dx \, dt \leq \|\theta_{\kappa_q , i} (0, \cdot ) \|_{L^2 (\T^3)}^2  \lesssim a_q^{\gamma}.
\]
Hence
\begin{align*}
\| S_{\kappa_q}^{(4)} \|_{L^1((0,1))} &\leq  \kappa_q \sum_{i = 1}^{N_q} \int_{\mathcal{J}_{q,i}} \left| \int_{\T^3} \nabla {\theta}_{\kappa_q, i}(t,x) \cdot \nabla ({\theta}_{\kappa_q, i} - \tilde{\theta}_{\kappa_q, i})(t,x)\, dx \right| \, dt \\
&\leq \sum_{i = 1}^{N_q} \left(  \kappa_q  \int_{\mathcal{J}_{q,i}} \int_{\T^3} |\nabla {\theta}_{\kappa_q, i}|^2 \, dx \, dt \right)^{\sfrac{1}{2}} \left(  \kappa_q \int_{\mathcal{J}_{q,i}} \int_{\T^3} |\nabla ({\theta}_{\kappa_q, i} - \tilde{\theta}_{\kappa_q, i})|^2 \, dx \, dt \right)^{\sfrac{1}{2}} \\
&\overset{\eqref{eq:TildeApproxOfHatL2}}{\lesssim}  N_q a_q^{\sfrac{\gamma}{2}} a_q^{\sfrac{\gamma}{2}} a_0^{\varepsilon \delta /2}  \overset{\eqref{item:BoundOnTheNumberOfCutoffs}}{\lesssim} a_0^{\sfrac{\eps \delta}{2}}.
\end{align*}
Similarly $\| S_{\kappa_q}^{(5)} \|_{L^1((0,1))} \lesssim a_0^{\sfrac{\eps \delta}{2}}$.
\end{proof}

\section{Proof of Theorem~\ref{t_Onsager}}\label{sec:4D-NS-Proof}

The aim of this section is to prove Theorem~\ref{t_Onsager}. We start by proving the key lemma which says that the last component of $v_\nu$ solution to the 4d Navier--Stokes equations with body force $F_\nu$ exhibits anomalous dissipation with further properties.  In the last subsection, we use this lemma to prove Theorem~\ref{t_Onsager}.

\subsection{Strategy of the proof}
The strategy of the proof is to use the $(3+\frac{1}{2})-d$ Navier--Stokes equations with a body force. More precisely, we consider the $4d$ Navier--Stokes equations with body force acting on the first three components with variable $(t, x, y, z, w) \in (0,1) \times \T^4$, where all the system is independent on the variable $w$, which reduces to the coupled system 
    \[
    \begin{cases}
        \partial_t u_\nu + u_\nu \cdot \nabla u_\nu + \nabla p_\nu = \nu \Delta u_\nu + F_\nu
            \\
        \partial_t \theta_\nu + u_\nu \cdot \nabla \theta_\nu = \nu \Delta \theta_\nu \,,
        \\
        u_\nu (0, \cdot ) = u_{\nu, \initial} (\cdot ), \quad \theta_{\nu} (0, \cdot ) = \theta_{\nu, \initial } (\cdot ) \,,
        \\
        \diver u_\nu =0 \,,
    \end{cases} 
    \]
where the unknown are $u_\nu : [0,1] \times \T^3 \to \R^3$, $\theta_\nu: [0,1] \times \T^3 \to \R$ and $p_\nu : [0,1] \times \T^3 \to \R $ and $F_\nu : [0,1] \times \T^3 \to \R^3$, $u_{\nu, \initial } : \T^3 \to \R^3$ and $\theta_{\nu, \initial}: \T^3 \to \R$ are given. The proof will rely on the objects defined in Section \ref{subsec:ConstructionsForNavierStokes4D}.

\subsection{Anomalous dissipation lemma}\label{subsec:LemmaAnomalousDissipationSpecialSequence}
The goal of this subsection is to prove the following lemma.
\begin{lemma}\label{lemma:LemmaAnomalousDissipationSpecialSequence}
 Let $\beta > 0$. Consider the collection of autonomous divergence-free velocity fields $\{ u_{\nu} \}_{\nu \geq 0}$ defined in Subsection~\ref{subsec:ConstructionsForNavierStokes4D} and the initial datum $\theta_{\initial}$ defined in Subsection~\ref{subsec:ConstructionInitialDatum}. Under the assumption that $a_0$ is selected small enough, everything that follows holds: The sequence of unique solutions $\{ \theta_{\nu} \}_{\nu \geq 0}$ of the advection-diffusion equations
 \begin{equation}\label{eq:ADV-DIFF-FOR-NS}
 \left\{
 \begin{array}{ll}
 \partial_t {\theta}_{\nu} + u_{\nu} \cdot \nabla {\theta}_{\nu} = \nu \Delta {\theta}_{\nu}; \\
 {\theta}_{\nu} (0,x,y,z) = {\theta}_{\initial}(x,y,z); \\
 \end{array}
 \right.
\end{equation}
exhibit anomalous dissipation, i.e.
\begin{equation}
\limsup_{\nu \to 0} \nu \int_0^1 \int_{\T^3} |\nabla \theta_{\nu}|^2 \, dx dt > 0.
\end{equation}
In addition, there exists a measure $\mu \in \mathcal{M} ((0,1) \times \T^3)$ such that  up to not relabelled subsequences
\begin{equation}
  \nu |\nabla \theta_{\nu}|^2 \weak \mu
\end{equation}
in the sense of measures on $(0,1) \times \T^3$ and the measure $\mu$  satisfies $\| \mu \|_{TV} \geq 1/4$ and
\begin{equation}\label{eq:NS-Lemma-H-Minus-1-Norm}
  \left \| \frac{1}{2} \partial_t |\theta_0|^2  + \diver (u \frac{|\theta_0|^2}{2}) + \mu \right \|_{H^{-1}((0,1) \times \T^3)} \leq \beta \,,
\end{equation}
where $\theta_0$ is the unique solution of \eqref{eq:ADV-DIFF-FOR-NS} with $\nu = 0$ and initial datum $\theta_{\initial}$. Moreover, $\mu_T = \pi_{\#} \mu$, where $\pi$ is the projection in time map $\pi : (0,1) \times \T^3 \to (0,1)$, has a non-trivial absolutely continuous part w.r.t. the Lebesgue measure $\mathcal{L}^1$ and its singular part is such that $\| \mu_{T, \text{sing}} \|_{TV} \leq \beta$.
Furthermore, $\theta_{\nu} \weak \theta_0$ weakly$^{\ast}$ in $L^{\infty}((0,1) \times \T^3)$ and (up to not relabelled subsequences) we have $\| \theta_{\nu} - \theta_0 \|_{L^{\infty}((0,1) \times \T^3)} < \beta$ and
\[
 e(t) = \int_{\T^3} | \theta_0 (t, x )|^2 dx
\]
is smooth in  $[0,1]$  and it is such that $e(1) < e(0)$.
\end{lemma}

The proof of this lemma is very similar to that of Theorem~\ref{t_anomalous_autonomous}. For the convenience of the reader, rather than just pointing out the differences between the two proofs, we give a less detailed, yet complete proof of this lemma. We refer to the proof of Theorem~\ref{t_anomalous_autonomous} whenever it is possible in order to shorten the proof.

\begin{proof}
Fix some arbitrary  $q \in m \N$. As in the proof of Theorem~\ref{t_anomalous_autonomous}, the goal is to prove that 
\begin{equation}\label{eq:LTwoNormAlmostZeroForTheSpecial3DCase}
\| {\theta}_{\nu_q}(1, \cdot ) \|_{L^{2}(\T^3)}^2 \lesssim a_0^{\eps \delta}.
\end{equation}
By selecting $a_0$ sufficiently small combined with the energy balance, we deduce that
\begin{equation*}
\limsup_{q \to \infty} \int_0^1 \| \nabla \theta_{\nu_q}(t) \|_{L^2(\T^3)}^2 \, dt > 0.
\end{equation*}
Let $\{ \chi_j \}_{j = 1}^{N_q} \subset C^{\infty}(\T)$ be the sequence of smooth functions and $\{ z_i \}_{i = 1}^{N_q} \subset \T$ the sequence of points defined in Subsection~\ref{subsec:PartitionOfUnityInZ}. The claims stated in this section will be proved immediately after this proof. Throughout this entire proof, we assume without loss of generality that $q$ is sufficiently large so that $\tilde{t}_q \leq \overline{t}_q$ thanks to \eqref{item:EpsDeltaThree}.
\\
\textbf{Step 1: Decomposition}
For each $j = 1, \ldots, N_q$, define $\theta_{\initial, j} \colon \T^3 \to \R$ as $\theta_{\initial, j}(x,y,z)  =  \theta_{\initial}(x,y,z) \chi_j(z)$. Then we define $\theta_{\nu_q, j}$ as the solution to the advection-diffusion equation
\begin{equation}
 \left\{
 \begin{array}{ll}
 \partial_t {\theta}_{\nu_q, j} + u_{\nu_q} \cdot \nabla {\theta}_{\nu_q, j} = \nu_q \Delta {\theta}_{\nu_q , j}; \\
 {\theta}_{\nu_q , j} (0,x,y,z) = {\theta}_{\initial , j}(x,y,z). \\
 \end{array}
 \right.
\end{equation}
Accordingly, we also specify that ${\theta}_{0,j } \colon [0,1/2] \times \T^3 \to \R$ defined as ${\theta}_{0,j}(t,x,y,z) = \theta_{\stat }(x,y,z) \varphi_{\initial} (z-t) \chi_j(z - t)$  is the unique solution  to the advection equation (see Lemma \ref{lemma:uniqueness})
\begin{equation}
 \left\{
 \begin{array}{ll}
 \partial_t {\theta}_{0,j} + u \cdot \nabla {\theta}_{0,j} = 0; \\
 {\theta}_{0,j} (0,x,y,z) = {\theta}_{\initial, j}(x,y,z). \\
 \end{array}
 \right.
\end{equation}
\textbf{Step 2: $L^2$-stability between $\theta_{\nu_q, j}$ and $\theta_{0, j}$ until time $t_{q,j}$.}
The goal of this step is to prove that 
\begin{equation}\label{eq:Special3DStabilityUpToCertainTime}
 \| {\theta}_{\nu_q, j}(t , \cdot ) - {\theta}_{0,j}(t, \cdot ) \|_{L^{2}(\T^3)} \lesssim a_q^{\gamma + \eps} \qquad \text{for all }  t \in [0,t_{q,j}] \quad \text{for all } j =1, 2, \ldots , N_q \,.
\end{equation}
Since $\supp (\theta_{0,j}(t, \cdot)) \subset \T^2 \times (z_i + t, z_i + t + \sfrac{a_q^{\gamma}}{8})$ and $u_{\nu_q}(x,y,z) = u(x,y,z)$ for all $z \in (0, \sfrac{1}{2} - T_q + \overline{t}_q)$, we deduce that in the time interval $[0, t_{q,j}]$, ${\theta}_{0,j}$ solves the advection equation
 \begin{equation}
 \left\{
 \begin{array}{ll}
 \partial_t {\theta}_{0,j} + u_{\nu_q} \cdot \nabla {\theta}_{0,j} = 0 \text{ on } [0, t_{q,i}] \times \T^3; \\
 {\theta}_{0,j} (0,x,y,z) = {\theta}_{\initial, j}(x,y,z). \\
 \end{array}
 \right.
\end{equation}
Therefore, by Lemma~\ref{lemma:AdvectionDiffusionAndTransport}, we get
\[
 \| {\theta}_{\nu_q, j}(t, \cdot ) - {\theta}_{0,j}(t, \cdot ) \|_{L^{2}(\T^3)}^2 \leq \nu_q \int_0^t \| \nabla {\theta}_{0,j} (s, \cdot ) \|_{L^2(\T^3)}^2 \, ds \,,
\]
for all $t \in [0, t_{q,j}]$.
The right-hand side of the inequality above was already estimated in the proof of Theorem~\ref{t_anomalous_autonomous} (see Lemma~\ref{lemma:EstimateOnTheGradientOfTheSolutionToTheTransportEquation}). From these estimates, we deduce \eqref{eq:Special3DStabilityUpToCertainTime} in the exact same way as in the proof of Theorem~\ref{t_anomalous_autonomous}.
\\
\textbf{Step 3: Decay of $\| \theta_{\nu_q, j}(t, \cdot) \|_{L^2(\T^3)}$ in $[t_{q,j}, t_{q,j} + \sfrac{\tilde{t}_q}{2}]$.}
The goal of this step is to prove that
\begin{equation}\label{eq:3DSpecialGoalOfStep3}
 \| {\theta}_{\nu_q, j}(t_{q,j} + \sfrac{\tilde{t}_q}{2}, \cdot ) \|_{L^2(\T^3)}^2 \lesssim a_0^{\eps \delta} a_q^{\gamma}.
\end{equation}
This step is divided into 4 substeps following the same strategy as Step 3 of the proof of Theorem~\ref{t_anomalous_autonomous}. \\
\textbf{Substep 3.1: Almost periodicity of $\theta_{\nu_q, j}(t_{q,j}, \cdot, \cdot, \cdot)$.} By the exact same arguments as in Substep 3.1 in the proof of Theorem~\ref{t_anomalous_autonomous}, we have
\begin{equation}\label{eq:3DSpecialSymmetriesL2Distance}
 \| {\theta}_{\nu_q ,j}(t_{q,j}, \cdot, \cdot, \cdot) - \varphi_{\initial}( \cdot - t_{q,j})\chi_j(\cdot - t_{q, j}) (\theta_{\text{chess}})_{a_0^{\eps \delta}}((2a_q)^{-1} \cdot, (2a_{q})^{-1} \cdot) \|_{L^2(\T^3)}^2 \lesssim a_0^{\eps \delta} a_q^{\gamma}.
\end{equation}
for some $C_1 >0$ independent on $q \in m\N$.
 
\textbf{Substep 3.2: Approximation of ${\theta}_{\nu_q, j}$ on $[t_{q,j}, t_{q,j} + \sfrac{\tilde{t}_q}{2}]$.} Let $\tilde{\theta}_{\nu_q, j}$ be as in Substep 3.3 in the proof of Theorem~\ref{t_anomalous_autonomous}. This function solves 
\begin{equation}\label{eq:ThetaTildeForNavierStokes}
\left\{
\begin{array}{l}
\partial_t \tilde{\theta}_{\nu_q, j} + \partial_z \tilde{\theta}_{\nu_q ,j} = \nu_q \Delta \tilde{\theta}_{\nu_q ,j} \qquad \text{ on } [t_{q,j}, t_{q,j} + \sfrac{\tilde{t}_q}{2}] \times \T^3;
\\
\tilde{\theta}_{\nu_q , j}(t_{q, j}, x, y, z) = \varphi_{\initial}(z - t_{q,j})\chi_j (z - t) (\theta_{\text{chess}})_{a_0^{\eps \delta}}( (2 a_q)^{-1}  x,  (2 a_q)^{-1}  y).
\end{array}
\right.
\end{equation}
By the same computations as in Substep 3.2 of the proof of Theorem~\ref{t_anomalous_autonomous}, we find that
\begin{align*}
 \| {\theta}_{\nu_q, j}(t, \cdot ) - \tilde{\theta}_{\nu_q,j}(t, \cdot ) \|_{L^2(\T^3)}^2 & + 2 \nu_q \int_{t_{q,j}}^{t} \| \nabla ({\theta}_{\nu_q, j} - \tilde{\theta}_{\nu_q, j})(s, \cdot)\|_{L^2(\T^3)}^2 \, ds \notag 
 \\
 & \lesssim \frac{1}{\nu_q} \int_{t_{q,j}}^{t} \| u^{(1,2)} \cdot \nabla_{x,y} \tilde{\theta}_{\nu_q,j}(s, \cdot) \|_{L^2(\T^3)}^2 \, ds + \| {\theta}_{\nu_q, j}(t_{q,j}, \cdot ) - \tilde{\theta}_{\nu_q,j}(t_{q,j}, \cdot ) \|_{L^2(\T^3)}^2 \,.
\end{align*}
By construction of the function $u_{\nu_q}$, in each point of $\T^3$, we either have that $u_{\nu_q} = u$ and/or $u_{\nu_q} = 0$. Therefore, $\| u_{\nu_q}^{(1,2)} \cdot \nabla_{x,y} \tilde{\theta}_{\nu_q,j}(s) \|_{L^2(\T^3)} \leq \| u^{(1,2)} \cdot \nabla_{x,y} \tilde{\theta}_{\nu_q,j}(s) \|_{L^2(\T^3)}$. Thus, by the computations already done in Substep 3.2 in the proof of Theorem~\ref{t_anomalous_autonomous}, we deduce that
\begin{equation}\label{eq:L2DistanceBetweenHatAndTilde}
  \| {\theta}_{\nu_q, j}(t, \cdot) - \tilde{\theta}_{\nu_q,j}(t, \cdot) \|_{L^2(\T^3)}^2 + 2 \kappa_q \int_{t_{q,j}}^{t} \| ({\theta}_{\nu_q, j} - \tilde{\theta}_{\nu_q,j})(s, \cdot) \|_{L^2(\T^3)}^2 \, ds \lesssim a_0^{\eps \delta} a_q^{\gamma}.
\end{equation}
for all $t \in [t_{q,j}, t_{q,j} + \sfrac{\tilde{t}_q}{2} ]$. \\
\textbf{Substep 3.3: Dissipation of ${\theta}_{\nu_q, j}$ on $[t_{q,j}, t_{q,j} + \sfrac{\overline{t}_q}{2}]$.}
It follows immediately from Substep 3.3 of the proof of Theorem~\ref{t_anomalous_autonomous} that 
\begin{equation}\label{eq:3DSpecialDissipationOfTilde}
 \| \tilde{\theta}_{\nu_q, j}(t_{q,j} + \sfrac{\overline{t}_q}{2}, \cdot ) \|_{L^2(\T^3)}^2 \leq a_q^{\gamma + \eps}.
\end{equation}
and
by the same arguments as in Substep 3.4 of the proof of Theorem~\ref{t_anomalous_autonomous} (by using \eqref{eq:L2DistanceBetweenHatAndTilde} and \eqref{eq:3DSpecialDissipationOfTilde}) we get
\begin{equation*}
 \| {\theta}_{\nu_q, j}(t_{q,j} + \sfrac{\tilde{t}_q}{2}, \cdot ) \|_{L^2(\T^3)}^2 \lesssim a_0^{\eps \delta} a_q^{\gamma},
\end{equation*}
which proves \eqref{eq:3DSpecialGoalOfStep3}. 
\\
\textbf{Step 4: Rapid decay of the $L^2$-norm of $\theta_{\nu_q}$.}
As in Step 4 of the proof of Theorem~\ref{t_anomalous_autonomous}, using the exact same arguments, we can prove \eqref{eq:LTwoNormAlmostZeroForTheSpecial3DCase}. \\

\textbf{Step 5: Smoothness of the energy $e$ and $H^{-1}_{t,x}$ closeness.}
As in Step 5 of the proof of Theorem~\ref{t_anomalous_autonomous}, we observe that due to Lemma~\ref{lemma:uniqueness} and the fact that $\{ \theta_{\nu} \}_{\nu \geq 0}$ is uniformly bounded in $L^{\infty}((0,1) \times \T^3)$, we have $\theta_{\nu} \weak \theta_0$ in $L^{\infty}((0,1) \times \T^3)$. As in Step 5 of Theorem~\ref{t_anomalous_autonomous}, $e$ is smooth in {$[0,1]$}. Following the same approach as in Step 5 of the proof of Theorem~\ref{t_anomalous_autonomous} we prove that there exists a universal constant $C > 0$ such that
\begin{equation}
 \limsup_{q \in m \N, q \to \infty} \| \theta_{\nu_q} - \theta_0 \|_{L^{\infty}((0,1); L^2(\T^3))} \leq C a_0^{\eps \delta}.
\end{equation}
As a consequence, there exists $f_{\text{err}} \in L^{\infty}((0,1) \times \T^3)$ such that $\| f_{\text{err}} \|_{L^{\infty}((0,1); L^2(\T^3))}^2 \leq C a_0^{\eps \delta}$ and
\begin{equation}
 |\theta_{\nu_q}|^2 \weak |\theta_0|^2 + f_{\text{err}}
\end{equation}
 in $L^{\infty}((0,1) \times \T^3)$, up to not relabelled subsequences.
Therefore, due to Lemma~\ref{lemma:local-energy} and defining $\mu$ as the limit of  $ \nu_q |\nabla \theta_{\nu_q}|^2 \weak \mu$ up to subsequences with $q \in m\N$, the following identity
\begin{equation}
\frac{1}{2} \partial_t |\theta_0|^2 + \frac{1}{2} \partial_t f_{\text{err}} +\frac{1}{2} \diver(u|\theta_0|^2) + \frac{1}{2} \diver(u f_{\text{err}}) + \mu =0
\end{equation}
holds in the sense of distributions. From this, we deduce the $H^{-1}_{t,x}$ closeness. Indeed, we obtain
\begin{equation}
 \left\|\frac{1}{2} \partial_t |\theta_0|^2 + \frac{1}{2} \diver(u |\theta_0|^2) +\mu \right\|_{H^{-1}((0,1) \times \T^3)} \leq \|f_{\text{err}} + u f_{\text{err}} \|_{L^2((0,1) \times \T^3)}  \leq C a_0^{\eps \delta} < \beta \,,
\end{equation}
where the last holds thanks to the choice of $a$ in Subsection \ref{subsec:ParmeterDefinitions}.

\noindent \textbf{Step 6: Convergence behaviour of $- \frac{1}{2}  e_{\nu_q}^{\prime}$ for $q \in m \N$.}
In order to finish the proof of the lemma, we investigate the behaviour of
\[
 e_{\nu_q}^{\prime}(t) =  \nu_q \int_{\T^3} |\nabla \theta_{\nu_q}(t,x)|^2 \, dx
\]
as $q \to \infty$. By the global energy balance (see \eqref{e:energy-equality-global}) it is clear the up to subsequences $\{ e_{\nu}^{\prime} \}_{\nu}$ weak-$\ast$ converges to a measure $\mu$. As in the proof of Theorem~\ref{t_anomalous_autonomous}, we define $A_{\nu_q}, S_{\nu_q} \colon [0,1] \to \R$ as
\[
 A_{\nu_q}(t) =  \nu_q \sum_{j = 1}^{N_q} \int_{\T^3} |\nabla \tilde{\theta}_{\nu_q, j}(t,x)|^2 \mathbbm{1}_{[t_{q,j}, t_{q,j} +  {\tilde{t}_q}]}(t) \, dx \quad \text{and} \quad S_{\nu_q} = - \frac{1}{2}  e_{\nu_q}^{\prime}(t) - A_{\nu_q}(t).
\]
We will prove that 
\begin{equation}\label{eq:WhatWeWantToProveInStepSixOfTheAnomalousDissipationLemma}
 \limsup_{q \to \infty} \| S_{\nu_q} \|_{L^1((0,1))} \lesssim a_0^{\frac{\eps \delta}{2}} \quad \text{ and } \quad  \sup_q \| A_{\nu_q} \|_{L^{\infty}(0,1)} < \infty \quad  \| \mu_T \|_{TV} \geq 1/4 \,,
\end{equation}
where $\mu_T \in \mathcal{M}(0,1)$ is defined as $\mu_T = \pi_{\#} \mu   $.
From this we are able to conclude that up to subsequences $A_{\nu_q}$ weak-$\ast$ converges in $L^{\infty}$ to some $\mathcal{A} \in L^{\infty}((0,1))$ and $S_{\nu_q}$ weak-$\ast$ converges in $\mathcal{M}((0,1))$ to some measure $\mathcal{S} \in \mathcal{M}((0,1))$. 
The properties in \eqref{eq:WhatWeWantToProveInStepSixOfTheAnomalousDissipationLemma} imply that $\| \mathcal{S} \|_{TV} < \beta < 1/4$ provided that  $a_0$ is chosen to be small enough depending on $\beta$ and universal constants. Hence the absolutely continuous part of the measure $\mu = \mathcal{A} + \mathcal{S}$ is non-trivial. The proof of \eqref{eq:WhatWeWantToProveInStepSixOfTheAnomalousDissipationLemma} is divided into 4 substeps following the same strategy as Substep 6 in the proof of Theorem~\ref{t_anomalous_autonomous}.
\\
\textbf{Substep 6.1:}
We have
\[
- \frac{1}{2} e_{\nu_q}^{\prime} = \underbrace{ \nu_q \sum_{i = 1}^{N_q} \int_{\T^3} |\nabla \theta_{\nu_q, i}(t,x)|^2 \, dx}_{= A_{\nu_q}^{(1)}(t)} + \underbrace{ \nu_q \sum_{i,j = 1, i \neq j}^{N_q} \int_{\T^3} \nabla \theta_{\nu_q, i}(t,x) \cdot \nabla \theta_{\nu_q, j}(t,x) \, dx}_{= S_{\nu_q}^{(1)}(t)}.
\]
We will prove the following result:
 there exist universal constants $C>0$ and $Q \in \N$ such that for all $q \geq Q$ we have
 \begin{equation}\label{eq:EstimateOfQuatities-E1-Navier-Stokes}
  \| S_{\nu_q}^{(1)} \|_{L^1((0,1))} \leq C a_0^{\sfrac{\eps \delta}{2}}.
 \end{equation}
In order to prove this claim, it suffices to adapt the proofs of Lemmas~\ref{lemma:EstimateOfQuantitiy-II} and \ref{lemma:EstimateOfQuantitiy-I} to this slightly different context. To adapt the proof of Lemma~\ref{lemma:EstimateOfQuantitiy-II}, we can repeat the exact same proof, the only difference being that $\overline{\theta}_{\nu_q, j}$ defined in the proof would not solve \eqref{eq:1D-Adv-Diff-Cutoffs-Prod-With-Solution} but instead solves
 \begin{equation*}
 \left\{
 \begin{array}{ll}
 \partial_t \overline{\theta}_{\nu_q, j} + u_{\nu_q} \cdot \nabla \overline{\theta}_{\nu_q, j} = \kappa_q \Delta \overline{\theta}_{\nu_q, j} - 2 \kappa_q \nabla {\theta}_{\nu_q, j} \cdot \nabla \overline{\psi}_{\nu_q, j}; \\
 \overline{\theta}_{\nu_q, j}(0, \cdot) = {\theta}_{\nu_q, j}(0, \cdot). \\
 \end{array}
 \right.
\end{equation*}
 In order to adapt Lemma~\ref{lemma:EstimateOfQuantitiy-I} to this context, no other modifications than replacing \eqref{eq:AutonomousNormalGoalOfStep3} and \eqref{eq:TildeApproxOfHatL2} by \eqref{eq:3DSpecialGoalOfStep3}  and \eqref{eq:L2DistanceBetweenHatAndTilde} are needed.

\textbf{Substep 6.2:} Note that
\begin{align*}
 A_{\nu_q}^{(1)}(t) &=  \nu_q \sum_{i = 1}^{N_q} \int_{\T^3} |\nabla \theta_{\nu_q, i}(t,x)|^2 \, dx = \underbrace{ \nu_q \sum_{i = 1}^{N_q} \int_{\T^3} |\nabla \theta_{\nu_q, i}(t,x)|^2 \, dx \mathbbm{1}_{[t_{q,i}, t_{q,i} + \tilde{t}_q]}(t)}_{= A_{\nu_q}^{(2)}(t)} \\
 & \quad + \underbrace{ \nu_q \sum_{i = 1}^{N_q} \int_{\T^3} |\nabla \theta_{\nu_q, i}(t,x)|^2 \, dx \mathbbm{1}_{[0, t_{q,i})}(t)}_{= S_{\nu_q}^{(2)}(t)} + \underbrace{ \nu_q \sum_{i = 1}^{N_q} \int_{\T^3} |\nabla \theta_{\nu_q, i}(t,x)|^2 \, dx \mathbbm{1}_{(t_{q,i} + \tilde{t}_q, 1]}(t)}_{= S_{\nu_q}^{(3)}(t)}.
\end{align*}
It suffices to repeat the proof of Lemma~\ref{lemma:EstimatesOfTheQuantities-E2-E3} by replacing \eqref{eq:AutonomousNormalGoalOfStep3} and \eqref{eq:3DSpecialGoalOfStep3} to prove the following estimate.
 There exists a universal constant $C$ such that
 \begin{equation}\label{eq:EstimateOfQuatities-E2-E3-Navier-Stokes}
  \| S_{\nu_q}^{(2)} \|_{L^1((0,1))} + \| S_{\nu_q}^{(3)} \|_{L^1((0,1))} \leq C a_0^{\eps \delta}.
 \end{equation}

\textbf{Substep 6.3: }
For all $q \in m\N$ and $i = 1, \ldots, N_q$, define the interval $\mathcal{J}_{q,i} = [t_{q,i}, t_{q,i} + \tilde{t}_q]$. Observe that 
\begin{align*}
 A_{\nu_q}^{(3)}(t) &=  \nu_q \sum_{i = 1}^{N_q} \int_{\T^3} |\nabla {\theta}_{\nu_q, i}(t,x)|^2 \, dx \mathbbm{1}_{\mathcal{J}_{q,i}}(t) = \underbrace{ \nu_q \sum_{i = 1}^{N_q} \int_{\T^3} |\nabla \tilde{\theta}_{\nu_q, i}(t,x)|^2 \, dx \mathbbm{1}_{\mathcal{J}_{q,i}}(t)}_{= A_{\nu_q}(t)} \\
 &\quad + \underbrace{ \nu_q \sum_{i = 1}^{N_q} \int_{\T^3} \nabla {\theta}_{\nu_q, i}(t,x) \cdot \nabla ({\theta}_{\nu_q, i} - \tilde{\theta}_{\nu_q, i})(t,x)\, dx \mathbbm{1}_{\mathcal{J}_{q,i}}(t)}_{= S_{\nu_q}^{(4)}(t)} \\
 &\quad + \underbrace{ \nu_q \sum_{i = 1}^{N_q} \int_{\T^3} \nabla ({\theta}_{\nu_q, i} - \tilde{\theta}_{\nu_q, i})(t,x) \cdot \nabla \tilde{\theta}_{\nu_q, i}(t,x)\, dx \mathbbm{1}_{\mathcal{J}_{q,i}}(t)}_{= S_{\nu_q}^{(5)}(t)} \\
\end{align*}
It suffices to repeat the proof of Lemma~\ref{lemma:EstimateOfQuatities-E6-E7} and replacing \eqref{eq:TildeApproxOfHatL2} by \eqref{eq:ThetaTildeForNavierStokes} to prove the following estimate.
There exists a universal constant $C$ such that
 \begin{equation}\label{eq:EstimateOfQuatities-E6-E7-Navier-Stokes}
  \| S_{\nu_q}^{(4)} \|_{L^1((0,1))}, \| S_{\nu_q}^{(5)} \|_{L^1((0,1))} \leq C a_q^{\sfrac{\eps}{2}}.
 \end{equation}
Notice that $S_{\nu_q} = \sum_{\ell = 1}^{5} S_{\nu_q}^{(\ell)}$ and therefore in virtue of \eqref{eq:EstimateOfQuatities-E1-Navier-Stokes}, \eqref{eq:EstimateOfQuatities-E2-E3-Navier-Stokes} and \eqref{eq:EstimateOfQuatities-E6-E7-Navier-Stokes}
\begin{equation}\label{eq:L1BoundOnTheError-Navier-Stokes}
\| S_{\nu_q} \|_{L^1((0,1))} \lesssimlarge a_0^{\frac{\eps \delta}{2}}. 
\end{equation}
This proves the first inequality in \eqref{eq:WhatWeWantToProveInStepSixOfTheAnomalousDissipationLemma}.

\textbf{Substep 6.4: }  In this step we prove that  $\| \mu \|_{TV} \geq 1/4$ and $\sup_{q} \| A_{\kappa_q} \|_{L^{\infty}((0,1))} < \infty$.

Firstly, we notice that $\mu$ is the weak* limit up to subsequence (in the sense of measure) of the non-negative sequence $\{ \nu_q  |\nabla \theta_{\nu_q} |^2 \}_{q \in m \N }$ that is such that \eqref{eq:total-diss} holds, therefore $\| \mu_T \|_{TV} \geq 1/4$ holds.
We now prove the second property.
Recall that $\tilde{\theta}_{\nu_q, i}$ solves \eqref{eq:ThetaTildeForNavierStokes} and $\mathcal{J}_{q,i} = [t_{q,i}, t_{q,i} + \tilde{t}_q]$. Define the function $\overline{N}_q \colon [0,1] \to \N$ as
\[
 \overline{N}_q(t) = \# \{ j \in \{ 1, \ldots, N_q \} : \mathbbm{1}_{\mathcal{J}_{q,i}}(t) \neq 0 \}.
\]
It follows from straightforward computations that $\| \overline{N}_q(t) \|_{L^{\infty}((0,1))} \lesssim a_0^{- \eps \delta} a_q^{\frac{\gamma}{1 + \delta} - 10 \eps - \gamma}$.
For all $t \in \mathcal{J}_{q,i}$,
\[
 \| \nabla \tilde{\theta}_{\nu_q, i}(t, \cdot) \|_{L^2(\T^3)}^2 \leq \| \nabla \tilde{\theta}_{\nu_q, i}(t_{q,i}, \cdot) \|_{L^2(\T^3)}^2 \lesssim a_0^{-6 \eps \delta} a_q^{-2 + \gamma}
\]
From this we get
\begin{align*}
 |A_{\nu_q}(t)| &\leq 2 \kappa_q \sum_{i = 1}^{N_q} \int_{\T^3} |\nabla \tilde{\theta}_{\nu_q, i}|^2 \, dx \mathbbm{1}_{\mathcal{J}_{q,i}} (t) = 2 \nu_q \sum_{i = 1}^{N_q} \| \nabla \tilde{\theta}_{\nu_q, i} \|_{L^2(\T^3)}^2 \mathbbm{1}_{\mathcal{J}_{q,i}} (t) \\
 &\leq 2 \nu_q \overline{N}_q(t) a_q^{-2 + \gamma} \lesssim a_0^{-7 \eps \delta}.
\end{align*}
 This proves the last inequality in \eqref{eq:WhatWeWantToProveInStepSixOfTheAnomalousDissipationLemma}. Combining this with \eqref{eq:L1BoundOnTheError-Navier-Stokes} implies that $\sup_{q} \| - \frac{1}{2} e_{\nu_q}^{\prime} \|_{L^{1}((0,1))} < \infty$. Therefore, up to subsequences
 \[
  A_{\nu_q} \overset{*}{\rightharpoonup} \mathcal{A} \in L^{\infty}((0,1)), \quad 
  S_{\nu_q} \overset{*}{\rightharpoonup} \mathcal{S} \in \mathcal{M}([0,1]) \quad \text{and} \quad - \frac{1}{2} e_{\nu_q}^{\prime} \overset{*}{\rightharpoonup} \mu_T \in \mathcal{M}([0,1])\,,
 \]
 where $\mu_T= \pi_{\#} \mu $.
 In virtue of \eqref{eq:L1BoundOnTheError-Navier-Stokes}, we find that $\| \mu_T - \mathcal{A} \|_{TV} = \| \mathcal{S} \|_{TV} \lesssim a_0^{\frac{\eps \delta}{2}} < \beta$ where the last holds up to choosing $a_0$ sufficiently small in  Subsection \ref{subsec:ParmeterDefinitions} depending on $\beta$ and universal constants. The proof of Lemma~\ref{lemma:LemmaAnomalousDissipationSpecialSequence} is complete.
\end{proof}

\subsection{Proof of Theorem~\ref{t_Onsager}}
Let $\{ u_\nu \}_{\nu \geq 0}, \{ F_\nu \}_{\nu \geq 0}$ and $\{ v_{\initial, \nu} \}_{\nu \geq 0}$ be defined as in Subsection~\ref{subsec:ConstructionsForNavierStokes4D}. By the previous lemma, the solutions $\{ \theta_{\nu} \}_{\nu \geq 0}$ solving the 3D advection-diffusion equations
\begin{equation}
 \left\{
 \begin{array}{ll}
 \partial_t {\theta}_{\nu} + u_{\nu} \cdot \nabla {\theta}_{\nu} = \nu \Delta {\theta}_{\nu}; \\
 {\theta}_{\nu} (0,x,y,z) = {\theta}_{\initial}(x,y,z). \\
 \end{array}
 \right.
\end{equation}
exhibit anomalous dissipation i.e.
\begin{equation}
 \limsup_{\nu \to 0} \nu \int_0^1 \int_{\T^3} |\nabla \theta_{\nu}|^2 \, dx dy dz \, dt > 0.
\end{equation}
For each $\nu \geq 0$, we define $v_{\nu} \colon [0,1] \times \T^4 \to \R^4$ as
\begin{equation} \label{v_nu}
 v_{\nu} (t,x,y,z,w)  =  
 \begin{pmatrix}
  u_{\nu}(x,y,z) \\
  \theta_{\nu}(t,x,y,z)
 \end{pmatrix}
 =
 \begin{pmatrix}
  u_{\nu}^{(1)}(x,y,z) \\
  u_{\nu}^{(2)}(x,y,z) \\
  u_{\nu}^{(3)}(x,y,z) \\
  \theta_{\nu}(t,x,y,z)
 \end{pmatrix}.
\end{equation}
Standard computations yield
\begin{equation}
 (v_{\nu} \cdot \nabla)v_{\nu} = 
  \begin{pmatrix}
  \partial_z u_{\nu}^{(1)} \\
  \partial_z u_{\nu}^{(2)} \\
  0 \\
  u_{\nu} \cdot \nabla_{x,y,z} \theta_{\nu}
 \end{pmatrix}.
\end{equation}
Recalling the collection of forces $\{ F_{\nu} \}_{\nu \geq 0}$ defined in Subsection~\ref{subsec:ConstructionsForNavierStokes4D}, we see that with pressure $p_{\nu} = 0$, $v_{\nu}$ solves
\begin{equation}
 \left\{
 \begin{array}{ll}
 \partial_t v_{\nu} + (v_{\nu} \cdot \nabla)v_{\nu} + \nabla p_{\nu} - \nu \Delta v_{\nu} = F_{\nu} ; \\
 \diver v_{\nu} = 0; \\
 v_{\nu}(0, x, y, z, w) = v_{\initial, \nu}(x, y, z, w). \\
 \end{array}
 \right.
\end{equation}
Notice that
\begin{equation} \label{AD_final}
 \limsup_{\nu \to 0} \nu \int_0^1 \int_{\T^3} |\nabla v_{\nu}|^2 \, dx dy dz \, dt \geq \limsup_{\nu \to 0} \nu \int_0^1 \int_{\T^3} |\nabla \theta_{\nu}|^2 \, dx dy dz \, dt > 0
\end{equation}
where the last inequality follows from Lemma~\ref{lemma:LemmaAnomalousDissipationSpecialSequence}. This proves \eqref{e:zeroth-law}. It follows from Lemma~\ref{lemma:LemmaAnomalousDissipationSpecialSequence} and the definition of $p_{\nu}$ above that $(v_{\nu},p_{\nu}) \weak (v_0, p_0)$ in $L^{\infty}((0,1) \times \T^4)$ {with $p_0 = 0$} and that $e$ is smooth in $[0,1]$. 

The fact that $v_0$ defined in \eqref{v_nu} solves the forced Euler equation with body force $F_0$ and that
\begin{equation}
 \int_{\T^4} F_0(t,x) \cdot v_0(t,x) \, dx = \int_{\T^4} \partial_z v_0^{(1,2)}  \cdot v_0^{(1,2)}  =0 \quad \forall t \in (0,1).
\end{equation}
follows from standard computations, observing that $\partial_z v_0^{(1,2)}, v_0^{(1,2)}$ are time-independent continuous functions. The remaining desired properties of the collection of body forces $\{ F_{\nu} \}_{\nu \in (0,1)}$ follow from Lemma~\ref{lemma:AboutTheCollectionOfBodyForces}. 

We observe that by definition of $v_\nu$ we have 
$$\nu |\nabla v_\nu|^2 = \nu | \nabla u_\nu|^2 + \nu | \nabla \theta_\nu |^2$$
and thanks to the estimate \eqref{prop:estimate_u} and the definition of  $u_\nu$ \eqref{u_nu}, for any $\nu \in (\nu_{q+1} , \nu_q ]$ we have 
$$\nu |\nabla u_\nu|^2 \leq \nu \sup_{j \leq q -1 } \| \nabla u_\nu \|_{L^\infty (\T^2 \times \mathcal{I}_j)}^2 \leq  C \nu_q a_{q-1}^{2- 2 \gamma} a_q^{-2 (1+ \eps \delta)} \leq a_{q-1}^{2 - 3 \gamma} \,,  $$
where the last holds thanks to \eqref{d:k_q}. We observe that, thanks to $\gamma < 2/3$, the last term goes to zero as $q \to \infty$.

Therefore, by definition of the velocity fields $\{ u_{\nu} \}_{\nu \in (0,1)}$ and Lemma~\ref{lemma:LemmaAnomalousDissipationSpecialSequence}, up to not relabelled subsequences, we have 
\begin{equation}
 \nu |\nabla v_{\nu}| \weak \mu
\end{equation}
where $\mu$ is the measure in Lemma~\ref{lemma:LemmaAnomalousDissipationSpecialSequence}. From \eqref{eq:NS-Lemma-H-Minus-1-Norm}, we deduce \eqref{e:mu-duchon-robert}.

\section{Duchon--Robert distribution vs. anomalous dissipation measure} \label{sec:duchon-robert-anomalous}

The following theorem is an interesting observation that arises from  computations of \cite{BCCDLS22} with other choices of viscosity parameters $\nu_q$.

\begin{proposition}[Duchon--Robert measure $\neq$ anomalous dissipation measure] \label{prop:duchon-anomalous}
For any $\alpha <1$ there exists a divergence-free initial datum $v_{\initial}$ and force $F_0 \in C^\alpha ((0,2)\times \T^3)$ and a {\em weak physical solution} of $3D$ forced Euler equations $v_0 \in L^\infty ((0,2) \times \T^3)$ such that there exists a unique  anomalous dissipation measure $\mu \in \mathcal{M}((0,2) \times \T^3)$ associated to $v_0$ such that $\mu \equiv 0$ and 
$$\mu  \neq \mathcal{D}[v_0]  \,,$$
where $\mathcal{D}[v_0]$ is the Duchon--Robert distribution (see \eqref{duchon-robert}) associated to $v_0$.
\end{proposition}

\begin{proof}

Let $\beta=0$ with $ \alpha$, $\gamma$, $\epsilon$, $\delta$, $\{ a_q \}_{q \in \N}$ and $\{ T_q \}_{q \in \N}$ as in \cite[Section 6]{BCCDLS22}. Let $u \in C^{\infty}_{\loc}((0,1) \times \T^2; \R^2) {\cap L^{\infty}([0,2] \times \T^2; \R^2)}$ be the velocity fields constructed in \cite[Proposition 3.1]{BCCDLS22} with these parameters (restricted to the time interval $(0,1)$). Extend it as zero to $(0,2) \times \T^2$ whenever $t \geq 1$. 
We define the sequence of viscosity parameters 
$$\nu_q = a_q^{2 + 7 \epsilon (1+ \delta)}.$$
For any $\nu \in (0, a_0^{3})$, there exists a unique $q = q(\nu) \in \N$ such that $\nu \in (\nu_{q+1} , \nu_q]$. For any $\nu \in (0, a_0^{3})$ given this unique $q$, we set $K_\nu = [0,1-T_q]$ and define
\begin{equation}
     u_\nu (t,x ) = u(t,x) \mathbbm{1}_{K_\nu} (t) \quad t \in (0,2), \, x \in \T^2.
\end{equation}
We observe that $u_\nu$ is smooth for any $\nu \in (0, a_0^{3})$ (notice that {$u \in C^{\infty}_{\loc}((0,1) \times \T^2; \R^2)$ and} $u \equiv 0$ locally around $1-T_q$ for any $q$) and
$$\| u_\nu - u \|_{L^\infty((0,2) \times \T^2)} \to 0 $$
as $\nu \to 0$. 
The velocity fields $u$ satisfies the uniqueness of forward trajectories {since it belongs to $C^{\infty}_{\loc}((0,1) \times \T^2; \R^2) \cap L^{\infty}([0,2] \times \T^2; \R^2)$ and $u(t, \cdot) \equiv 0$ for all $t > 1$} by construction. By the superposition principle, the advection equation with velocity field $u$ has a unique {bounded} solution $\theta \colon (0,2) \times \T^2 \to \R$ with initial datum $\theta_{\initial} \in C^{\infty}(\T^2)$ where $\theta_{\initial}$ is defined as in \cite[Proposition 3.1]{BCCDLS22}.
{Hence, $\theta$ must be the weak limit of solutions of \eqref{e:ADV-DIFF} with $u$ and $\nu \to 0$.}
Thanks to \cite[Property (5) of Proposition 3.1]{BCCDLS22}, energy equality \eqref{e:energy-equality-global} {applied to the solutions of \eqref{e:ADV-DIFF} with $u$ and $\nu \to 0$} and uniqueness of the advection equation solutions in $L^\infty$, it is straightforward that $\theta$ is a dissipative solution in $(0,2)$, namely 
\begin{align} \label{eq:dissipation-theta-0}
\| \theta (t, \cdot ) \|_{L^2 (\T^2)} \leq \frac{1}{2} \| \theta_{\initial} \|_{L^2 (\T^2)}  \qquad \text{ for any  } t > 1.
\end{align} 
For $\nu \in (  \nu_{q+1},  \nu_q]$ we define $ \theta_{\nu} : [0,2] \times \T^3 \to \R$ to be the unique smooth solution to the advection-diffusion equation \eqref{e:ADV-DIFF} with diffusion parameter $\nu$, initial datum $\theta_{\initial} \in C^\infty (\T^2)$ and velocity field $ u_\nu$, 
i.e.
\begin{align*}
\partial_t \theta_{\nu} + {u}_{\nu} \cdot \nabla \theta_{\nu} = \nu \Delta \theta_{ \nu} \, .
\end{align*} 
Up to this point, all functions were defined on $(0,2) \times \T^2$. 
From now on, we consider all these functions as being defined on $(0,2) \times \T^3$ by extending them independently to the third spatial variable\footnote{i.e. for any function $f$ defined on $(0,2) \times \T^2$, we redefine it as $\tilde{f}$ on $(0,2) \times \T^3$ by $\tilde{f}(t, x_1, x_2, x_3) = f(t, x_1, x_2)$.}.  
Then, we define smooth functions $F_{\nu}, v_\nu : [0,2] \times \T^3 \to \R^3$ and $p_\nu : [0,2] \times \T^3 \to \R$ as 
\begin{align*}
F_{\nu}(t,x) 
    & = \begin{pmatrix}
\partial_t  u_\nu (t,x) - \nu \Delta  u_\nu (t,x)
\\
0
\end{pmatrix}
\\
v_{\nu} (t,x ) & = \begin{pmatrix}
   u_\nu(t,x)
 \\
 \theta_{\nu} (t,x)
\end{pmatrix}
\\p_{\nu} &= 0\,,
\end{align*}
which are independent on $x_3$. 
 Finally, we set
\begin{align*}
 v_{\initial } = \begin{pmatrix}
0 
\\
 \theta_{\initial}
\end{pmatrix}\, .
\end{align*}
Similarly to \cite{BCCDLS22} we can verify that $\| F_\nu - F_0 \|_{C^\alpha_{t,x}}  \to 0 $ where $F_0 = (\partial_t u , 0)$ as $\nu \to 0$ and $\| u_\nu - u \|_{C^\alpha_{t,x}} \to 0$ as $\nu \to 0$  and thanks to uniqueness of the advection equation $\theta_{\nu } $ is $L^\infty$ weakly* converging to $\theta$ as $\nu \to 0$, which implies that $v_\nu \rightharpoonup^* v_0$, where
$$v_{0} (t,x )  = \begin{pmatrix}
  u(t,x)
 \\
 \theta (t,x)
\end{pmatrix} $$
 is a solution to the 3D forced Euler equations by a direct computation. Since $\theta$ {satisfies} \eqref{eq:dissipation-theta-0} we have that $\mathcal{D}[v_0] \neq 0$. We now show that there exists a unique anomalous dissipation measure $\mu \in \mathcal{M} ((0,2) \times \T^3)$ and it is $\mu \equiv 0$.
Let $\nu \in (0, a_0^{3})$. We have that 
\begin{align*}
\nu \int_0^2 \int_{\T^3} |\nabla \theta_{\nu} (t,x)|^2 dx dt = \nu \int_0^{1 -T_q} \int_{\T^3} |\nabla \theta_{\nu} (t,x)|^2 dx dt + \nu \int_{1  - T_q}^2 \int_{\T^3} |\nabla \theta_{\nu} (t,x)|^2 dx dt \,.
\end{align*}
Thanks to \cite[Property (4) of Proposition 3.1]{BCCDLS22} we have that 
\begin{align} \label{eq:estimate-nu-nabla-theta-1}
\nu \int_0^{1-T_q} \int_{\T^3} \nabla {\theta_\nu} (t, x) \, dt dx   \leq \nu_q \sum_{k=0}^{q} a_k^{-2 - 6 \epsilon (1+ \delta)} \leq q \nu_q  a_q^{-2 - 6 \epsilon (1+ \delta)} \leq a_q^\epsilon  \to 0 \,,
\end{align}
where the last inequality holds for $q$ sufficiently large to reabsorb $q  a_q^{\epsilon \delta} \leq 1$.
{Using also that $u_\nu (t, \cdot ) \equiv 0$ for any $t \geq 1 - T_q$ we find that for any $t \geq 1 - T_q$, $\theta_{\nu}$ solves the heat equation. Hence $\nabla \theta_{\nu}$ solves the heat equation for any $t \geq 1 - T_q$. and by standard heat equation regularity we have}
\begin{align}\label{eq:estimate-nu-nabla-theta-2}
\nu \int_{1  - T_q}^2 \int_{\T^3} |\nabla \theta_{\nu} (t,x)|^2 dx dt  \leq 2 \nu \| \nabla \theta_{\nu} (1 - T_q , \cdot) \|_{L^\infty (\T^2)}^2 \leq   \nu_q  a_q^{-2 - 6 \epsilon (1+ \delta)} \leq a_q^\epsilon \,.
\end{align}

 Using the estimates of \cite[Property (3) of Proposition 3.1]{BCCDLS22} we also have 
 \begin{align} \label{eq:nu-u_q}
 \nu \| \nabla u_\nu \|_{L^2((0,2) \times \T^3)}^2  & \leq \nu_q \sum_{j =0}^{q-1} \| \nabla u_\nu \|_{L^\infty ((1-T_j, 1- T_{j+1}) \times \T^3)}^2 \notag
 \\
 & \leq C a_q^{2 + 7 \epsilon (1+ \delta)} \sum_{j=0}^{q-1} a_j^{2 - 2 \gamma} a_{j+1}^{-2 (1 +  \epsilon \delta)} 
 \\
 & \leq C q a_q^{2 + 7 \epsilon (1+ \delta)} a_q^{-2 (1+ \epsilon \delta)} \leq C q a_q^{7 \epsilon + 5 \epsilon \delta}
\leq  a_q^\epsilon \,, \notag
 \end{align}
 where we used that $\gamma <1$ and $q$ sufficiently large to reabsorb $C q a_q^{\epsilon} \leq 1$.
 Therefore
 we conclude that 
 $$ \mu = \lim_{\nu \to 0} \nu |  \nabla v_\nu |^2 = 0 \,,$$
thanks to \eqref{eq:estimate-nu-nabla-theta-1}, \eqref{eq:estimate-nu-nabla-theta-2} and \eqref{eq:nu-u_q}.
\end{proof}

As a corollary, the previous theorem yields a sequence of vanishing viscosity solutions $\{v_\nu \}_{\nu >0}$ of the $3D$ forced Navier--Stokes equations such that $v_\nu \rightharpoonup^* v_0$ where $v_0$ is a solution to the $3D$ forced Euler equations but they are distant in $L^p$.

\begin{corollary} \label{corollary:duchon-anomalous}
 Let  $\alpha \in (0,1)$. There exists  a weak physical solution $v_0 \in L^\infty$ of the $3D$ forced Euler equations with initial datum $v_{\initial } \in C^\infty$ and force $F_0 \in C^\alpha$ for which there exists $\varepsilon >0$ such that  the vanishing viscosity sequence $\{v_{\nu} \}_{\nu >0} \subset C^\infty$  of the $3D$ forced Navier--Stokes equations with forces $F_\nu $ satisfies $\sup_{\nu >0} \| v_\nu \|_{L^\infty} < \infty$, $F_\nu \to F_0 $ in $C^\alpha$  and
$$ \| v_\nu - v_0 \|_{L^p} \geq \varepsilon \qquad \text{ for any } p \in [1, \infty] \,.$$
\end{corollary} 

\begin{proof}
Take $v_0$ the weak physical solution of $3D$ forced Euler given in Proposition \ref{prop:duchon-anomalous}.
We argue by contradiction assuming that there exists $p \in [1, \infty]$ such that
$$ \liminf_{\nu \to 0}   \| v_\nu - v_0 \|_{L^p}  =0$$
and thanks to the $L^\infty$ bounds on $v_\nu$ and $v_0$ we can assume $p=3$.
 Therefore there exists a subsequence $\{ v_{\nu_q} \}_{\nu_q}$ such that $v_{\nu_q} \to v_0$ in $L^3$ and thanks to the fact that along this subsequence all the terms pass into the limit in \eqref{duchon-robert} we have that there exists an anomalous dissipation measure associated to $v_0$ $\mu \in \mathcal{M}((0,2) \times \T^3)$ which is equal to the Duchon-Robert distribution of $v_0 $  (defined in \eqref{duchon-robert}), namely $\mu = \lim_{\nu_q \to 0} \nu_q |\nabla v_{\nu_q}|^2 = \mathcal{D}[v_0]$ in the sense of distributions. However, Proposition \ref{prop:duchon-anomalous} showed that there exists a unique anomalous dissipation measure associated to $v_0$ $\mu$ which is such that $\mu \equiv 0 \neq \mathcal{D}[v_0]$ leading to a contradiction.
\end{proof}

\bibliographystyle{alpha}
 \bibliography{biblio}
 
\end{document}